%% file: main.tex
\documentclass{amsart}
\usepackage[latin1]{inputenc}
\usepackage[english]{babel}
\usepackage{csquotes}

\usepackage{color} 
\usepackage[colorlinks=true,linktoc=all,linkcolor=blue,citecolor=red,filecolor=pink,urlcolor=blue]{hyperref}

\usepackage[hyperref=true,style=alphabetic,citestyle=alphabetic,backend=bibtex,maxnames=100,doi=false,isbn=false,url=false,giveninits=true]{biblatex}

\usepackage{enumitem}
\usepackage{graphicx}

\usepackage[mathcal,mathscr]{euscript}

\usepackage{caption}
\usepackage{subcaption}
\usepackage{comment}
\usepackage{import}

\usepackage[mathcal,mathscr]{euscript}

\usepackage{amsmath, amssymb, amsfonts}
\usepackage{tikz} 
\usetikzlibrary{knots}
\usetikzlibrary{arrows}
\usepackage{pst-knot}
\usepackage{graphicx, overpic}
\usepackage{microtype}
\usepackage{MnSymbol,wasysym}

\usetikzlibrary{decorations.pathreplacing,decorations.markings}
\usetikzlibrary{patterns}
\usetikzlibrary{arrows}
\usepackage{extarrows}
\usetikzlibrary{decorations.markings}

\tikzset{middlearrow/.style={
		decoration={markings,
			mark= at position 0.55 with {\arrow{#1}} ,
		},
		postaction={decorate}
	}
}

\theoremstyle{plain}                    
\newtheorem{theorem}{Theorem}[section]
\newtheorem{lemma}[theorem]{Lemma}
\newtheorem{proposition}[theorem]{Proposition}
\newtheorem{corollary}[theorem]{Corollary}

\newtheorem{introtheorem}{Theorem}

\newcommand{\theoremnumber}{} 
\newtheorem*{maintheorem}{Theorem \theoremnumber}
\newenvironment{maintheoremc}[1]
  {\renewcommand{\theoremnumber}{#1}%
  \begin{maintheorem}}
  {\end{maintheorem}}
  
  \newtheorem{introcorollary}{Corollary}
\newcommand{\corollarynumber}{} 
\newtheorem*{maincorollary}{Corollary \corollarynumber}
\newenvironment{maincorollaryc}[1]
  {\renewcommand{\corollarynumber}{#1}%
  \begin{maincorollary}}
  {\end{maincorollary}}

\theoremstyle{definition}

\newtheorem{example}[theorem]{Example}
\newtheorem{remark}[theorem]{Remark}

\numberwithin{equation}{section}

\usepackage{todonotes}

\usepackage[normalem]{ulem}

\newcommand{\nn}{\mathbb N}
\newcommand{\zz}{\mathbb Z}
\newcommand{\qq}{\mathbb Q}
\newcommand{\rr}{\mathbb R}

\newcommand{\ff}{\mathbb F} 

\newcommand{\raag}[1]{A_{#1}} 
\newcommand{\bbg}[1]{BB_{#1}}
\newcommand{\vv}[1]{V(#1)} 
\newcommand{\ee}[1]{E(#1)} 
\newcommand{\lk}[1]{\operatorname{lk}\left( #1 \right)} 
\newcommand{\spoke}[1]{\operatorname{spoke}\left( #1 \right)} 
\newcommand{\relspoke}[2]{\operatorname{spoke}\left( #1 , #2 \right)} 
\newcommand{\st}[1]{\operatorname{st}\left( #1 \right)} 
\newcommand{\relstar}[2]{\operatorname{st}\left( #1 , #2 \right)} 

\newcommand{\flag}[1]{\Delta_{#1}} 
\newcommand{\bns}[1]{\Sigma^1(#1)} 
\newcommand{\bnsc}[1]{\bns{#1}^c} 
\newcommand{\chars}[1]{S(#1)} 
\newcommand{\living}[1]{\mathcal L(#1)} 
\newcommand{\dead}[1]{\mathcal D(#1)} 
\newcommand{\livingedge}[1]{\mathcal{LE}(#1)} 
\newcommand{\deadedge}[1]{\mathcal{DE}(#1)} 

\newcommand{\iep}[3]{\operatorname{IEP}(#1,#2,#3)} 

\makeatletter
\@namedef{subjclassname@2020}{\textup{2020} Mathematics Subject Classification}
\makeatother

\begin{document}

\title[A graphical description of the BNS-invariants of BBGs]{A graphical description of  \\  the BNS-invariants of Bestvina--Brady groups \\ and the RAAG recognition problem}

\author{Yu-Chan Chang}
\address{Department of Mathematics and Computer Science - Wesleyan University, 265 Church Street, Middletown, CT 06459, USA}
\email{yuchanchang74321@gmail.com}

\author{Lorenzo Ruffoni}
\address{Department of Mathematics - Tufts University, 177 College Avenue, Medford, MA 02155, USA}
\email{lorenzo.ruffoni2@gmail.com}


\begin{abstract}
A finitely presented Bestvina--Brady group (BBG) admits a presentation involving only commutators.
We show that if a graph admits a certain type of spanning tree, then the associated BBG is a right-angled Artin group (RAAG). 
As an application, we obtain that the class of BBGs contains the class of RAAGs.
On the other hand, we provide a criterion to certify that certain finitely presented BBGs are not isomorphic to RAAGs (or more general Artin groups). 
This is based on a description of the Bieri--Neumann--Strebel invariants of finitely presented BBGs in terms of separating subgraphs, analogous to the case of RAAGs.
As an application, we characterize when the BBG associated to a 2-dimensional flag complex is a RAAG in terms of certain subgraphs. 
\end{abstract}

\maketitle

\tableofcontents



\section{Introduction}
\addtocontents{toc}{\protect\setcounter{tocdepth}{1}}

Let $\Gamma$ be a finite simplicial graph and denote its vertex set and edge set by $\vv \Gamma$ and $\ee \Gamma$, respectively. The associated \textit{right-angled Artin group} (RAAG) $\raag \Gamma$ is the group defined by the following finite presentation
$$
\raag \Gamma=\big\langle \vv \Gamma \ \big\vert \ [v,w] \ \text{whenever} \ (v,w)\in \ee \Gamma\big\rangle.
$$
\noindent RAAGs have been a central object of study in geometric group theory because of the beautiful interplay between algebraic properties of the groups and combinatorial properties of the defining graphs, and also because they contain many interesting subgroups, such as the fundamental group of many surfaces and $3$-manifolds, and more generally, specially cubulated groups; see \cite{HW08}.

The \textit{RAAG recognition problem} consists in deciding whether a given group is a RAAG.
Several authors have worked on this problem for various classes of groups,
for instance, the pure symmetric automorphism groups of RAAGs in \cite{CharneyRuaneStambaughVijayanTheAutomorphismgroupofagraphproductiwithnoSIL} and \cite{KobanPiggottTheBNSofthepuresymmetricautomorphismofRAAG}, the pure symmetric outer automorphism groups of RAAGs in  \cite{DayWadeSubspaceArrangementBNSinvariantsandpuresymmetricOuterAutomorphismsofRAAGs}, and a certain class of subgroups of  RAAGs and RACGs in \cite{DaniLevcovitzRightangledArtinsubgroupsofRAACsandRAAGs} and of mapping class groups in \cite{KoberdaRAAGsandaGeneralizedIsoProblemforFGsubgpsofMCGs}.
An analogous recognition problem for right-angled Coxeter groups has been considered in \cite{CEPR16}

However, the RAAG recognition problem is not easy to answer in general, even when the given group shares some essential properties with RAAGs. 
For example, the group $G$ with the following presentation
$$
G=\big\langle a,b,c,d,e \ \big\vert \ [a,b], [b,c], [c,d], [b^{-1}c,e] \big\rangle
$$
is finitely presented with only commutator relators; it is CAT$(0)$ and splits as a graph of free abelian groups. However, it is not a RAAG; see \cite[Example 2.8]{PapadimaSuciuAlgebraicinvariantsforBBGs}.
Even more is true: Bridson \cite{BridsonOntheRecognitionofRAAGs} showed that there is no algorithm to determine whether or not a    group presented by commutators is a RAAG, answering a question by Day and Wade \cite[Question 1.2]{DayWadeSubspaceArrangementBNSinvariantsandpuresymmetricOuterAutomorphismsofRAAGs}.

In this article, we study the RAAG recognition problem for a class of normal subgroups of RAAGs, namely, the \textit{Bestvina--Brady groups} (BBGs).
Let $\chi \colon \raag \Gamma \to \zz$ be the homomorphism sending all the generators to $1$. The BBG defined on $\Gamma$ is the kernel of $\chi$ and is denoted by $\bbg\Gamma$.
For example, the group $G$ from above is the BBG defined on the \textit{trefoil graph} (see Figure~\ref{fig:trefoil}).
BBGs were introduced and studied in \cite{bestvinabradymorsetheoryandfinitenesspropertiesofgroups}, 
and they have become popular as a source of pathological examples in the study of finiteness properties and cohomology of groups. 
For instance, some BBGs are finitely generated but not finitely presented; and there are some BBGs that are potential counterexamples to either the Eilenberg--Ganea conjecture or the Whitehead conjecture.

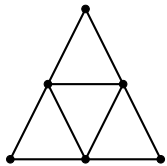
\begin{figure}[h]
    \centering
    \input{pictures/trefoil}
    \caption{The trefoil graph}
    \label{fig:trefoil}
\end{figure}

\clearpage
Inspired by the example of the group $G$ from above, we are interested in understanding how much a BBG can be similar to a RAAG without being a RAAG.
In particular, we are interested in a criterion that can be checked directly on the defining graph.
It is well-known that two RAAGs are isomorphic if and only if their defining graphs are isomorphic; see \cite{DromsIsomorphismsofGraphGroups}.
However, this is not the case for BBGs. 
For instance, the BBG defined on a tree with $n$ vertices is always the free group of rank $n-1$.
Nevertheless, some  features of BBGs can still be seen directly from the defining graphs. 
For example, it was proved in \cite{bestvinabradymorsetheoryandfinitenesspropertiesofgroups} that $\bbg \Gamma$ is finitely generated if and only if $\Gamma$ is connected; and $\bbg \Gamma$ is finitely presented if and only if the  flag complex $\flag \Gamma$ associated to $\Gamma$ is simply connected. 
When a BBG is finitely generated, an explicit presentation was found by Dicks and Leary \cite{DicksLearypresentationsforsubgroupsofArtingroups}.
More properties that have been discussed from a graphical perspective include various cohomological invariants in \cite{PapadimaSuciuAlgebraicinvariantsforBBGs,DimacaPapadimaSuciuQuasiKahlerBBGs,PapadimaandSuciuBNSRinvariantsandHomologyJumpingLoci,LearySaadetogluTheCohomologyofBBGs}, Dehn functions in   \cite{YCCIdentifyingDehnFunctionsofBBGfromtheirdefininggraphs}, and graph of groups decompositions in \cite{ChangJSJofBBGs,lorenzo,DR22}.

In this paper, we add to this list a solution to the RAAG recognition problem for BBGs whose associated flag complexes are $2$-dimensional (equivalently, the largest complete subgraphs of the defining graphs are triangles).
Unless otherwise stated, we will always assume $\Gamma$ is connected.
We note that it is natural to make two additional assumptions. 
The first one is that $\Gamma$ is biconnected, that is, it has no cut vertices (otherwise, one can split $\bbg \Gamma$ as the free product of the BBGs on the biconnected components of $\Gamma$; see Corollary~\ref{cor:biconnected components}).
The second assumption is that the associated flag complex $\flag \Gamma$ is simply connected (otherwise, the group $\bbg \Gamma$ is not even finitely presented).
Our solution to the RAAG recognition problem in dimension 2 is in terms of the presence or absence of two particular types of subgraphs.
A \textit{tree 2-spanner} $T$ of the graph $\Gamma$ is a spanning tree such that for any two vertices $x$ and $y$, we have $d_T(x,y)\leq 2 d_\Gamma (x,y)$. 
A \textit{crowned triangle} of the associated flag complex $\flag \Gamma$ is a triangle whose edges are not on the boundary of $\flag \Gamma$ (see  \S\ref{section: BBGs on 2-dim flag complexes} for the precise definition). For instance, the central triangle of the trefoil graph in Figure~\ref{fig:trefoil} is a crowned triangle. 

\begin{introtheorem}\label{main thm 2dim}
Let $\Gamma$ be a biconnected graph such that $\flag \Gamma$ is $2$-dimensional and simply connected. Then the following statements are equivalent. 
\begin{enumerate}

    \item   $\Gamma$ admits a tree $2$-spanner.
   
    \item   $\flag \Gamma$ does not contain crowned triangles.
    
    \item   $\bbg \Gamma$ is a RAAG.
    
    \item   $\bbg \Gamma$ is an Artin group.
\end{enumerate}
\end{introtheorem}

Our proof of Theorem~\ref{main thm 2dim} relies on two  conditions that are independent, in the sense that they work separately and regardless of the dimension of $\flag \Gamma$.
The first one is a sufficient condition for a BBG to be a RAAG that is based on the existence of a tree 2-spanner (see \S\ref{intro: condition RAAG}).
The second one is a sufficient condition for any finitely generated group not to be a RAAG that is based on certain properties of the \textit{Bieri--Neumann--Strebel invariant} (BNS-invariant) and may be of independent interest (see \S\ref{intro: condition not RAAG}).
We prove that these two conditions are equivalent when the flag complex $\flag \Gamma$ is 2-dimensional (see \S\ref{intro: 2-dim}).

This allows one to recover the fact that the group $G$ from above (that is, the BBG defined on the trefoil graph from Figure~\ref{fig:trefoil}) is not a RAAG. This was already known by the results of \cite{PapadimaSuciuAlgebraicinvariantsforBBGs} or \cite{DayWadeSubspaceArrangementBNSinvariantsandpuresymmetricOuterAutomorphismsofRAAGs}.
While the results in these two papers apply to groups that are more general than the group $G$, they do not address the case of a very minor modification of that example, such as the BBG defined on the graph in Figure~\ref{fig: extended PS no orientation}.
This BBG shares all the properties with the group $G$ described above. But again, it is not a RAAG by Theorem~\ref{main thm 2dim} since the defining graph contains a crowned triangle; see Example~\ref{ex:extended trefoil continued}.

The features of the BNS-invariant that we use to show that a BBG is not a RAAG turn out to imply that the BBG cannot even be a more general Artin group.
This relies on the theory of \textit{resonance varieties} developed by Papadima and Suciu in \cite{PapadimaSuciuAlgebraicinvariantsforRAAGs,PapadimaSuciuAlgebraicinvariantsforBBGs}. 
Roughly speaking, we show that for the BBGs under consideration in this paper, the resonance varieties are the same as the complements of the BNS-invariants (see \S \ref{sec:resonance varieties}).

\begin{figure}[h]
    \centering
    \input{pictures/extended_PS_unoriented}
    \caption{The extended trefoil graph. The BBG defined by it has this presentation: $\big\langle a,b,c,d,e,f \ \big\vert \ [a,b], [b,c], [c,d], [b^{-1}c,e], [e,f] \big\rangle$}
    \label{fig: extended PS no orientation}
\end{figure}
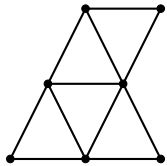


\subsection{The condition to be a RAAG: tree 2-spanners}\label{intro: condition RAAG}
As observed in \cite[Corollary 2.3]{PapadimaSuciuAlgebraicinvariantsforBBGs}, when $\bbg \Gamma$ is finitely presented, any spanning tree $T$ of $\Gamma$ provides a finite presentation whose relators are commutators. 
If $T$ is a tree $2$-spanner, then this presentation can actually be simplified to a standard RAAG presentation, and in particular, the group $\bbg \Gamma$ is a RAAG.
We can even identify the defining graph for this RAAG in terms of the \textit{dual graph} $T^\ast$ of $T$, that is, the graph whose vertices are edges of $T$, and two vertices are adjacent if and only if the corresponding edges of $T$ are contained in the same triangle of $\Gamma$.
Note that the following result does not have any assumption on the dimension of $\flag\Gamma$.

\begin{introtheorem}\label{main thm BBG=RAAG}
If $\Gamma$ admits a tree $2$-spanner $T$, then $\bbg \Gamma$ is a RAAG. More precisely, the Dicks--Leary presentation can be simplified to the standard RAAG presentation with generating set $\ee T$. Moreover, we have $\bbg \Gamma \cong \raag{T^\ast}$.
\end{introtheorem}

Here are two applications. The first one is that if $\Gamma$ admits a tree $2$-spanner, then $\flag \Gamma$ is contractible; see Corollary~\ref{cor:tree2spanner implies contractible}. 
The second application is that for any graph $\Lambda$, the BBG defined on the cone over $\Lambda$ is isomorphic to $\raag \Lambda$, regardless of the structure of $\Lambda$; see Corollary~\ref{cor: cone graph gives an isomorphism between BBG and RAAG}.
This means that the class of BBGs contains the class of RAAGs. That is, every RAAG arises as the BBG defined on some graph.

As we have mentioned before, two RAAGs are isomorphic if and only if their defining graphs are isomorphic, but this is not true for BBGs.
However, when two graphs admit tree $2$-spanners, the associated BBGs and the dual graphs completely determine each other.

\begin{introcorollary}\label{intro: BBGs iso iff dual graphs iso}
Let $\Gamma$ and $\Lambda$ be two graphs admitting tree 2-spanners $T_\Gamma$ and $T_\Lambda$, respectively.
Then $\bbg \Gamma \cong \bbg \Lambda$ if and only if $T_\Gamma^\ast \cong T_\Lambda^\ast$.
\end{introcorollary}

This provides new examples of non-isomorphic graphs defining isomorphic BBGs; see Example~\ref{ex:iso bbg non iso graphs}. 
On the other hand, when $\Gamma$ does not admit a tree $2$-spanner, the presentation for $\bbg \Gamma$ associated to any spanning tree is never a RAAG presentation. 
However, there might be a RAAG presentation not induced by a spanning tree.
In order to obstruct this possibility, we need to look for invariants that do not depend on the choice of a generating set.
We will consider the BNS-invariant $\bns {\bbg \Gamma}$ of $\bbg \Gamma$.


\subsection{The BNS-invariants of BBGs from the defining graphs}\label{intro: BNS invariant}
The BNS-invariant $\bns G$ of a finitely generated group $G$ is a certain open subset of the character sphere $\chars G$, that is, the unit sphere in the space of group homomorphisms $\operatorname{Hom}(G,\rr)$.
This invariant was introduced in \cite{bierineumannstrebelageometricinvariantofdiscretegroups} as a tool to study finiteness properties of normal subgroups of $G$ with abelian quotients, such as kernels of characters. In general, the BNS-invariants are hard to compute.

The BNS-invariants of RAAGs have been characterized in terms of the defining graphs by Meier and VanWyk in \cite{meierthebierineumannstrebelinvariantsforgraphgroups}. 
The BNS-invariants of BBGs are less understood. 
In \cite[Theorem 15.8]{PapadimaandSuciuBNSRinvariantsandHomologyJumpingLoci}, Papadima and Suciu gave a cohomological upper bound for the BNS-invariants of BBGs. 
Recently, Kochloukova and  Mendon\c{c}a have shown in \cite[Corollary 1.3]{kochloukovamendonontheBNSRsigmainvariantsoftheBBGs} how to reconstruct the BNS-invariant of a BBG from that of the ambient RAAG.
However, an explicit description of the BNS-invariant of a BBG from its defining graph is still in need (recall that the correspondence between BBGs and graphs is not as explicit as in the case of RAAGs). 

Since the vertices of $\Gamma$ are generators for $\raag \Gamma$, a convenient way to describe characters of $\raag \Gamma$ is via vertex-labellings.
Inspired by  this, in the present paper, we encode characters of $\bbg \Gamma$ as edge-labellings.
This relies on the fact that the edges of $\Gamma$ form a generating set for $\bbg \Gamma$, under our standing assumption that $\Gamma$ is connected (see \cite{DicksLearypresentationsforsubgroupsofArtingroups} and \S\ref{sec:coordinates}).
We obtain the following graphical criterion for a character of a BBG to belong to the BNS-invariant.
The condition appearing in the following statement involves the \textit{dead edge subgraph}  $\deadedge \chi$ of a character $\chi$ of $\bbg \Gamma$, which is the graph consisting of edges on which $\chi$ vanishes.
This is reminiscent of the living subgraph criterion for RAAGs in \cite{meierthebierineumannstrebelinvariantsforgraphgroups}.
However, it turns out that the case of BBGs is better understood in terms of the dead edge subgraph (see  Example~\ref{ex:no good living edge subgraph criterion}). An analogous dead subgraph criterion for RAAGs was considered in \cite{lorenzo}.

\begin{introtheorem}[Graphical criterion for the BNS-invariant of a BBG]\label{intro: graphical criterion}
Let $\Gamma$ be a biconnected graph with $\flag \Gamma$ simply connected. 
Let $\chi\in\mathrm{Hom}(\bbg \Gamma,\rr)$ be a non-zero character. Then $[\chi]\in \bns{\bbg \Gamma}$ if and only if $\deadedge \chi$ does not contain a full subgraph that separates $\Gamma$.
\end{introtheorem}

Theorem~\ref{intro: graphical criterion} allows one to work explicitly in terms of graphs with labelled edges.
In particular, we show in Corollary~\ref{cor:equivalent biconnected} that the following are equivalent: the graph $\Gamma$ is biconnected, 
the BNS-invariant $\bns{\bbg \Gamma}$ is nonempty, and $\bbg \Gamma$ \textit{algebraically fibers} (that is, it admits a homomorphism to $\zz$ with finitely generated kernel).

In the same spirit, we obtain the following graphical description for (the complement of) the BNS-invariants of BBGs.
Here, a \textit{missing subsphere} is a subsphere of the character sphere that is in the complement of the BNS-invariant (see \S\ref{sec:BNS stuff} for details).

\begin{introtheorem}[Graphical description of the BNS-invariant of a BBG]\label{intro:graphical description}
Let $\Gamma$ be a biconnected graph with $\flag \Gamma$ simply connected.
Then $\bnsc{\bbg \Gamma}$ is a union of missing subspheres corresponding to  full separating subgraphs. More precisely,
\begin{enumerate}
    
    \item $\bnsc{\bbg \Gamma}= \bigcup_\Lambda S_\Lambda$, where $\Lambda$ ranges over the minimal  full separating subgraphs of $\Gamma$.

    \item There is a bijection between maximal missing subspheres in  $\bnsc{\bbg \Gamma}$ and minimal full separating subgraphs of $\Gamma$.
    
\end{enumerate}
\end{introtheorem}

In particular, as observed in \cite[Corollary 1.4]{kochloukovamendonontheBNSRsigmainvariantsoftheBBGs}, the set $\bnsc{\bbg \Gamma}$ carries a natural structure of a rationally defined spherical polyhedron. A set of defining equations can be computed directly by looking at the minimal full separating subgraphs of $\Gamma$.
This is analogous to the case of RAAGs; see \cite{meierthebierineumannstrebelinvariantsforgraphgroups}.

As a corollary of our description, we can identify the complement of the BNS-invariant with the first resonance variety (see Proposition~\ref{prop:bns resonance}). 
This improves the inclusion from \cite[Theorem 15.8]{PapadimaandSuciuBNSRinvariantsandHomologyJumpingLoci} to an equality.
Once again, this is analogous to the case of RAAGs; see \cite[Theorem 5.5]{PapadimaSuciuAlgebraicinvariantsforRAAGs}. 
It should be noted that there are groups for which the inclusion is strict; see \cite{SU21}.


\subsection{The condition not to be a RAAG: redundant triangles}\label{intro: condition not RAAG}
The BNS-invariant of a RAAG or BBG is the complement of  a certain arrangement of subspheres of the character sphere. 
(Equivalently, one could consider the arrangement of linear subspaces given by the linear span of these subspheres.)
The structural properties of this arrangement do not depend on any particular presentation of the group, so this arrangement turns out to be a useful invariant.
In \S\ref{sec:IEP}, inspired by the work of  \cite{KobanPiggottTheBNSofthepuresymmetricautomorphismofRAAG,DayWadeSubspaceArrangementBNSinvariantsandpuresymmetricOuterAutomorphismsofRAAGs}, we consider the question of whether the maximal members in this arrangement are ``in general position'', that is, whether they satisfy the inclusion-exclusion principle.

In \cite{KobanPiggottTheBNSofthepuresymmetricautomorphismofRAAG}, Koban and Piggott proved that the maximal members in the arrangement for a RAAG satisfy the inclusion-exclusion principle.
Day and Wade in \cite{DayWadeSubspaceArrangementBNSinvariantsandpuresymmetricOuterAutomorphismsofRAAGs} developed a homology theory to detect when an arrangement does not satisfy the inclusion-exclusion principle.
These results can be used together with our description of the BNS-invariants of BBGs to see that many BBGs are not RAAGs.
However, some BBGs elude Day--Wade's homology theory, such as the BBG defined on the graph in Figure~\ref{fig: extended PS no orientation}.
This motivated us to find an additional criterion to certify that a group $G$ is not a RAAG.
A more general result in Proposition~\ref{prop:criterion non RAAG} roughly says  that if there are three maximal subspheres of $\bnsc{G}$ that are not ``in general position'', then $G$ is not a RAAG.

We are able to apply Proposition~\ref{prop:criterion non RAAG} to a wide class of BBGs.
This is based on the notion of a \textit{redundant triangle}. 
Loosely speaking, a redundant triangle is a triangle in $ \Gamma$ such that the links of its vertices are separating subgraphs of $\Gamma$ that do not overlap too much (see \S\ref{sec:redundant triples BBGs} for the precise definition).
The presence of such a triangle provides a triple of missing subspheres (in the sense of our graphical description; see Theorem~\ref{intro:graphical description}) that does not satisfy the inclusion-exclusion principle.

\begin{introtheorem}\label{main thm BBGnotRAAG}
 Let $\Gamma$ be a biconnected graph such that $\flag \Gamma$ is simply connected.
 If $\Gamma$ has a redundant triangle, then $\bbg \Gamma$ is not a RAAG.
 \end{introtheorem}

We emphasize that Theorem~\hyperref[thm:redundant triple criterion]{E} works without any assumptions on the dimension of $\flag \Gamma$.
On the other hand, the obstruction is $2$-dimensional, in the sense that it involves a triangle, regardless of the dimension of $\flag \Gamma$; see Example~\ref{ex:higher dimensional}.


\subsection{The 2-dimensional case: proof of Theorem~\ref{main thm 2dim}}\label{intro: 2-dim}
The two conditions described in \S\ref{intro: condition RAAG} and \S\ref{intro: condition not RAAG} are complementary when $\Gamma$ is biconnected and $\flag\Gamma$ is $2$-dimensional and simply connected.
This follows from some structural properties enjoyed by $2$-dimensional flag complexes.

In Proposition~\ref{prop: tree 2-spanner iff no crowned triangles}, we establish that $\Gamma$ admits a tree $2$-spanner if and only if $\flag \Gamma$ does not contain crowned triangles. 
The ``if'' direction relies on a decomposition of $\flag \Gamma$ into certain elementary pieces, namely, the cones over certain 1-dimensional flag complexes. 
It then follows from Theorem~\ref{main thm BBG=RAAG} that $\bbg \Gamma$ is a RAAG.

On the other hand, we show in Lemma~\ref{lem:crowned tri is redundant in dim 2} that every crowned triangle is redundant in dimension two.
It then follows from Theorem~\ref{main thm BBGnotRAAG} that $\bbg \Gamma$ is not a RAAG.
The theory of resonance varieties (see \S\ref{sec:resonance varieties}) allows us to conclude that $\bbg \Gamma$ cannot be a more general Artin group either.

Figure~\ref{fig:implications} illustrates the various implications.
The only implication we do not prove directly is that if $\bbg \Gamma$ is a RAAG, then $\Gamma$ has a tree $2$-spanner. 
This implication follows from the other ones, and in particular, it means that one can write down the RAAG presentation for $\bbg \Gamma$ associated to the tree $2$-spanner. 
This fact is a priori not obvious but quite satisfying.

For the sake of completeness, we note that Theorem~\ref{main thm 2dim} fails for higher-dimensional flag complexes; see Remark~\ref{rem: higher dimensional}.

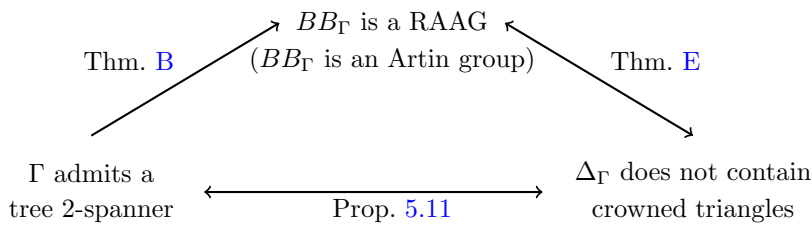
\begin{figure}[h]
    \centering
\begin{tikzpicture}[scale=1]
\node at (0,0) {$\bbg \Gamma$ is a RAAG};
\node at  (0,-.5) {($\bbg \Gamma$ is an Artin group)};
\node at (-4,-2) {$\Gamma$ admits a};
\node at (-4,-2.5) {tree $2$-spanner};
\node at (4,-2) {$\flag \Gamma$ does not contain}; 
\node at (4,-2.5) {crowned triangles};

\draw [thick, <->] (-2.5,-2.25)--(2,-2.25);
\node at (0,-2.5) {Prop.~\ref{prop: tree 2-spanner iff no crowned triangles}};

\draw [thick, ->] (-4,-1.5)--(-1.5,0);
\node at (-3.5,-.5) {Thm.~\ref{main thm BBG=RAAG}};

\draw [thick, ->] (1.5,0)--(4,-1.5);
\node at (3.5,-.5) {Thm.~\ref{main thm BBGnotRAAG}};
\end{tikzpicture}    
\caption{The implications in the proof of Theorem~\ref{main thm 2dim}}
    \label{fig:implications}
\end{figure}

\subsection{Structure of the paper.}
The rest of the paper is organized as follows. 
In \S\ref{section: preliminaries}, we fix some terminology and give some background on BBGs. 
In \S\ref{section: BBGs that are RAAGs}, we study tree $2$-spanners and use them to provide a sufficient condition for a BBG to be a RAAG (Theorem~\ref{main thm BBG=RAAG}). 
We also give many examples.
In \S\ref{section: BBGs that are not RAAGs}, we present a graphical criterion (Theorem~\ref{intro: graphical criterion}) and a graphical description (Theorem~\ref{intro:graphical description}) for the BNS-invariants of BBGs. 
We use this to provide a sufficient condition for a BBG not to be a RAAG (Theorem~\ref{main thm BBGnotRAAG}).
This is based on a study of the inclusion-exclusion principle for the arrangements that define the complement of the BNS invariants.
We discuss the relation with resonance varieties in \S \ref{sec:resonance varieties}.
In \S\ref{section: BBGs on 2-dim flag complexes}, we provide a solution to the RAAG recognition problem for BBGs defined on $2$-dimensional flag complexes (Theorem~\ref{main thm 2dim}).
In the end, we include some observations about the higher dimensional case.

\bigskip

\noindent \textbf{Acknowledgements}
We thank Tullia Dymarz for bringing \cite{kochloukovamendonontheBNSRsigmainvariantsoftheBBGs} to our attention and Daniel C. Cohen,  Pallavi Dani, Max Forester, Wolfgang Heil, Alexandru Suciu, and Matthew Zaremsky for helpful conversations.
We thank the School of Mathematics at the Georgia Institute of Technology for their hospitality during a visit in which part of this work was done.
We thank the referee for the very careful reading and helpful comments.
The second author acknowledges support from the AMS and the Simons Foundation.


\section{Preliminaries}\label{section: preliminaries}
\addtocontents{toc}{\protect\setcounter{tocdepth}{2}}

\subsection{Notation and terminology}\label{subsection: notation and terminology}
In this paper, unless otherwise stated, a \textit{graph} $\Gamma$ is a finite $1$-dimensional simplicial complex, not necessarily connected. 
We denote by $\vv \Gamma$ the set of its \textit{vertices} and by $\ee \Gamma$  the set of its \textit{edges}.
We do not fix any orientation on $\Gamma$, but we often need to work with \textit{oriented edges}. If $e$ is an oriented edge, then we denote its initial vertex and terminal vertex by $\iota e$ and $\tau e$, respectively; we denote by $\bar e$ the same edge with opposite orientation.
We always identify edges of $\Gamma$ with the unit interval and equip $\Gamma$ with the induced length metric.
A \textit{subgraph} of $\Gamma$ is a simplicial subcomplex, possibly not connected, possibly not full.

A \textit{path}, a \textit{cycle}, and a \textit{complete graph} on $n$ vertices are denoted by $P_n$, $C_n$, and $K_n$, respectively.
(Note that by definition, there is no repetition of edges in a path or cycle.)
A \textit{clique} of $\Gamma$ is a complete subgraph.
A \textit{tree} is a simply connected graph.
A \textit{spanning tree} of a graph $\Gamma$ is a subgraph $T\subseteq \Gamma$ such that $T$ is a tree and $\vv T = \vv \Gamma$.

The \emph{link} of a vertex $v \in \vv \Gamma$, denoted by $\lk{v,\Gamma}$, is the full subgraph induced by the vertices that are adjacent to $v$. 
The \emph{star} of $v$ in $\Gamma$, denoted by $\st{v,\Gamma},$ is the full subgraph on $\lk{v,\Gamma}\cup\lbrace v\rbrace$.
More generally, let $\Lambda$ be  a subgraph  of $\Gamma$.
The \textit{link} of $\Lambda$ is the full subgraph $\lk{\Lambda,\Gamma}$ induced by vertices at distance $1$ from $\Lambda$.
The \emph{star} of $\Lambda$ in $\Gamma$, denoted by $\st{\Lambda,\Gamma},$ is the full subgraph on $\lk{\Lambda,\Gamma}\cup \vv \Lambda$.

The \emph{join} of two graphs $\Gamma_{1}$ and $\Gamma_{2}$, denoted by $\Gamma_{1}\ast\Gamma_{2}$, is the full graph on $V(\Gamma_{1})\cup V(\Gamma_{2})$ together with an edge joining each vertex in $V(\Gamma_{1})$ to each vertex in $V(\Gamma_{2})$. 
A vertex in a graph that is adjacent to every other vertex is called a \emph{cone vertex}. A graph that has a cone vertex is called a \emph{cone graph}. In other words, a cone graph $\Gamma$ can be written as a join $\lbrace v\rbrace\ast\Gamma'$. In this case, we also say that $\Gamma$ is a \textit{cone over} $\Gamma'$. 
The \textit{complement of $\Lambda$ in $\Gamma$} is the full subgraph  $\Gamma\setminus \Lambda$ spanned by $\vv \Gamma \setminus \vv \Lambda$.
We say that $\Lambda$ is \textit{separating} if $\Gamma\setminus\Lambda$  is disconnected.
A \textit{cut vertex} of $\Gamma$ is a vertex that is separating as a subgraph.
A \textit{cut edge} of $\Gamma$ is an edge that is separating as a subgraph.
A graph is \textit{biconnected} if it has no cut vertices. If a graph is not biconnected, its \textit{biconnected components} are the maximal biconnected subgraphs.

Given a graph $\Gamma$, the \textit{flag complex} $\flag \Gamma$ on $\Gamma$ is the simplicial complex obtained by gluing a $k$-simplex to $\Gamma$ for every collection of $k+1$ pairwise adjacent vertices of $\Gamma$ (for $k\geq 2$).
The \textit{dimension} of $\flag \Gamma$ is denoted by $\dim \flag \Gamma$ and defined to be the maximal dimension of a simplex in $\flag \Gamma$.
(If $\flag \Gamma$ $1$-dimensional, then it coincides with $\Gamma$, and the following terminology agrees with the one introduced before.)
If $Z$ is a subcomplex of $\flag \Gamma$, the \textit{link} of $Z$  in $\flag \Gamma$, denoted by $\lk{Z,\flag \Gamma}$, is defined as the full subcomplex of $\flag \Gamma$ induced by the vertices at a distance one from $Z$.
Similarly, the \textit{star}  of $Z$ in $\flag \Gamma$, denoted by $\st{Z,\flag \Gamma}$, is defined as the full subcomplex induced by $\lk{Z,\flag \Gamma} \cup Z$.


\subsection{The Dicks--Leary presentation}\label{sec:DL presentation}
Let $\Gamma$ be a graph, and let $\raag \Gamma$ be the associated RAAG.
Let $\chi \colon \raag \Gamma \to \zz$ be the homomorphism sending all the generators to $1$. The \textit{Bestvina--Brady group} (BBG) on $\Gamma$, denoted by $\bbg\Gamma$, is defined to be the kernel of $\chi$.
When $\Gamma$ is connected, the group $\bbg \Gamma$ is finitely generated (see \cite{bestvinabradymorsetheoryandfinitenesspropertiesofgroups}) and has the following (infinite) presentation, called the \emph{Dicks--Leary presentation}.

\begin{theorem}\textnormal{(\cite[Theorem 1]{DicksLearypresentationsforsubgroupsofArtingroups})}\label{thm:DL presentation embedding}
Let $\Gamma$ be a   graph. If $\Gamma$ is connected, then $\bbg \Gamma$ is generated by the set of oriented edges of $\Gamma$, and the relators are words of the form $e^{n}_{1}\dots e^{n}_{l}$ for each oriented cycle $(e_{1},\dots,e_{l})$, where $n,l\in\zz$, $n\neq 0$, and $l\geq2$. 
Moreover, the group $\bbg \Gamma$ embeds in $\raag \Gamma$ via $e \mapsto \tau e(\iota e)^{-1}$ for each oriented edge $e$. 
\end{theorem}

For some interesting classes of graphs, the Dicks--Leary presentation can be considerably simplified.
For instance, when the flag complex $\flag \Gamma$ on $\Gamma$ is simply connected, the group $\bbg \Gamma$ admits the following finite presentation.

\begin{corollary}\textnormal{(\cite[Corollary 3]{DicksLearypresentationsforsubgroupsofArtingroups})}\label{cor:DL presentation}
When the flag complex $\flag \Gamma$ on $\Gamma$ is simply connected, the group $\bbg \Gamma$ admits the following finite presentation: the generating set is the set of the oriented edges of $\Gamma$, and the relators are $e\bar{e}=1$ for every oriented edge $e$, and $e_{i}e_{j}e_{k}=1$ and $e_{k}e_{j}e_{i}=1$ whenever $(e_{i},e_{j},e_{k})$ form  an oriented triangle; see Figure \ref{Oriented triangle}. 
\end{corollary}

\begin{figure}[ht]
\begin{tikzpicture}[scale=0.5]
\draw [thick, middlearrow={stealth}] (4,3)--(0,0);
\draw [thick, middlearrow={stealth}] (0,0)--(6,0);
\draw [thick, middlearrow={stealth}] (6,0)--(4,3);

\node [left] at (2,2) {$e_{i}$};
\node at (3,-1) {$e_{j}$};
\node [right] at (5,2) {$e_{k}$};

\draw [fill] (0,0) circle (4pt);
\draw [fill] (6,0) circle (4pt);
\draw [fill] (4,3) circle (4pt);
\end{tikzpicture}
\caption{Oriented triangle.}
\label{Oriented triangle}
\end{figure}
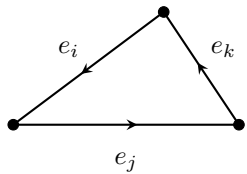

\begin{remark}\label{rem:DL presentation commute}
In the notations of Corollary~\ref{cor:DL presentation}, it follows that  $e_i$, $e_j$, and $e_k$ generate a $\zz^2$ subgroup.
\end{remark}

\begin{example}\label{ex:bbg basic examples}
If $\Gamma$ is a tree on $n$ vertices, then $\bbg \Gamma$ is a free group of rank $n-1$.
If $\Gamma=K_n$ is a complete graph on $n$ vertices, then $\bbg \Gamma = \zz^{n-1}$.
\end{example}

Moreover, as observed by Papadima and Suciu, the edge set of a spanning tree is already enough to generate the whole group.

\begin{corollary}\textnormal{(\cite[Corollary 2.3]{PapadimaSuciuAlgebraicinvariantsforBBGs})}\label{cor:PS presentation}
Let $T$ be a spanning tree of $\Gamma$. 
When the flag complex $\flag \Gamma$ on $\Gamma$ is simply connected, the group $\bbg \Gamma$ admits a finite presentation in which the generators are the edges of $T$, and every defining relator is a commutator between words in the generators.
\end{corollary}

\begin{remark}[Oriented vs unoriented edges]\label{rem:orientation}
The presentation from Corollary \ref{cor:DL presentation}  is very symmetric but clearly redundant because each (unoriented) edge appears twice.
The orientation is just an accessory tool, and one can obtain a shorter presentation by choosing an arbitrary orientation for each edge $e$, dropping the relator $e\bar e$, and allowing inverses in the relators whenever needed.
For instance, this is what happens in  Corollary~\ref{cor:PS presentation}.
Strictly speaking, each choice of orientation for the edges results in a slightly different presentation.
However, switching the orientation of an edge simply amounts to replacing a generator with its inverse.
Therefore, in the following sections, we will naively regard the generators in Corollary~\ref{cor:PS presentation} as being given by unoriented edges of $T$, and we will impose a specific orientation only when needed in a technical argument.
\end{remark}


\section{BBGs that are RAAGs}\label{section: BBGs that are RAAGs}

When $\Gamma$ is a tree or complete graph, the group $\bbg \Gamma$ is a free group or abelian group, respectively. Hence, it is a RAAG (see Example~\ref{ex:bbg basic examples}).
In this section, we identify a wider class of graphs whose associated BBGs are RAAGs.

\subsection{Tree 2-spanners}\label{sec:tree 2-spanner}
Let $\Gamma$ be a connected   graph. Recall from the introduction that a tree $2$-spanner of $\Gamma$ is a spanning tree $T$ of $\Gamma$ such that for all $x,y\in \vv T$, we have $d_T(x,y)\leq 2 d_\Gamma (x,y)$. 
If $\Gamma$ is a tree, then $\Gamma$ is a tree $2$-spanner of itself.
Here we are interested in more general graphs which admit tree $2$-spanners.
We start by proving some useful properties of tree $2$-spanners. 

\begin{lemma}\label{tree 2spanner dicothomy}
Let $T$ be a tree $2$-spanner of $\Gamma$, and let $e\in \ee \Gamma$. 
Then either $e\in \ee T$ or there is a unique triangle $(e,f,g)$ such that $f,g\in \ee T$. 
\end{lemma}
\proof
Write $e=(x,y)$, so $d_T(x,y)\leq 2 d_\Gamma (x,y)=2$.
If $e$ is not an edge of $T$, then $d_T(x,y)=2$. So, there must be some $z\in \vv T$ adjacent to both $x$ and $y$ in $\Gamma$ such that the edges $f=(x,z)$ and $g=(y,z)$ are in $T$. Obviously, the edges $e$, $f$, and $g$ form a triangle.

To see that such a triangle is unique, let  $(e,f',g')$ be another triangle such that $f',g'\in \ee T$. Then  $(f,g,f',g')$ is a cycle in the spanning tree $T$, which leads to a contradiction.
\endproof

\begin{lemma}\label{tree 2spanner triangle dicothomy}
Let $T$ be a tree $2$-spanner of $\Gamma$.
Then in every triangle of $\Gamma$, either no edge is in $T$ or two edges are in $T$.
\end{lemma}
\proof
Let $(e,f,g)$ be a triangle in $\Gamma$, and assume by contradiction that $e\in \ee T$ but $f,g\not \in \ee T$.
Then by Lemma \ref{tree 2spanner dicothomy}, the edges $f$ and $g$ are contained in uniquely determined triangles $(f,f_1,f_2)$ and $(g,g_1,g_2)$, respectively, with $f_1,f_2,g_1,g_2\in \ee T$. 
Then $(e,f_1,f_2,g_1,g_2)$ is a loop in $T$, which is absurd since $T$ is a tree.
\endproof

\begin{lemma}\label{tree 2spanner tetrahedron}
Let $T$ be a tree $2$-spanner of $\Gamma$, and let $(e,f,g)$ be a triangle in $\Gamma$ with no edges from $T$.
Then there are edges $e',f',g' \in \ee T$ that together with $e,f,g$ form a $K_4$ in $\Gamma$.
\end{lemma}
\proof
By Lemma~\ref{tree 2spanner dicothomy}, 
there are uniquely determined triangles $(e,e_1,e_2)$, $(f,f_1,f_2)$, and $(g,g_1,g_2)$ such that   $e_1,e_2, f_1,f_2,g_1,g_2 \in \ee T$.  
Let $v_e$ be a common vertex shared by $e_1$ and $e_2$ and similarly define $v_f$ and $v_g$.
If at least two vertices among $v_e$, $v_f$, and $v_g$ are distinct, then concatenating the edges $e_1,e_2,f_1,f_2,g_1,g_2$ gives a non-trivial loop in $T$, which is absurd. Thus, we have $v_e=v_f=v_g$. 
Therefore, there is a $K_{4}$ induced by the vertex $v_{e}$ and the triangle $(e,f,g)$.
\endproof


We establish the following result about the global structure of $\flag \Gamma$.
(We will prove in Corollary~\ref{cor:tree2spanner implies contractible} that if $\Gamma$ is a tree 2-spanner, then $\flag \Gamma$ is even contractible.)

\begin{lemma}\label{tree 2-spanner implies simply connected}
If $\Gamma$ has a tree $2$-spanner, then $\flag \Gamma$ is simply connected. 
\end{lemma}

\begin{proof}
It is enough to check that every cycle of $\Gamma$ bounds a disk in $\flag \Gamma$.
Let $T$ be a tree $2$-spanner of $\Gamma$, and let $C=(e_1,e_2,\dots,e_n)$ be a cycle of $\Gamma$. 
If $n=3$, then by construction $C$ bounds a triangle in $\flag \Gamma$. So, we may assume $n\geq 4$.
If $C$ contains a pair of vertices connected by an edge not in $C$, then $C$ can be split into the concatenation of two shorter cycles. 
So, we assume that $C$ contains no such a pair of vertices, that is, a chordless cycle.  In particular, all edges of $C$ are distinct.

For each $e_i\in \ee C$, either $e_i\in \ee T$ or $e_i\not \in \ee T$. 
In the second case, by Lemma~\ref{tree 2spanner dicothomy}, there are two edges $e_i^-$ and $e_i^+$ in $\ee T$ such that $(e_i,e_i^-,e_i^+)$ form a triangle.
We denote by $w_i$ the common vertex of $e_i^-$ and $e_i^+$. Note that $w_i \not \in \vv C$ and $e_i^-,e_i^+\not \in \ee C$ because $C$ is assumed to be chordless and of length $n\geq 4$.
Let $L$ be the loop obtained by the following surgery on $C$ (see Figure~\ref{fig:tree2spanner_simplyconnected}, left): for each edge $e_i$, if  $e_i\in \ee T$, then keep it; otherwise, replace it with the concatenation of the two edges $e_i^-$ and $e_i^+$.
Then $L$ is a loop made of edges of $T$. Since $T$ is a tree, the loop $L$ is contractible. Thus, if we start from a vertex of $L$ and travel along the edges of $L$ back to the starting vertex, then we must travel along each edge an even number of times in opposite directions.

Since $e_i^-,e_i^+\not \in \ee C$, each edge of $C$ appears at most once in $L$.
So, if some edge of $C$ appears in $L$, then  $L$  is  not contractible.
This proves that $\ee C \cap \ee T = \varnothing$, and therefore, we have $L=(e_1^-,e_1^+,e_2^-,e_2^+,\dots e_n^-,e_n^+)$.
Once again, since the edges of $L$ must appear an even number of times, the loop $L$ contains a repeated edge. That is, we have $e_i^+= e_{i+1}^-$ and $w_i=w_{i+1}$ for some $i$. 
Deleting the vertex $v_{i+1}$ (and the repeated edge), we obtain a shorter cycle in $T$, made of edges from $L$.
Iterating the process, we see that $w_1,\dots, w_n$ are all actually the same vertex, say $w$ (see Figure~\ref{fig:tree2spanner_simplyconnected}, right). 
Notice that every vertex of $C$ is adjacent to $w$, so $C$ is entirely contained in $\st{w,\Gamma}$. 
Therefore, the cycle $C$ bounds a triangulated disk in $\flag \Gamma$ as desired.
\end{proof}

\begin{figure}[ht!]
    \centering
    \input{pictures/tree2spanner_simply_connected}
    \caption{The construction of the loop $L$ from the cycle $C$ (left), and its contraction to a cone vertex (right).}
    \label{fig:tree2spanner_simplyconnected}
\end{figure}
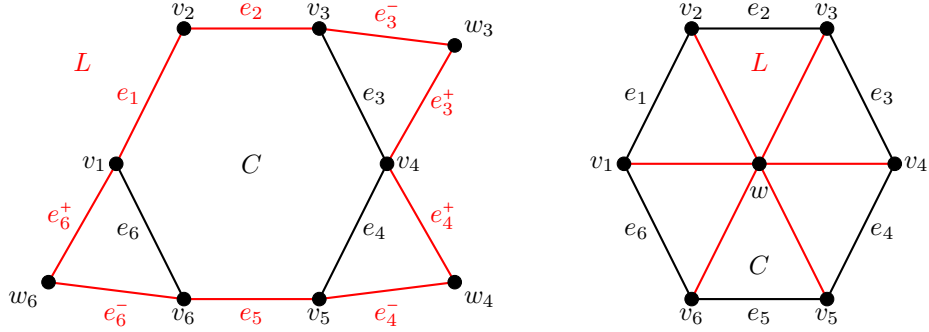


In the next statement, we show that if $\Gamma$ has a tree 2-spanner, then $\bbg \Gamma$ is a RAAG.
Even more: the tree 2-spanner itself provides a RAAG  presentation for $\bbg \Gamma$.
Let $T$ be a tree 2-spanner for $\Gamma$.
Recall from the introduction that the dual graph $T^\ast$ of $T$ is the graph whose vertices are edges of $T$, and two vertices are adjacent if and only if the corresponding edges of $T$ are contained in the same triangle of $\Gamma$.
Roughly speaking, the dual graph encodes the way in which $T$ sits inside $\Gamma$.

\begin{maintheoremc}{B}\label{containing a tree 2-spanner implies that BBG is a RAAG}
If $\Gamma$ admits a tree $2$-spanner $T$, then $\bbg \Gamma$ is a RAAG. More precisely, the Dicks--Leary presentation can be simplified to the standard RAAG presentation with generating set $\ee T$. Moreover, we have $\bbg \Gamma \cong \raag{T^\ast}$.
\end{maintheoremc}

\begin{proof}
Let $T$ be a tree $2$-spanner of $\Gamma$.
By Lemma~\ref{tree 2-spanner implies simply connected}, the flag complex $\flag \Gamma$ is simply connected. 
By Corollary~\ref{cor:DL presentation}, the Dicks--Leary presentation for $\bbg \Gamma$ is finite. The generators are the oriented edges of $\Gamma$, and the relators correspond to the oriented triangles in $\Gamma$.
By Corollary~\ref{cor:PS presentation}, the presentation can be further simplified by discarding all edges not in $T$ to obtain a presentation that only involves commutators between words in the generators.
We explicitly note that to achieve this, one also needs to choose an arbitrary orientation for each edge of $T$ (compare Remark~\ref{rem:orientation}).
To ensure that the resulting presentation is a standard RAAG presentation, we need to check that it is enough to use relators that are commutators of edges of $T$ (as opposed to commutators of more general words).
In order to do this, we check what happens to the Dicks--Leary presentation from Corollary~\ref{cor:DL presentation} when we remove a generator corresponding to an edge that is not in $T$.
The relators involving such an edge correspond to the triangles of $\Gamma$ that contain it.
One of them is the special triangle from Lemma~\ref{tree 2spanner dicothomy}, and there might be other ones corresponding to other triangles.

Let $e\in \ee \Gamma \setminus \ee T$. 
By Lemma~\ref{tree 2spanner dicothomy}, we know that there is a unique triangle $(e,f,g)$ with $f,g\in \ee T$. 
Then $(e^{\varepsilon_1},f^{\varepsilon_2},g^{\varepsilon_3})$ is an oriented triangle (in the sense of Figure~\ref{Oriented triangle}) for some suitable $\varepsilon_j = \pm 1$, where the negative exponent stands for a reversal in the orientation.
When we drop $e$ from the generating set, the relations $e^{\varepsilon_1}f^{\varepsilon_2}g^{\varepsilon_3}=1=g^{\varepsilon_3}f^{\varepsilon_2}e^{\varepsilon_1}$ can be replaced by $f^{\varepsilon_2}g^{\varepsilon_3}=e^{-\varepsilon_1}=g^{\varepsilon_3}f^{\varepsilon_2}$, hence, by the commutator $[f^{\varepsilon_2},g^{\varepsilon_3}]$ (compare with Remark~\ref{rem:DL presentation commute}).
But such a commutator can always be replaced by $[f,g]$. This is completely insensitive to the chosen orientation. This shows that the relators of the presentation from Corollary~\ref{cor:DL presentation}, which arise from the triangles provided by Lemma~\ref{tree 2spanner dicothomy}, are turned into commutators between generators in the presentation from Corollary~\ref{cor:PS presentation}.

We need to check what happens to the other type of relators.
We now show that they follow from the former type of relators and hence can be dropped.
As before, let $e\in \ee \Gamma \setminus \ee T$, and let $(e,f,g)$ be the triangle from Lemma~\ref{tree 2spanner dicothomy} having $f,g\in \ee T$.
Let $(e,f',g')$ be another triangle containing $e$.
Since $e\not \in \ee T$, by the uniqueness in Lemma~\ref{tree 2spanner triangle dicothomy} we have $e,f',g'\not \in \ee T$.
Therefore, by Lemma~\ref{tree 2spanner tetrahedron}, there are $e'',f'',g'' \in \ee T$ that form a $K_4$  together with $e$, $f'$, and $g'$; see the left picture of Figure~\ref{fig:proof of a BBG with a tree 2-spanner is a RAAG}. 
Up to relabelling, say that $e''$ is the edge of this $K_4$ that is disjoint from $e$. 
Then $(e,f'',g'')$ is a triangle containing $e$ with $f'',g''\in \ee T$. 
Again, since the triangle $(e,f,g)$ is unique, we have $\{f'',g''\}=\{f,g\}$.
In particular, the triangles $(e,f,g)$ and $(e,f',g')$ are part of a common $K_4$; see the right picture of Figure~\ref{fig:proof of a BBG with a tree 2-spanner is a RAAG}. 
The edges of this $K_4$ that are in $T$ are precisely $e''$, $f$, and $g$, and any two of them commute by Remark~\ref{rem:DL presentation commute}. 
So, the relator $ef'g'$ follows from the fact that $e$, $f'$, and $g'$ can be rewritten in terms of $f$, $g$, and $e''$.
In particular, this relator can be dropped.

Therefore, the Dicks--Leary presentation for $\bbg \Gamma$ can be simplified to a presentation in which the generating set is $\ee T$, and the relators are commutators $[e_{i},e_{j}]$, where $e_{i}$ and $e_{j}$ are in $\ee T$ and are contained in the same triangle of $\flag \Gamma$.
In particular, we have $\bbg \Gamma \cong \raag{T^\ast}$.
\end{proof}

\begin{figure}[ht!]
    \centering
    \input{pictures/proof_BBG_tree2-spanner_RAAG}
    \caption{The graph on the left shows the triangle $(e,f,g)$ and a $K_{4}$ consisting of the edges $e, f', g', e'', f''$, and $g''$. The red edges are in $\ee T$. The graph on the right illustrates the uniqueness of the triangle $(e,f,g)$.}
    \label{fig:proof of a BBG with a tree 2-spanner is a RAAG}
\end{figure}
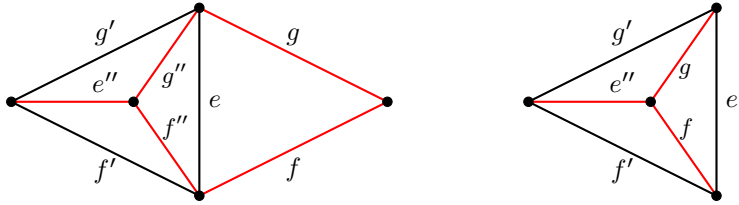

\begin{remark}
It is natural to ask which graphs admit a tree $2$-spanner. 
The problem of determining whether a graph admits a tree $2$-spanner is NP-complete (see \cite{BernPhDThesis}).
However, if a graph contains a tree $2$-spanner, then it can be found in linear time (see \cite[Theorem 4.5]{CaiCorneilTreeSpanners}).
\end{remark}

As a consequence, we have the following criterion to detect whether two BBGs are isomorphic in terms of the defining graphs in the special case where they admit tree 2-spanners.

\begin{maincorollaryc}{1}\label{cor:isom of bbgs}
Let $\Gamma$ and $\Lambda$ be two graphs admitting tree 2-spanners $T_\Gamma$ and $T_\Lambda$, respectively.
Then $\bbg \Gamma \cong \bbg \Lambda$ if and only if $T_\Gamma^\ast \cong T_\Lambda^\ast$.
\end{maincorollaryc}

\begin{proof}
The result follows from Theorem~\hyperref[containing a tree 2-spanner implies that BBG is a RAAG]{B} and the fact that two RAAGs are isomorphic if and only if their defining graphs are isomorphic; see Droms \cite{DromsIsomorphismsofGraphGroups}.
\end{proof}

\begin{remark}\label{rmk: noniso graphs give iso BBGs}
In general, non-isomorphic graphs can define isomorphic BBGs. 
For example, any two trees with $n$ vertices define the same BBG (the free group of rank $n-1$).
Notice that every tree is a tree 2-spanner of itself with a totally disconnected dual graph.
Even when $\Gamma$ admits a tree 2-spanner with a connected dual graph, the group $\bbg \Gamma$ does not determine $\Gamma$; see Example~\ref{ex:iso bbg non iso graphs}.
\end{remark}

\begin{example}\label{ex:iso bbg non iso graphs}
Denote the graphs in Figure~\ref{fig:isomorphic BBGs with nonisomorphic tree 2 spanners} by $\Gamma$ and $\Lambda$. 
Let  $T_\Gamma$ and $T_\Lambda$ be the tree $2$-spanners of $\Gamma$ and $\Lambda$, respectively, given by the red edges in the pictures.
One can see that $\Gamma \not\cong \Lambda$ as well as $T_\Gamma \not \cong T_\Lambda$. 
However, the dual graphs $T_\Gamma^\ast$ and $T_\Lambda^\ast$ are isomorphic to the path on five vertices $P_5$. Thus, by Theorem~\hyperref[containing a tree 2-spanner implies that BBG is a RAAG]{B} and  Corollary~\hyperref[cor:isom of bbgs]{1}, we have $\bbg \Gamma \cong \raag {P_5} \cong \bbg \Lambda$.

\begin{figure}[h]
    \centering
    \input{pictures/Example_Whitney_Twist}
    \caption{Non-isomorphic biconnected graphs that give isomorphic BBGs.}
    \label{fig:isomorphic BBGs with nonisomorphic tree 2 spanners}
\end{figure}
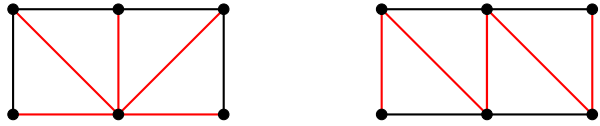
\end{example}

The following graph-theoretic statements might be of independent interest.
The first one says that any two tree 2-spanners for a graph $\Gamma$ sit in the same way inside $\Gamma$ (even though they do not have to be isomorphic as trees; see Example~\ref{ex:iso bbg non iso graphs}).
The second one strengthens the conclusion of Lemma~\ref{tree 2-spanner implies simply connected}.

\begin{corollary}
If $T_1$ and $T_2$ are tree 2-spanners for $\Gamma$, then  $T_1^\ast\cong T_2^\ast$.
\end{corollary}
\begin{proof}
Take $\Gamma=\Lambda$ in Corollary~\hyperref[cor:isom of bbgs]{1}.
\end{proof}

\begin{corollary}\label{cor:tree2spanner implies contractible}
If $\Gamma$ admits a tree $2$-spanner, then $\flag \Gamma$ is contractible.
\end{corollary}
\begin{proof}
By Theorem~\hyperref[containing a tree 2-spanner implies that BBG is a RAAG]{B}, the group $\bbg \Gamma$ is isomorphic to the RAAG $\raag \Lambda$ on some graph $\Lambda$. 
The Salvetti complex associated to $\Lambda$ is a finite classifying space for $\raag \Lambda$, so the group $\bbg \Gamma \cong \raag \Lambda$ is of type $F$.
It follows from \cite{bestvinabradymorsetheoryandfinitenesspropertiesofgroups} that $\flag \Gamma$ is simply connected and acyclic.
By the Hurewicz Theorem, the homotopy group $\pi_k(\flag \Gamma)$ is trivial for $k\geq 1$. By the Whitehead Theorem, we can conclude that $\flag \Gamma$ is contractible.
\end{proof}

\subsection{Joins and 2-trees}
In this section, we describe some ways of constructing new graphs out of old ones in such a way that the BBG defined on the resulting graph is a RAAG.

\subsubsection{Joins}
Recall from Section~\ref{subsection: notation and terminology} the definition of the join of two graphs.
It corresponds to a direct product operation on the associated RAAGs. The following corollary can also be found in \cite[Example 2.5]{PapadimaSuciuAlgebraicinvariantsforBBGs} and \cite[Proposition 3.4]{YCCIdentifyingDehnFunctionsofBBGfromtheirdefininggraphs}.
 
\begin{corollary}\label{cor: cone graph gives an isomorphism between BBG and RAAG}
Let $\Lambda$ be a graph.
If $\Gamma=\lbrace v\rbrace\ast\Lambda$, then $\bbg \Gamma\cong \raag \Lambda$.
\end{corollary}
\begin{proof}
Since $v$ is a cone vertex, the edges that are incident to $v$ form a tree $2$-spanner $T$ of $\Gamma$.
By Theorem~\hyperref[containing a tree 2-spanner implies that BBG is a RAAG]{B}, we know that $\bbg \Gamma$ is a RAAG, namely, $\bbg \Gamma \cong \raag{T^\ast}$.
The result follows from the observation that $T^\ast \cong \Lambda$.
\end{proof}

For instance, if $\Gamma$ does not contain a full subgraph isomorphic to $C_4$ or  $P_3$, then $\Gamma$ is a cone (see the first Lemma in \cite{DromsSubgroupsofGraphGroups}), and the previous corollary applies.
Actually, in this case every subgroup of $\raag \Gamma$ is known to be a RAAG by the main Theorem in \cite{DromsSubgroupsofGraphGroups}.

\begin{remark}
Corollary~\ref{cor: cone graph gives an isomorphism between BBG and RAAG} implies that the class of BBGs contains the class of RAAGs, that is, every RAAG arises as the BBG defined on some graph.
\end{remark}

\begin{remark}\label{rem: cone on a non-RAAG BBG graph is a RAAG}
Corollary~\ref{cor: cone graph gives an isomorphism between BBG and RAAG} indicates that the fact that $\bbg \Gamma$ is not a RAAG is not obviously detected by subgraphs in general.
Indeed, if $\Gamma$ is a cone over $\Lambda$, then $\bbg \Gamma$ is always a RAAG, regardless of the fact that $\bbg \Lambda$ is a RAAG or not.
\end{remark}

\begin{corollary}\label{cor: join graph gives a RAAG}
Let $\Lambda$ be a graph and $\Gamma'$ a cone graph.
If $\Gamma=\Gamma'\ast\Lambda$, then $\bbg\Gamma$ is a RAAG. 
\end{corollary}
\begin{proof}
Since $\Gamma'$ is a cone graph, so is $\Gamma$. Therefore, the group $\bbg\Gamma$ is a RAAG by Corollary \ref{cor: cone graph gives an isomorphism between BBG and RAAG}.
\end{proof}

\begin{corollary}\label{cor:center implies BBG is RAAG}
If $\raag \Gamma$ has non-trivial center, then $\bbg \Gamma$ is  a RAAG.
\end{corollary}

\begin{proof}
By \cite[The Centralizer Theorem]{ServatiusAutomorphismsofGraphGroups}, when $\raag \Gamma$ has non-trivial center, there is a complete subgraph $\Gamma' \subseteq \Gamma$ such that each vertex of $\Gamma'$ is adjacent to every other vertex of $\Gamma$. That is, the graph $\Gamma$ decomposes as $\Gamma=\Gamma'\ast \Lambda$, where $\vv \Lambda =\vv \Gamma \setminus \vv \Gamma'$. Since a complete graph is a cone graph, the result follows from Corollary~\ref{cor: join graph gives a RAAG}.
\end{proof}

\begin{remark}
BBGs defined on arbitrary graph joins are not necessarily isomorphic to RAAGs. For example, the cycle of length four $C_4$ is the join of two pairs of non-adjacent vertices. The associated RAAG is $\ff_2\times \ff_2$, and the associated BBG is not a RAAG because it is not even finitely presented (see \cite{bestvinabradymorsetheoryandfinitenesspropertiesofgroups}).
\end{remark}


\subsubsection{2-trees}\label{section: example 2-trees}
Roughly speaking, a \textit{$2$-tree} is a graph obtained by gluing triangles along edges in a tree-like fashion. 
More formally, the class of $2$-trees is defined recursively as follows: the graph consisting of a single edge is a $2$-tree, and then a graph $\Gamma$ is a $2$-tree if it contains a vertex $v$ such that the neighborhood of $v$ in $\Gamma$ is an edge and the graph obtained by removing $v$ from $\Gamma$ is still a $2$-tree.
The trefoil graph from Figure~\ref{fig:trefoil} is an explicit example of a $2$-tree. 
A general $2$-tree may not be a triangulation of a $2$-dimensional disk as it can have branchings; see Figure~\ref{fig: A 2-tree that is not a triangulation of a disk} for an example. 
It is not hard to see that the flag complex on a $2$-tree is simply connected. So, the associated BBG is finitely presented and has only commutator relators.

Cai showed that a $2$-tree contains no trefoil subgraphs if and only if it admits a tree $2$-spanner; see \cite[Proposition 3.2]{caionspanning2trees}. The next corollary follows from Cai's result and Theorem~\hyperref[containing a tree 2-spanner implies that BBG is a RAAG]{B}. 
In Section~\ref{section: BBGs on 2-dim flag complexes}, we will prove a more general result that especially implies the converse of the following statement.

\begin{corollary}\label{cor: trefoil-free 2-tree implies BBG=RAAG}
Let $\Gamma$ be a $2$-tree.
If $\Gamma$ is trefoil-free, then $\bbg \Gamma$ is a RAAG. 
\end{corollary}

\begin{example}[A bouquet of triangles]\label{ex: a bouquet of triangles}
Let $\Gamma$ be the $2$-tree shown in Figure~\ref{fig: A 2-tree that is not a triangulation of a disk}. Since $\Gamma$ does not contain trefoil subgraphs, the group $\bbg \Gamma$ is a RAAG by Corollary~\ref{cor: trefoil-free 2-tree implies BBG=RAAG}. The reader can check that the red edges form a tree $2$-spanner of $\Gamma$.
\end{example}

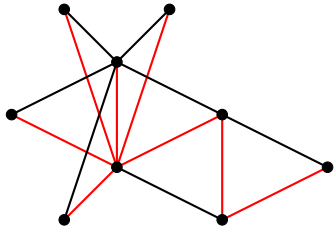
\begin{figure}[ht!]
    \centering
    \input{pictures/ex_bbg_on_a_bouquet_of_triangles}
    \caption{A $2$-tree whose flag complex is not a triangulated disk.}
    \label{fig: A 2-tree that is not a triangulation of a disk}
\end{figure}


\section{BBGs that are not RAAGs}\label{section: BBGs that are not RAAGs}

While in the previous section, we have provided a condition on a graph $\Gamma$ to ensure that $\bbg \Gamma$ is a RAAG, in this section, we want to obtain a condition on $\Gamma$ to guarantee that $\bbg \Gamma$ is not a RAAG.
The main technical tool consists of a description of the BNS-invariants of BBGs in terms of the defining graphs.
Recall that we are always assuming that $\Gamma$ is connected, and that this ensures that $\bbg \Gamma$ is finitely generated; see \cite{bestvinabradymorsetheoryandfinitenesspropertiesofgroups,DicksLearypresentationsforsubgroupsofArtingroups}.

\subsection{BNS-invariants of finitely generated groups}\label{sec:BNS stuff}

Let $G$ be a finitely generated group. A \emph{character} of $G$ is a homomorphism $\chi\colon G\rightarrow\rr$. Two characters $\chi_{1}$ and $\chi_{2}$ are equivalent, denoted by $\chi_{1}\sim\chi_{2}$, whenever $\chi_{1}=\lambda\chi_{2}$ for some positive real number $\lambda$. Denote by $[\chi]$ the equivalence class of $\chi$. 
The set of equivalence classes of non-zero characters of $G$ is called the \emph{character sphere} of $G$:
$$
\chars G=\big\lbrace[\chi] \ \vert \ \chi\in\mathrm{Hom}(G,\rr)\setminus\lbrace0\rbrace\big\rbrace.
$$
The character sphere naturally identifies with the unit sphere in $\mathrm{Hom}(G,\rr)$ (with respect to some background inner product), so by abuse of notation, we will often write $\chars G \subseteq \mathrm{Hom}(G,\rr)$.
A character $\chi\colon G \to \rr$ is called \textit{integral}, \textit{rational}, or \textit{discrete} if its image is an infinite cyclic subgroup of $\zz, \qq$, or $\rr$, respectively. In particular, an integral character is rational,  a rational character is discrete, and the equivalence class of a discrete character always contains an integral representative.

Let $\mathcal{S}$ be a finite generating set for $G$, and let $\operatorname{Cay}(G,\mathcal{S})$ be the Cayley graph for $G$ with respect to $\mathcal S$.
Note that the elements of $G$ are identified with the vertex set of the Cayley graph.
For any character $\chi\colon G\rightarrow\rr$, let  $\operatorname{Cay}(G,\mathcal{S})_{\chi\geq0}$ be the full subgraph of $\operatorname{Cay}(G,\mathcal{S})$ spanned by $\{g\in G \mid \chi (g)\geq 0\}$.
Bieri, Neumann, and Strebel \cite{bierineumannstrebelageometricinvariantofdiscretegroups} introduced a geometric invariant of $G$, known as the \emph{BNS-invariant} $\bns G$ of $G$, which is defined as the following subset of $\chars G$:
$$
\bns G =\big\lbrace[\chi]\in \chars G \ \big\vert \ \mathrm{Cay}(G,\mathcal{S})_{\chi\geq0} \ \text{is connected}\big\rbrace.
$$
They also proved that the BNS-invariant of $G$ does not depend on the generating set $\mathcal S$.

The interest in $\bns G$ is due to the fact that it can detect finiteness properties of normal subgroups with abelian quotients, such as kernels of characters. 
For instance, the following statement can be taken as an alternative definition of what it means for a discrete character to belong to $\bns G$ (see \cite[\S 4]{bierineumannstrebelageometricinvariantofdiscretegroups} or \cite[Corollary A4.3]{strebelnotesonthesigmainvariants}).

\begin{theorem}\label{thm:bns fg}
Let $\chi\colon G \to \rr$ be a discrete character. Then $\ker (\chi) $ is finitely generated if and only if both $[\chi]$ and $[-\chi]$ are in $\bns G$.
\end{theorem}

As a major motivating example, when $G$ is the fundamental group of a compact $3$-manifold $M$, the BNS-invariant $\bns G$ describes all the possible ways in which $M$ fibers over the circle with fiber a compact surface (see \cite{ST62,TH86,bierineumannstrebelageometricinvariantofdiscretegroups}).

\begin{remark}\label{rmk:bns antipodal}
For each group $G$ of interest in this paper, it admits an automorphism that acts as the antipodal map $\chi \mapsto -\chi$ on $\mathrm{Hom}(G,\rr)$. 
In this case, the BNS-invariant $\bns G$ is invariant under the antipodal map. Therefore, its rational points correspond exactly to discrete characters with finitely generated kernels.
\end{remark}

\begin{remark}[The complement and the missing subspheres]\label{rmk:missing subspheres}
It is often the case that the BNS-invariant $\bns G$ is better described in terms of its complement in the character sphere $\chars G$. 
Moreover, for many groups of interest, the complement of the BNS-invariant is often a union of subspheres (see \cite{meierthebierineumannstrebelinvariantsforgraphgroups,kochloukovamendonontheBNSRsigmainvariantsoftheBBGs,KO21,BG84,CL16} for examples).
In this paper, the \textit{complement} of $\bns G$ is by definition $\bnsc {G}=\chars G \setminus \bns G$.
A \textit{great subsphere} is defined as a subsphere of $\chars G$ of the form $S_W=\chars G \cap W$, where $W$ is  a linear subspace of $\operatorname{Hom}(G,\rr)$ going through the origin.
We say that a great subsphere $S_W$ is a \textit{missing subsphere} if $S_W\subseteq \bnsc G$.
The subspace $W$ is the linear span of $S_W$ and is called a \textit{missing subspace}.
\end{remark}


\subsubsection{The BNS-invariants of \texorpdfstring{RAAGs}{right-angled Artin groups}}\label{sec:bns for raags}

The BNS-invariants of RAAGs have a nice description given by Meier and VanWyk \cite{meierthebierineumannstrebelinvariantsforgraphgroups}. Let $\Gamma$ be a graph and $\chi\colon \raag \Gamma\rightarrow\rr$ a character of $\raag \Gamma$. Define the \emph{living subgraph} $\living \chi$ of $\chi$ to be the full subgraph of $\Gamma$ on the vertices $v$ with $\chi(v)\neq0$ and the \emph{dead subgraph} $\dead \chi$ of $\chi$ to be the full subgraph of $\Gamma$ on the vertices $v$ with $\chi(v)=0$. 
Note that $\living \chi$ and $\dead \chi$ are disjoint, and they do not necessarily cover $\Gamma$.
A subgraph $\Gamma'$ of $\Gamma$ is \emph{dominating} if every vertex of $\Gamma\setminus\Gamma'$ is adjacent to some vertex of $\Gamma'$.

\begin{theorem}\textnormal{(A graphical criterion for $\bns{\raag \Gamma}$, \cite[Theorem 4.1]{meierthebierineumannstrebelinvariantsforgraphgroups})}\label{BNS-invariant for raag}
Let $\chi\colon \raag \Gamma\rightarrow\rr$ be a character. Then $[\chi]\in \bns{\raag \Gamma}$ if and only if  $\living \chi$ is connected and dominating.
\end{theorem}

By Theorem~\ref{thm:bns fg}, if $\chi$ is a discrete character, then $\living \chi$ detects whether $\ker (\chi)$ is finitely generated.
Indeed, in a RAAG, the map sending a generator to its inverse is a group automorphism. Hence, the set $\bns{\raag \Gamma}$ is invariant under the antipodal map $\chi\mapsto -\chi$ (see Remark~\ref{rmk:bns antipodal}).

We find it convenient to work with the following reformulation of the condition in Theorem~\ref{BNS-invariant for raag}. 
It previously appeared inside the proof of \cite[Corollary 3.4]{lorenzo}.
We include a proof for completeness.
For the sake of clarity, in the following lemma, the graph $\Lambda$ is a  subgraph of $\dead \chi$ that is separating as a subgraph of $\Gamma$, but it may not separate $\dead \chi$ (see Figure~\ref{fig:bad dead subgraph} for an example).

\begin{lemma}\label{lem:dead subgraph criterion for RAAGs}
Let $\chi\colon \raag \Gamma \rightarrow \rr$ be a non-zero character. Then the following statements are equivalent.
\begin{enumerate}
    \item \label{item:living} The living graph $\living \chi$ is either not connected or not dominating.
    \item \label{item:dead} There exists a full subgraph $\Lambda \subseteq \Gamma$ such that $\Lambda$ separates $\Gamma$ and $\Lambda \subseteq \dead \chi$. 
\end{enumerate}

\end{lemma}
\proof
We begin by proving that \eqref{item:living} implies \eqref{item:dead}. 
If $\living \chi$ is not connected, then $\Lambda=\dead \chi$ separates $\Gamma$.
If $\living \chi$ is not dominating, then there is a vertex $v\in \vv \Gamma$ such that $\chi$ vanishes on $v$ and on all the vertices adjacent to $v$.
Since $\chi$ is non-zero, the vertex $v$ is not a cone vertex. 
In particular, the graph $\Lambda = \lk{v,\Gamma}$ is a subgraph of $\dead \chi$. Moreover, the subgraph $\Lambda$ is a full separating subgraph of $\Gamma$, as desired. 

To prove that \eqref{item:dead} implies \eqref{item:living}, assume that $\living \chi$ is connected and dominating. 
Let $\Lambda \subseteq \dead \chi$ be a full subgraph of $\Gamma$, and let $u_1, u_2\in \vv \Gamma \setminus \vv \Lambda$. 
We want to show that $u_1$ and $u_2$ can be connected in the complement of $\Lambda$. There are three cases.
Firstly, if $u_{1}$ and $u_{2}$ are vertices of $\living \chi$, then they are connected by a path entirely in $\living \chi$. 
Secondly, if both $u_{1}$ and $u_{2}$ are vertices of $\dead \chi$, then they are adjacent to some vertices in $\living \chi$, say $v_{1}$ and $v_{2}$, respectively. 
Then we can extend a path in $\living \chi$ between $v_{1}$ and $v_{2}$ to a path between $u_{1}$ and $u_{2}$ avoiding $\vv\Lambda$.
Finally, suppose that $u_{1}$ is a vertex of $\living \chi$ and $u_{2}$ is a vertex of $\dead \chi$, then $u_{2}$ is adjacent to a vertex $v_{2}$ of $\living \chi$. Again, we can extend a path in $\living \chi$ between $u_{1}$ and $v_{2}$ to a path between $u_{1}$ and $u_{2}$ avoiding $\vv \Lambda$. As a result, we have connected $u_1$ to $u_2$ with a path disjoint from $\Lambda$. This shows that $\Lambda$ is not separating, which contradicts \eqref{item:dead}.
\endproof

Notice that a subgraph $\Lambda$ arising from Lemma~\ref{lem:dead subgraph criterion for RAAGs} may not be connected and may not be equal to $\dead \chi$. Also, it may not even be a union of connected components of $\dead \chi$.
This is in particular true when looking for a minimal such $\Lambda$; see Figure~\ref{fig:bad dead subgraph}.

\begin{figure}[ht!]
    \centering
    \input{pictures/bad_dead_subgraph_tikz_version}
    \caption{The subgraph $\Lambda$ in Lemma~\ref{lem:dead subgraph criterion for RAAGs} may not be a union of connected components of $\dead \chi$. Here, the graph $\Lambda$ is given by the two red vertices.}
    \label{fig:bad dead subgraph}
\end{figure}
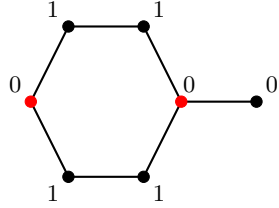

\begin{remark}\label{rem:complement BNS for RAAGs}
It follows from Theorem~\ref{BNS-invariant for raag} that  $\bnsc{\raag \Gamma}$ is a rationally defined spherical polyhedron, given by a union of missing subspheres (see \cite[Theorem 5.1]{meierthebierineumannstrebelinvariantsforgraphgroups}).
Moreover, each (maximal) missing subsphere consists of characters that vanish on a (minimal) separating subgraph of $\Gamma$, thanks to Lemma~\ref{lem:dead subgraph criterion for RAAGs} (see also \cite[Proposition A4.14]{strebelnotesonthesigmainvariants}).
For example, the missing hyperspheres in $\bnsc{\raag\Gamma}$ are in bijective correspondence with the cut vertices of $\Gamma$. 
We will further discuss the correspondence in Example \ref{ex: bns of raag on tree}.
\end{remark}

\subsubsection{The BNS-invariants of \texorpdfstring{BBGs}{Bestvina-Brady groups}}

As in the case of RAAGs, some elementary properties of the BNS-invariants of BBGs can be seen directly from the defining graph.
\begin{example}
The graph  $\Gamma$ is complete if and only if $\bns{\bbg \Gamma}=S(\bbg \Gamma)$. 
Indeed, in this case, the group $\bbg \Gamma$ is free abelian.
\end{example}

\begin{example}\label{ex:cut vertex implies empty bns}
At the opposite extreme, if $\Gamma$ is connected and has a cut vertex, then $\bns{\bbg \Gamma} =\varnothing$ (see \cite[Corollary 15.10]{PapadimaandSuciuBNSRinvariantsandHomologyJumpingLoci}). 
Vice versa, if $\Gamma$ has no cut vertices and $\bbg \Gamma$ is finitely presented, then we will prove in Corollary~\ref{cor:bns for bbg non empty}  that $\bns{\bbg \Gamma}\neq \varnothing$.
\end{example}

The following lemma shows that the BNS-invariant of a BBG is invariant under the antipodal map,  as in the case of a RAAG. 

\begin{lemma}\label{graph orientation reversing map is an isomorphism on BBG}
For all $\chi\colon\bbg \Gamma\to\rr$, if $[\chi]\in\bns{\bbg \Gamma}$, then $[-\chi]\in\bns{\bbg \Gamma}$.
\end{lemma}

\begin{proof}
Choose an orientation for the edges of $\Gamma$, and let $f\colon\Gamma\to\Gamma$ be the map that reverses the orientation on each edge. 
Then $f$ induces an automorphism $f_{\ast}\colon\bbg \Gamma\to\bbg \Gamma$ which sends every generator $e$ to its inverse $e^{-1}$. 
Then the lemma follows from the fact that $-\chi = \chi \circ f_\ast$ (see Remark~\ref{rmk:bns antipodal}).
\end{proof}

Beyond these observations, not many explicit properties are known, and more refined tools are needed. We will use a recent result of Kochloukova and Mendon\c{c}a that relates the BNS-invariant of a BBG to that of the ambient RAAG. 
The following statement is a particular case of \cite[Corollary 1.3]{kochloukovamendonontheBNSRsigmainvariantsoftheBBGs}.

\begin{proposition}\label{character in the BNS of BBG iff all the extensions are in the BNS of RAAG}
Let $\Gamma$ be a connected graph, and let $\chi:\bbg \Gamma \to \rr$ be a character. 
Then $[\chi]\in\bns{\bbg \Gamma}$ if and only if $[\hat{\chi}]\in\bns{\raag \Gamma}$ for every character $\hat{\chi}:\raag \Gamma \to \rr$ that extends $\chi$.
\end{proposition}
  
Proposition~\ref{character in the BNS of BBG iff all the extensions are in the BNS of RAAG} allows one to recover the previous observations as well as more properties of the BNS-invariants of BBGs, which are reminiscent of those of RAAGs. For instance, the complement of the BNS-invariant of a BBG is a rationally defined spherical polyhedron (see \cite[Corollary 1.4]{kochloukovamendonontheBNSRsigmainvariantsoftheBBGs} and compare with Remark~\ref{rem:complement BNS for RAAGs}).

\subsubsection{Coordinates and labellings}\label{sec:coordinates}
Here, we want to describe a useful parametrization for $\operatorname{Hom}(\raag \Gamma, \rr)$ and $\operatorname{Hom}(\bbg \Gamma, \rr)$ in terms of labelled graphs.
This is based on the following elementary observation about a class of groups with a particular type of presentation that includes RAAGs and BBGs.

\begin{lemma}\label{lem:zero exp sum presentation}
Let $G$ be a group with a presentation $G=\langle \mathcal{S} | \mathcal{R}\rangle$ in which for each generator $s$ and each relator $r$, the exponent sum of $s$ in $r$ is zero.
Let $A$ be an abelian group.
Then there is a bijection between  $\operatorname{Hom}(G,A)$ and $\{ f\colon \mathcal{S} \to A \}$.
\end{lemma}
\begin{proof}
Given a homomorphism $G\to A$, one obtains a function $\mathcal{S} \to A$ just by restriction.
Conversely, let $f \colon \mathcal{S} \to A$ be any function and let $\widetilde f \colon \ff (\mathcal{S}) \to A$ be the induced homomorphism on the free group on $\mathcal{S}$.
Let $r\in \ff (\mathcal{S})$ be a relator for $G$. 
Since the exponent sum of each generator in $r$ is zero and $A$ is abelian, we have that $\widetilde{f}(r)$ is trivial in $A$.
Therefore, the homomorphism $\widetilde f\colon \ff (\mathcal{S}) \to A$ descends to a well-defined homomorphism $G\to A$.
\end{proof}

A typical example of a relator in which the exponent sum of every generator is zero is a commutator.
In particular, the standard presentation for a RAAG and the simplified Dicks--Leary presentation for a BBG in Corollary~\ref{cor:PS presentation} are presentations of this type.
We now show how Lemma~\ref{lem:zero exp sum presentation} can be used to introduce nice coordinates on $\operatorname{Hom}(G,\rr)$ for these two classes of groups.

Let $\Gamma$ be a connected graph, and let $\vv \Gamma = \lbrace v_{1},\dots,v_{n}\rbrace$. 
By Lemma~\ref{lem:zero exp sum presentation}, a homomorphism $\chi \colon \raag \Gamma \to \rr$ is uniquely determined by its values on $\vv \Gamma$.
Therefore, we get a natural identification
$$ \mathrm{Hom}(\raag \Gamma,\rr) \to \rr^{|\vv \Gamma|}, \quad \chi \mapsto (\chi(v_1),\dots, \chi(v_n)).$$
In other words, a character $\chi\colon \raag \Gamma \to \rr$ is the same as a labelling of $\vv \Gamma$ by real numbers.
A natural basis for $\mathrm{Hom}(\raag \Gamma,\rr)$ is given by the characters $\chi_1,\dots,\chi_n$, where $\chi_i(v_j)=\delta_{ij}$.

For BBGs, a similar description is available in terms of edge labellings.
Different from RAAGs, not every assignment of real numbers to the edges of $\Gamma$ corresponds to a character. 
Indeed, the labels along an oriented cycle must sum to zero.
So, assigning the labels on sufficiently many edges already determines the labels on the other ones.
To find a clean description, we assume that the flag complex $\flag \Gamma$ is simply connected, and we fix a spanning tree $T$ of $\Gamma$ with $\ee T=\lbrace e_{1},\dots,e_{m}\rbrace$. 
By Corollary~\ref{cor:PS presentation}, we know that the Dicks--Leary presentation can be simplified to have only $\ee T$ as a generating set and all relators are commutators.
By Lemma~\ref{lem:zero exp sum presentation}, we get an identification
$$ \mathrm{Hom}(\bbg \Gamma,\rr) \to \rr^{|\ee T|}, \quad \chi \mapsto (\chi(e_1),\dots, \chi(e_m)).$$
In other words, a character $\chi\colon \bbg \Gamma \to \rr$ is encoded by a labelling of $\ee T$ by real numbers.
To obtain a basis for $\mathrm{Hom}(\bbg \Gamma,\rr)$, one can take the characters $\chi_1,\dots,\chi_m$, where $\chi_i(e_j)=\delta_{ij}$, with respect to some arbitrary orientation of the edges of $T$ (compare Remark~\ref{rem:character on edges}).

\begin{remark}[Computing a character on an edge]\label{rem:character on edges}
Note that there is a slight abuse of notation: strictly speaking, in order to see an edge $e$ as an element of $\bbg \Gamma$, one needs to orient it.
So, it only makes sense to evaluate a character $\chi$ on an oriented edge (see Remark~\ref{rem:orientation}).
However, the value of $\chi(e)$ with respect to the two possible orientations just differs by a sign.
Indeed, if we change the orientation of an edge and the sign of the corresponding label, then we obtain a different description of the same character of $\bbg \Gamma$.
In particular, it makes sense to say that a character vanishes or not on an edge, regardless of orientation.
Moreover, given the symmetry of $\bns{\bbg \Gamma}$ under sign changes (see Remark~\ref{rmk:bns antipodal} and Lemma~\ref{graph orientation reversing map is an isomorphism on BBG}), it is still quite meaningful and useful to think of a character as an edge labelling for a spanning tree $T$.
\end{remark}

As a result of the previous discussions, we obtain  the following lemma, which we record for future reference.

\begin{lemma}\label{lem:characters of f.p BBGs are given by assigning values on spanning trees}
Let $\Gamma$ be a graph with $\flag \Gamma$ simply connected. 
Let $T$ be a spanning tree of $\Gamma$. 
Fix an orientation for the edges of $T$. 
Then the following statements hold.
\begin{enumerate}
    \item A character $\chi\colon\bbg \Gamma\to\rr$ is uniquely determined by its values on $\ee T$.
    \item Any assignment $\ee T \to \rr$ uniquely extends to a character $\chi\colon\bbg \Gamma\to\rr$.
\end{enumerate}
\end{lemma}

We conclude this section with a description of a natural restriction map.
Recall from Theorem~\ref{thm:DL presentation embedding} that $\bbg \Gamma$ embeds in $\raag \Gamma$ via $e\mapsto \tau e(\iota e)^{-1}$ for each oriented edge $e$, where $\iota e$ and $\tau e$ respectively denote the initial vertex and terminal vertex of $e$.
We have an induced restriction map
$$ r\colon \operatorname{Hom}(\raag \Gamma,\rr) \to \operatorname{Hom}(\bbg \Gamma,\rr), \ \hat \chi \mapsto r\hat\chi,$$
where $(r\hat \chi) (e)=\hat \chi (\tau e) - \hat \chi (\iota e)$.
See Figure~\ref{fig:ex_bns_bbg_conditions_needed} for examples.
The next result follows from the Dicks--Leary presentation in Theorem~\ref{thm:DL presentation embedding}, so it holds without additional assumptions on $\Gamma$.

\begin{lemma}\label{lem:existence extension}
Let $\Gamma$ be a connected graph.
The restriction map $r\colon \operatorname{Hom}(\raag \Gamma,\rr) \to \operatorname{Hom}(\bbg \Gamma,\rr)$ is a linear surjection, and its kernel consists of the characters defined by the constant functions $\vv \Gamma \to \rr$.
\end{lemma} 
\begin{proof}
The map $r$ is clearly linear. 
Let us prove that it is surjective.
Let $\chi\colon \bbg \Gamma \to \rr$ be a character.
Define a character $\hat{\chi}\in\mathrm{Hom}(\raag \Gamma,\rr)$ by prescribing its values on vertices as follows.
Fix some $v_0 \in \vv \Gamma$ and choose $ \hat\chi (v_0) \in \rr$ arbitrarily.
Pick a vertex $v\in \vv \Gamma$ and choose an oriented path $p=e_{1}\dots e_{k}$ connecting $v_0$ to $v$.
Here, we mean that $p$ is oriented from $v_0$ to $v$ and that all the edges along $p$ are given the induced orientation. In particular, the edges of $p$ can be seen as elements of $\bbg \Gamma$.
Define the value at $v$ to be:
\begin{equation}\label{eq:extension}
\hat{\chi}(v)=\hat \chi (v_0)+\sum^{k}_{i=1}\chi(e_{i}).    
\end{equation}

We claim that $\hat{\chi}\colon \vv \Gamma\to \rr$ is well-defined.
Suppose that there is another oriented path $p'=e'_{1}\dots e'_{h}$ 
from $v_0$ to $v$. 
Then the loop $p(p')^{-1}$ gives a relator in the Dicks--Leary presentation for $\bbg \Gamma$.
Thus, we have
$$\chi (e_{1}\dots e_{k} (e'_{1}\dots e'_{h})^{-1})=0.  $$
In other words, we have
$$\sum^{k}_{i=1}\chi(e_{i}) = \sum^{h}_{j=1}\chi(e'_{j})$$
Therefore, the value $\hat{\chi}(v)$ does not depend on the choice of a path from $v_0$ to $v$, as desired.
This provides that $\hat \chi:\raag \Gamma \to \rr$ is a character. 
A direct computation shows that for each oriented edge $e$, we have $\chi(e)=\hat \chi (\tau v)- \hat \chi (\iota v)$. That is, the character $\hat \chi$ is an extension of $\chi$ to $\raag \Gamma$.
 
To describe the kernel of $r$, note that if $\hat \chi$ is constant on $\vv \Gamma$,  then for each oriented edge $e$ we have $(r\hat \chi)(e)=\hat \chi (\tau e) - \hat \chi (\iota e)=0$.
Conversely, let $\hat \chi \in \ker (r)$. It follows from \eqref{eq:extension} that $\hat \chi(v)=\hat \chi(w)$ for any choice of $v,w\in \vv \Gamma$.
\end{proof}

Note that the (non-zero) characters defined by the constant functions $\vv \Gamma \to \rr$ all differ by a (multiplicative) constant.
In particular, they all have the same kernel, which is precisely $\bbg \Gamma$.
The restriction map $r$ has a natural linear section
$$ s\colon \operatorname{Hom}(\bbg \Gamma,\rr) \to \operatorname{Hom}(\raag \Gamma,\rr), \  \chi \mapsto s\chi,$$
defined as follows.
Let $\hat \chi$ be any extension of $\chi$ to $\raag \Gamma$.
Then define
$$ s\chi = \hat \chi - \frac{1}{|\vv \Gamma|} \sum_{v\in \vv \Gamma} \hat \chi (v).$$
The image of $s$ is a hyperplane $W$ going through the origin that can be regarded  as a copy of $\operatorname{Hom}(\bbg \Gamma,\rr)$ inside $\operatorname{Hom}(\raag \Gamma,\rr)$.

Recall that if $\vv \Gamma=\{v_1,\dots, v_n\}$, then $\operatorname{Hom}(\raag \Gamma,\rr)$ carries a canonical basis given by the characters $\chi_1,\dots, \chi_n$ such that $\chi_i(v_j)=\delta_{ij}$.
Fix an inner product that makes this basis orthonormal.
Then $\ker (r) = \operatorname{span}(1,\dots,1)$, the hyperplane $W$ is the orthogonal complement of $\ker (r)$, and the restriction map (or rather the composition $s\circ r$) is given by the orthogonal projection onto $W$.

It is natural to ask how this  behaves with respect to the BNS-invariants, that is, whether $r$ restricts to a map $\bns{\raag \Gamma} \to \bns{\bbg \Gamma}$.
In general, this is not the case. For instance, the set  $\bns{\bbg \Gamma}$ could be empty even if $\bns{\raag \Gamma}$ is not empty.
On the other hand, the restriction map $r$ maps each missing subspace of $\operatorname{Hom}(\raag \Gamma,\rr)$ into one of the missing subspaces of $\operatorname{Hom}(\bbg \Gamma,\rr)$ (compare Remark~\ref{rmk:missing subspheres} and Remark~\ref{rem:general position arrangements}).
Indeed, one way to reinterpret the content of Proposition~\ref{character in the BNS of BBG iff all the extensions are in the BNS of RAAG} (a particular case of \cite[Corollary 1.3]{kochloukovamendonontheBNSRsigmainvariantsoftheBBGs}) is to say that if $\chi\in \operatorname{Hom}(\bbg \Gamma, \rr) \cong W$, then $[\chi]\in \bns{\bbg \Gamma}$ if and only if the line parallel to $\ker (r)$ passing through $\chi$ avoids all the missing subspaces of $\operatorname{Hom}(\raag \Gamma,\rr)$.


\subsection{A graphical criterion for \texorpdfstring{$\bns{\bbg\Gamma}$}{the BNS-invariant of a BBG}}

In this subsection, we give a graphical criterion for the BNS-invariants of BBGs that is analogous to Theorem~\ref{BNS-invariant for raag}. 

Let $\chi\in\mathrm{Hom}(\bbg \Gamma,\rr)$ be a non-zero character. An edge $e\in \ee \Gamma$ is called a \emph{living edge} of $\chi$ if $\chi(e)\neq0$; it is called a \emph{dead edge} of $\chi$ if $\chi(e)=0$.
This is well-defined, regardless of orientation, as explained in Remark~\ref{rem:character on edges}.
We define the \emph{living edge subgraph}, denoted by $\livingedge \chi$, to be the subgraph of $\Gamma$ consisting of the living edges of $\chi$. 
The \emph{dead edge subgraph} $\deadedge \chi$ is the subgraph of $\Gamma$ consisting of the dead edges of $\chi$.
We will also say that $\chi$ \textit{vanishes on} any subgraph of $\deadedge \chi$ because the associated labelling (in  the sense of \S\ref{sec:coordinates}) is zero on each edge of $\deadedge \chi$.
Notice that $\livingedge \chi$ and $\deadedge \chi$ cover $\Gamma$, but they are not disjoint; they intersect at vertices. 
Also, note that $\livingedge \chi$ and $\deadedge \chi$ are not always full subgraphs and do not have isolated vertices.
Moreover, in general, the dead subgraph of an extension of $\chi$ is only a proper subgraph of $\deadedge \chi$. See Figure~\ref{fig:not full} for an example displaying all these behaviors.

\begin{figure}[h]
    \centering
    \input{pictures/living_edge_subgraph_dead_edge_subgraphs_not_full}
    \caption{The dead edge subgraph $\deadedge \chi$ consists of a pair of opposite edges, and the living edge subgraph  $\livingedge \chi$ consists of the remaining edges. Neither is a full subgraph. Moreover, the dead subgraph of any extension of $\chi$  is a proper subgraph of $\deadedge \chi$.}
\label{fig:not full}
\end{figure}
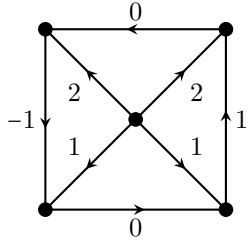

The next lemma establishes a relation between the dead edge subgraph of a character of a BBG and the dead subgraphs of its extensions to the ambient RAAG.
Note that the statement fails without the assumption that $\Lambda$ is connected (see the example in Figure~\ref{fig:not full} once again).

\begin{lemma}\label{a subgraph is edge dead iff there is an extension of characters on BBG to RAAG}
Let $\Gamma$ be a connected   graph and let $\Lambda\subseteq\Gamma$ be  a connected subgraph with at least one edge. 
Let $\chi\in\mathrm{Hom}(\bbg \Gamma,\rr)$ be a non-zero character. 
Then $\Lambda\subseteq\deadedge \chi$ if and only if there is an extension $\hat{\chi}\in\mathrm{Hom}(\raag \Gamma,\rr)$ of $\chi$ such that $\Lambda\subseteq\mathcal{D}(\hat{\chi})$.
\end{lemma}
\begin{proof}
Suppose $\Lambda\subseteq\deadedge \chi$. 
By Lemma~\ref{lem:existence extension}, there exists an extension of $\chi$ to $\raag \Gamma$, unique up to additive constants.
In particular, if we fix a vertex $v_0\in \vv \Lambda$, then we can find an extension $\hat \chi \in \mathrm{Hom}(\raag \Gamma,\rr)$ such that $\hat \chi (v_0)=0$.
Let $v\in \vv \Lambda$.
Since $\Lambda$ is connected, there is a path $p$ connecting $v_0$ to $v$ entirely contained in $\Lambda$.
Since $\hat \chi$ extends $\chi$ and $\chi$ vanishes on edges of $p$, a direct computation shows that $\hat \chi (v)=0$. Therefore, we have $\Lambda\subseteq\mathcal{D}(\hat{\chi})$.

For the other direction, let $\hat{\chi}\in\mathrm{Hom}(\raag \Gamma,\rr)$ be an extension of $\chi$ such that $\Lambda\subseteq\mathcal{D}(\hat{\chi})$. For every oriented edge $e=(v,w)\in \ee \Lambda$, we have $\chi(e)=\hat{\chi}(vw^{-1})=\hat{\chi}(v)-\hat{\chi}(w)=0$. Thus, the edge $e$ is in $\deadedge \chi$. Hence, we have $\Lambda\subseteq\deadedge \chi$.
\end{proof}

The main reason to focus on the dead edge subgraph instead of the living edge subgraph is that it is not clear how to transfer connectivity conditions from $\living{\hat \chi}$ to $\livingedge \chi$.
On the other hand, the disconnecting features of $\dead{\hat \chi}$ do transfer to $\deadedge \chi$.
This is showcased by the following example.

\begin{example}\label{ex:no good living edge subgraph criterion}
Let $\Gamma$ be a cone over the path $P_5$ and consider a character $\hat{\chi}\colon\raag \Gamma\to\zz$ as shown in Figure \ref{fig:edge living subgraph is not enough to construct an extension character}.
The living subgraph $\mathcal{L}(\hat{\chi})$ is neither connected nor dominating. 
It follows from Theorem~\ref{BNS-invariant for raag} that $[\hat \chi] \not \in \bns{\raag \Gamma}$, and therefore, the restriction $\chi=\hat{\chi}\vert_{\bbg \Gamma}\colon\bbg \Gamma\to\zz$ is not in $\bns{\bbg \Gamma}$ by Proposition~\ref{character in the BNS of BBG iff all the extensions are in the BNS of RAAG}.
However, the living edge subgraph $\livingedge \chi$ is connected and dominating.
On the other hand, note that $\dead{\hat \chi}$ contains a full subgraph that separates $\Gamma$ (compare with Lemma~\ref{lem:dead subgraph criterion for RAAGs}), and so does $\deadedge\chi$.
\end{example}

\begin{figure}[h]
    \centering
    \input{pictures/living_edge_subgraph_dead_edge_subgraphs_example}
    \caption{The living subgraph $\mathcal{L}(\hat{\chi})$ consists of two red vertices. It is neither connected nor dominating. The living edge subgraph $\livingedge \chi$ (labelled by $\pm1$) is connected and dominating.}
\label{fig:edge living subgraph is not enough to construct an extension character}
\end{figure}
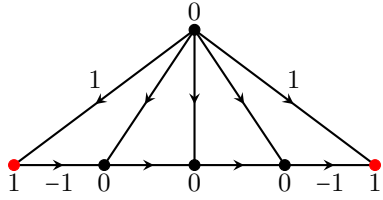

Our goal now is to show that the observations made in Example~\ref{ex:no good living edge subgraph criterion} about the dead edge subgraph hold in general.
We will need the following general topological facts that we collect for the convenience of the reader. 
Here and in the following, ``minimality'' is always with respect to the inclusion of subgraphs.
More precisely, a ``minimal full separating subgraph'' is a ``full separating subgraph whose full subgraphs are not separating''.

\begin{lemma}\label{lem:link connected}
Let $\Gamma$ be a biconnected graph with $\flag \Gamma$ simply connected.  
\begin{enumerate}
    \item \label{item:general MV} If $\Lambda\subseteq \Gamma$ is a connected full subgraph, then there is a bijection between the components of its complement and the components of its link.

    \item \label{item:link of a vertex is connected} The link of every vertex is connected.

    \item \label{item:minimal separating is connected} If $\Lambda \subseteq \Gamma$ is a minimal  full separating subgraph, then $\Lambda$ is connected and not a single vertex.
    
    \item \label{item:dimension at least 2} If $|\vv \Gamma| \geq 3$, then every edge is contained in at least one triangle. In particular, we have $\dim \flag \Gamma \geq 2$. 
    
\end{enumerate}
\end{lemma}
\begin{proof}

Proof of \eqref{item:general MV}.
Let $A=\flag \Gamma\setminus \flag \Lambda$ (set-theoretic difference) and $B=\st{\flag \Lambda,\flag \Gamma}$. 
Then $\flag \Gamma = A\cup B$, and $\lk{\flag \Lambda,\flag \Gamma}$ deformation retracts to $A\cap B$. 
The Mayer--Vietoris sequence for reduced homology associated to this decomposition of $\flag \Gamma$ provides the following exact sequence:
$$
\cdots
\to
H_1(\flag \Gamma) 
\to 
\widetilde H_0(\lk{\flag \Lambda,\flag \Gamma}) 
\to 
\widetilde H_0(A) \oplus \widetilde H_0(B)
\to
\widetilde H_0(\flag \Gamma)
\to 
0.
$$
We have $H_1(\flag \Gamma)=0=\widetilde H_0(\flag \Gamma)$ since $\flag \Gamma$ is simply connected.
Moreover, since $\Lambda$ is connected, the subcomplex $B$ is connected, and therefore, we obtain $\widetilde H_0(B) = 0$.
This gives a bijection between $\widetilde H_0(\lk{\flag \Lambda,\flag \Gamma})$ and $\widetilde H_0(A)$, as desired.

Proof of \eqref{item:link of a vertex is connected}.
Take $\Lambda=v$ to be a single vertex.
Since $\Gamma$ is biconnected, the vertex $v$ is not a cut vertex, so its complement is connected.
Then the conclusion follows from \eqref{item:general MV}.

Proof of \eqref{item:minimal separating is connected}. 
Let $\Lambda$ be a minimal full separating subgraph.
Then we can find two subcomplexes $A$ and $B$ of $\flag \Gamma$ such that $A\cup B = \flag \Gamma$ and $A\cap B=\flag \Lambda$. 
The Mayer--Vietoris sequence for reduced homology gives
$$
\cdots
\to
H_1(\flag \Gamma) 
\to 
\widetilde H_0(\flag \Lambda) 
\to 
\widetilde H_0(A) \oplus \widetilde H_0(B)
\to
\widetilde H_0(\flag \Gamma)
\to 
0.
$$
Arguing as in \eqref{item:general MV}, we have $H_1(\flag \Gamma)=0=\widetilde H_0(\flag \Gamma)$ since $\flag \Gamma$ is simply connected.
Therefore, we obtain $\widetilde H_0(\flag \Lambda) = \widetilde H_0(A) \oplus \widetilde H_0(B)$.
Suppose by contradiction that $\Lambda$ is disconnected. 
Then at least one of $A$ or $B$ is disconnected. 
Without loss of generality, say $A=A_1\cup A'$, with $A_1$ a connected component of $A$ and $A_1\cap A' = \varnothing$. 
Let $B'=B\cup A'$ and let $\Lambda'$ be the subgraph such that $\flag {\Lambda '} = A_1\cap B'$. Then $\Lambda'$ is a proper full subgraph of $\Lambda$ which separates $\Gamma$, contradicting the minimality of $\Lambda$.

Finally, if by contradiction $\Lambda$ were a single vertex, then it would be a cut vertex. 
But this is impossible because $\Gamma$ is biconnected.

Proof of \eqref{item:dimension at least 2}. Suppose by contradiction that there is an edge $e=(u,v)$ in $\flag \Gamma$ that is not contained in a triangle.
Since $\Gamma$ has at least three vertices, at least one endpoint of $e$, say $v$, is adjacent to at least another vertex different from $u$.
Since $e$ is not contained in a triangle, the vertex $u$ is an isolated component of $\lk{v,\Gamma}$.
Therefore, the subgraph $\lk{v,\Gamma}$ has at least two components, and hence, the vertex $v$ is a cut vertex of $\Gamma$ by \eqref{item:general MV}. 
This contradicts the fact that $\Gamma$ is biconnected.
\end{proof}

We now give a graphical criterion for a character to belong to the BNS-invariant of a BBG that is analogous to the living subgraph criterion in \cite{meierthebierineumannstrebelinvariantsforgraphgroups}, or rather to the dead subgraph criterion Lemma~\ref{lem:dead subgraph criterion for RAAGs} (see also \cite[Corollary 3.4]{lorenzo}).

\begin{maintheoremc}{C}[Graphical criterion for the BNS-invariant of a BBG]\label{thm:graphical criterion for bns of bbg}
Let $\Gamma$ be a biconnected graph with $\flag \Gamma$ simply connected. 
Let $\chi\in\mathrm{Hom}(\bbg \Gamma,\rr)$ be a non-zero character. Then $[\chi]\in \bns{\bbg \Gamma}$ if and only if $\deadedge \chi$ does not contain a full subgraph that separates $\Gamma$.
\end{maintheoremc}

\begin{proof}
Let $[\chi]\in \bns{\bbg \Gamma}$. Suppose by contradiction that there is a full subgraph $\Lambda\subseteq\deadedge \chi$ that separates $\Gamma$. 
Up to passing to a subgraph, we can assume that $\Lambda$ is a minimal full separating subgraph. 
So, by \eqref{item:minimal separating is connected} in Lemma~\ref{lem:link connected}, we can assume that $\Lambda$ is connected.
By Lemma \ref{a subgraph is edge dead iff there is an extension of characters on BBG to RAAG}, there is an extension $\hat{\chi} \in \mathrm{Hom}(\raag \Gamma,\rr)$ of $\chi$ such that $\Lambda\subseteq\mathcal{D}(\hat{\chi})$. 
Since $\Lambda$ separates $\Gamma$, by Lemma~\ref{lem:dead subgraph criterion for RAAGs}, we have $[\hat{\chi}]\notin\bns{\raag \Gamma}$, and therefore $[\chi]\notin \bns{\bbg \Gamma}$ by Proposition~\ref{character in the BNS of BBG iff all the extensions are in the BNS of RAAG}. Hence, we reach a contradiction. 

Conversely, assume $[\chi]\notin\bns{\bbg \Gamma}$.
Then by Proposition~\ref{character in the BNS of BBG iff all the extensions are in the BNS of RAAG}, there is an extension  $\hat{\chi}\in\mathrm{Hom}(\raag \Gamma,\rr)$ of $\chi$ such that $[\hat{\chi}]\notin\bns{\raag \Gamma}$. 
So, the living subgraph $\living{\hat{\chi}}$ is either disconnected or not dominating. 
Equivalently, by Lemma~\ref{lem:dead subgraph criterion for RAAGs}, the dead subgraph $\dead{\hat{\chi}}$ contains a full subgraph $\Lambda$ which separates $\Gamma$.  
Note that every edge of $\Lambda$ is contained in $\deadedge \chi$ because $\Lambda \subseteq \dead{\hat \chi}$.
A priori, the subgraph $\Lambda$ could have some components consisting of isolated points.
Once again, passing to a subgraph, we can assume that $\Lambda$ is a minimal full separating subgraph. By
\eqref{item:minimal separating is connected} in Lemma~\ref{lem:link connected}, we can also assume that $\Lambda$ is connected and not reduced to a single vertex.
Therefore, we have $\Lambda \subseteq\deadedge \chi$. This completes the proof.
\end{proof}

We give two examples to illustrate that the hypotheses of Theorem~\hyperref[thm:graphical criterion for bns of bbg]{C} are optimal.
Here, characters are represented by labels in the sense of \S\ref{sec:coordinates}.

\begin{example}[Simple connectedness is needed]\label{ex:simple connectedness is necessary}
Consider the cycle of length four $\Gamma=C_4$; see 
 the left-hand side of Figure~\ref{fig:ex_bns_bbg_conditions_needed}.
Then $\flag \Gamma$ is not simply connected. Note that in this case, the group $\bbg \Gamma$ is finitely generated but not finitely presented; see \cite{bestvinabradymorsetheoryandfinitenesspropertiesofgroups}.
Let $\hat\chi\in\mathrm{Hom}(\raag\Gamma,\rr)$  be the character of $\raag \Gamma$ that sends two non-adjacent vertices to $0$ and the other two vertices to $1$.
Let $\chi=\hat\chi\vert_{\bbg \Gamma}\in\mathrm{Hom}(\bbg \Gamma,\rr)$  be the restriction of $\hat\chi$ to $\bbg \Gamma$. 
Then the dead edge subgraph $\deadedge\chi$ is empty. In particular, it does not contain any subgraph that separates $\Gamma$. 
However, the living subgraph $\living{\hat \chi}$ consists of two opposite vertices, which is not connected. Thus, we have $[\hat\chi]\not \in \bns{\raag\Gamma}$. Hence, by Proposition~\ref{character in the BNS of BBG iff all the extensions are in the BNS of RAAG}, we obtain $[\chi] \not \in \bns{\bbg\Gamma}$.
\end{example}

\begin{figure}[h]
    \centering
    \input{pictures/ex_bns_bbg_square}
    \caption{Theorem~\hyperref[thm:graphical criterion for bns of bbg]{C} does not hold on a graph with $\flag \Gamma$ not simply connected (left), nor  on a graph with a cut vertex (right).}
    \label{fig:ex_bns_bbg_conditions_needed}
\end{figure}
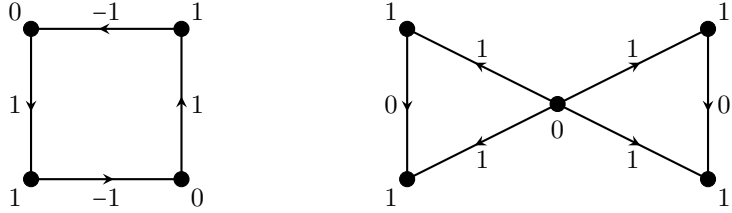

\begin{example}[Biconnectedness is needed]\label{ex:biconnectedness is necessary}
Let $\Gamma$ be the graph obtained by gluing two triangles at a vertex; see the right-hand side  of Figure~\ref{fig:ex_bns_bbg_conditions_needed}.
Let $\hat\chi\in\mathrm{Hom}(\raag\Gamma,\rr)$ be the character that sends the cut vertex to $0$ and all the other vertices to $1$.
Let $\chi=\hat\chi\vert_{\bbg \Gamma}\in\mathrm{Hom}(\bbg \Gamma,\rr)$  be the restriction of $\hat{\chi}$ to $\bbg \Gamma$. 
Then the dead edge subgraph $\deadedge\chi$ consists of the two edges that are not incident to the cut vertex. In particular, it does not contain any subgraph that separates $\Gamma$. 
However, the living subgraph $\living{\hat \chi}$ is not connected (also notice that $\living{\hat \chi}=\deadedge \chi$). Thus, we have $[\hat\chi]\not \in \bns{\raag\Gamma}$. Hence, Proposition~\ref{character in the BNS of BBG iff all the extensions are in the BNS of RAAG} implies $[\chi] \not \in \bns{\bbg\Gamma}$.
\end{example}

As mentioned in Example~\ref{ex:cut vertex implies empty bns}, the graph $\Gamma$ in Example~\ref{ex:biconnectedness is necessary} has a cut vertex, and hence, the BNS-invariant $\bns{\bbg\Gamma}$ is empty; see \cite[Corollary 15.10]{PapadimaandSuciuBNSRinvariantsandHomologyJumpingLoci}. 
As promised, we now show the following result.

\begin{corollary}\label{cor:bns for bbg non empty}
Let $\Gamma$ be a biconnected   graph with $\flag \Gamma$ simply connected. Then $\bns{\bbg \Gamma} \neq \varnothing$.
\end{corollary}
\begin{proof}
Let $T$ be a spanning tree of $\Gamma$.
Assign an orientation to each edge of $T$ and write $\ee T = \lbrace e_1,\dots, e_m \rbrace$. Let
$$ \chi \colon \ee T \to \rr, \ \ \chi (e_k)=10^k, \ \ k=1,\dots, m.$$
Then this defines a character thanks to Lemma~\ref{lem:characters of f.p BBGs are given by assigning values on spanning trees}.
We claim that $\chi$ does not vanish on any edge of $\Gamma$. Indeed, let $e\in \ee \Gamma$. The claim is clear if $e \in \ee T$. Suppose $e\not \in \ee T$, and let $(e_{i_1},\dots, e_{i_p})$ be the path in $T$ between the endpoints of $e$. Then $(e,e_{i_1},\dots, e_{i_p})$ is a cycle in $\Gamma$, and hence, the element $ee_{i_1}\dots e_{i_p}$ is a relator in $\bbg \Gamma$ by Theorem~\ref{thm:DL presentation embedding}. Therefore, we have
$$ 0 = \chi (e) + \chi (e_{i_1}) + \dots + \chi (e_{i_p}) = \chi (e) \pm 10^{k_{i_1}} \pm \dots \pm 10^{k_{i_p}},$$
where the signs are determined by the orientations of the corresponding edges. 
The sum $\pm 10^{k_{i_1}} \pm \dots \pm 10^{k_{i_p}}$ is never zero since all the exponents are different. 
Thus, we have $\chi(e)\neq0$. This proves the claim.
It immediately follows that $\deadedge \chi =\varnothing$, and therefore, we have $[\chi] \in \bns{\bbg \Gamma}$ by Theorem~\hyperref[thm:graphical criterion for bns of bbg]{C}.
\end{proof}

As a summary, we have the following corollary. 
Most implications are well-known. 
Our contribution is that \eqref{item:biconnected} implies \eqref{item:non empty bns}.
Recall that a finitely generated group $G$ \textit{algebraically fibers} if there is a surjective homomorphism $G\to\zz$ whose kernel is finitely generated.

\begin{corollary}\label{cor:equivalent biconnected}
Let $\Gamma$ be a connected graph such that $\flag \Gamma$ is simply connected. Then the following statements are equivalent.
\begin{enumerate}
    \item \label{item:biconnected} $\Gamma$ is biconnected.
    \item \label{item:non empty bns} $\bns{\bbg \Gamma} \neq \varnothing$.
    \item \label{item:free split} $\bbg \Gamma$  does not split as a free product. 
    \item \label{item:one ended} $\bbg \Gamma$ is $1$-ended.
    \item \label{item:alg fibration} $\bbg \Gamma$ algebraically fibers.
\end{enumerate}
\end{corollary}
\begin{proof}
The equivalence of \eqref{item:biconnected} and \eqref{item:non empty bns} is given by Corollary~\ref{cor:bns for bbg non empty} and \cite[Corollary 15.10]{PapadimaandSuciuBNSRinvariantsandHomologyJumpingLoci}.
Given that $\bbg \Gamma$ is torsion-free, the equivalence of  \eqref{item:free split} and \eqref{item:one ended}  is just Stalling's theorem about the ends of groups (see \cite[Theorem I.8.32]{BH99}).
The fact that
\eqref{item:non empty bns} implies
\eqref{item:free split} 
is discussed in \cite[Example 3 in A2.1a]{strebelnotesonthesigmainvariants}.
The fact that \eqref{item:free split} implies \eqref{item:biconnected} can be seen directly from the Dicks--Leary presentation from Theorem~\ref{thm:DL presentation embedding}.

Finally, we show that \eqref{item:alg fibration} is equivalent to \eqref{item:non empty bns}. 
It follows from Theorem~\ref{thm:bns fg} that $\bbg \Gamma$ algebraically fibers if and only if there exists a discrete character $\chi:\bbg \Gamma\to \rr$ such that both $[\chi]$ and $[-\chi]$ are in $\bns{\bbg \Gamma}$. 
This is actually equivalent to just requiring that $[\chi] \in \bns{\bbg \Gamma}$, because $\bns{\bbg \Gamma}$ is symmetric; see Lemma~\ref{graph orientation reversing map is an isomorphism on BBG}.
Note that the points of the character sphere $\chars{\bbg{\Gamma}}$ given by the equivalence classes of discrete characters are exactly the rational points.
In particular, since $\bns{\bbg \Gamma}$ is an open subset of the character sphere $\chars{\bbg{\Gamma}}$ (see Theorem A in \cite{bierineumannstrebelageometricinvariantofdiscretegroups}), it is non-empty if and only if it contains the equivalence class of a discrete character.
\end{proof}

We record the following consequence for future reference. It will reduce our discussion about the RAAG recognition problem to the case of biconnected graphs.

\begin{corollary}\label{cor:biconnected components}
Let $\Gamma$ be a connected graph with $\flag \Gamma$ simply connected, and let $\Gamma_1,\dots, \Gamma_n$ be its biconnected components.
Then $\bbg \Gamma$ is  a RAAG if and only if $\bbg{\Gamma_i}$ is  a RAAG for all $i=1,\dots,n$.
\end{corollary}
\begin{proof}
It is clear from the Dicks--Leary presentation that $\bbg \Gamma$ is the free product of the $\bbg{\Gamma_i}$.
Moreover, since $\Gamma_i$ is biconnected, each $\bbg{\Gamma_i}$ is freely indecomposable (see Corollary~\ref{cor:equivalent biconnected}).
If all the $\bbg{\Gamma_i}$ are RAAGs, then $\bbg \Gamma$ is a RAAG because the free product of RAAGs is a RAAG. This proves one implication.
For the converse implication, suppose  that $\bbg \Gamma$ is a RAAG, say $\bbg \Gamma=\raag \Lambda$ for some graph $\Lambda$.
Let $\Lambda_1,\dots,\Lambda_m$ be the connected components of $\Gamma$.
Then $\bbg \Gamma=\raag \Lambda$ can also be written as the free product of the RAAGs $\raag{\Lambda_j}$, each of which is freely indecomposable.
It follows that $m=n$, and for each $i$, there is some $j$ such that $\bbg{\Gamma_i} \cong \raag{\Lambda_j}$.
\end{proof}

\begin{remark}
We conclude this subsection by observing that when $\Gamma$ is a chordal graph, the statement in Theorem~\hyperref[thm:graphical criterion for bns of bbg]{C} can also be obtained as follows. 
By \cite[\S 3.2]{lorenzo}, the group $\bbg \Gamma$ splits as a finite graph of groups.
More precisely, the vertex groups correspond to the BBGs on the maximal cliques of $\Gamma$, and the edge groups correspond to BBGs on the minimal separating subgraphs of $\Gamma$ (that are also cliques because $\Gamma$ is chordal).
In particular, all these groups are finitely generated free abelian groups. Hence, one can apply the results from \S 2 of \cite{CL16}.
\end{remark}


\subsection{A graphical description of \texorpdfstring{$\bns{\bbg\Gamma}$}{the BNS-invariant of a BBG}}\label{sec:graphical description}
We now provide a graphical description of $\bns{\bbg \Gamma}$, that is, a way to compute the BNS-invariant of $\bbg \Gamma$ in terms of subgraphs of $\Gamma$.

Recall   from Remark~\ref{rem:complement BNS for RAAGs} that $\bnsc{\raag \Gamma}$ is given by an arrangement of missing subspheres parametrized by the separating subgraphs of $\Gamma$.
Thanks to \cite[Corollary 1.4]{kochloukovamendonontheBNSRsigmainvariantsoftheBBGs}, we know that
$\bnsc{\bbg \Gamma}$ is also an arrangement of missing subspheres.
Moreover, the restriction map $ r\colon \operatorname{Hom}(\raag \Gamma,\rr) \to \operatorname{Hom}(\bbg \Gamma,\rr)$  sends the missing subspheres of $\bnsc{\raag \Gamma}$ to those of $\bnsc{\bbg \Gamma}$ (see the discussion after Lemma~\ref{lem:existence extension}).
So, it makes sense to look for a description of the missing subspheres of $\bnsc{\bbg \Gamma}$ in terms of subgraphs of $\Gamma$, analogous to the one available for $\bnsc{\raag \Gamma}$.

However, recall from Example~\ref{ex:iso bbg non iso graphs} that $\bbg \Gamma$ does not completely determine $\Gamma$, so it is a priori not clear that $\bnsc{\bbg \Gamma}$ should admit such a description.
Moreover, the restriction map is not always well-behaved with respect to the vanishing behavior of characters, in the sense that the dead edge subgraph of a character can be strictly larger than the dead subgraph of any of its extensions; see Figure~\ref{fig:not full}. 
To address this, we need a way to construct characters with prescribed vanishing behavior.

For any subgraph $\Lambda$ of $\Gamma$, we define the following linear subspace of $\operatorname{Hom}(\bbg \Gamma, \rr)$
$$ W_\Lambda = \{ \chi \colon \bbg\Gamma \to \rr \mid \chi(e)=0, \, \forall e\in \ee \Lambda \} = \{ \chi \colon \bbg\Gamma \to \rr \mid  \Lambda \subseteq \deadedge \chi \}$$
and the great subsphere $S_\Lambda$ given by the following intersection
$$ S_\Lambda = W_\Lambda \cap \chars{\bbg \Gamma}.$$

Note that if a character $\chi$ of $\bbg \Gamma$ vanishes on a spanning tree of $\Gamma$, then $\chi$ is trivial (see Lemma~\ref{lem:characters of f.p BBGs are given by assigning values on spanning trees}). 
In other words, if $\Lambda$ is a spanning tree, then $W_\Lambda =0$ and $S_\Lambda=\varnothing$. 
We look for a condition on $\Lambda$ such that $W_\Lambda\neq 0$ and $S_\Lambda\neq \varnothing$.
Notice that the following lemma applies as soon as $\vv \Lambda \neq \vv \Gamma$, and that if it applies to $\Lambda$, then it also applies to all of its subgraphs.

\begin{lemma}\label{lem:non-spanning implies critical}
Let $\Gamma$ be a graph with $\flag \Gamma$ simply connected, and let $\Lambda\subseteq \Gamma$ be a subgraph.
Assume that there is an edge $e_0 \in \ee \Gamma$ with at least one endpoint not in $\vv \Lambda$.
Then there exists a character $\chi\colon \bbg \Gamma \to \rr$ such that $\chi(e_0)=1$ and $\chi(e)=0$ for all $e \in \ee \Lambda$. 
In particular, we have $[\chi]\in S_\Lambda$.
\end{lemma}

\begin{proof}
Let $T_\Lambda$ be a spanning forest of $\Lambda$, that is, a subgraph of $\Lambda$ that is a disjoint union of trees and contains all vertices of $\Lambda$ (note that we are not assuming that $\Lambda$ is connected).
Observe $e_0\not \in \ee \Lambda$ by assumption.
Therefore, we can extend $T_\Lambda \cup \{e_0\}$ to a spanning tree $T$ of $\Gamma$.
Orient the edges of $T$ arbitrarily and label the edges of $T_\Lambda$ by $0$ and all the remaining edges of $T$ by $1$.
By Lemma~\ref{lem:characters of f.p BBGs are given by assigning values on spanning trees}, this defines a character $\chi\colon \bbg \Gamma \to \rr$.
By construction, we have $\chi(e_0)=1$ and $\chi (e)=0$ for $e\in \ee{T_\Lambda}$.
Let $e  \in \ee \Lambda \setminus \ee{T_\Lambda}$. Since $T_\Lambda$ is a spanning forest of $\Lambda$, there is a unique path $p$ in $T_\Lambda$ from $\tau e$ to $\iota e$. Then $ep$ is a cycle in $\Gamma$, and therefore, it is a relator in the Dicks--Leary presentation for $\bbg \Gamma$. Since $\chi$ vanishes on $p$, it must also vanish on $e$, as desired.
\end{proof}

\begin{remark}\label{rem:isolated vertices}
Notice that if two subgraphs $\Lambda$ and $\Lambda '$ have the same edge sets, then $W_\Lambda = W_{\Lambda '}$ because these subspaces only depend on the edge sets.
In particular, we have $S_\Lambda = S_{\Lambda '}$.
This is the reason why we use the strict inclusion $\subsetneq$ instead of the weak inclusion $\subseteq$ in the statement \eqref{item:inclusion reverse} of the following lemma.
\end{remark}

\begin{lemma}\label{lem: missing subspheres}
Let $\Gamma$ be a biconnected graph with $\flag \Gamma$ simply connected, and let $\Lambda$ and $\Lambda '$ be 
full separating subgraphs.
Then we have the following statements.
\begin{enumerate}
    \item \label{item:extensible is in complement} 
     $S_\Lambda$ is a missing subsphere, that is, we have $S_\Lambda \subseteq \bnsc{\bbg \Gamma}$.
    
    \item \label{item:inclusion reverse}   $\Lambda ' \subsetneq \Lambda$ if and only if  $S_{\Lambda} \subsetneq S_{\Lambda '}$.
\end{enumerate}
\end{lemma}

\begin{proof}
Proof of \eqref{item:extensible is in complement}. If $[\chi] \in S_\Lambda$, then $\deadedge \chi$ contains $\Lambda$, which is a separating subgraph.
Then the statement follows from Theorem~\hyperref[thm:graphical criterion for bns of bbg]{C}.

Proof of \eqref{item:inclusion reverse}. The implication $\Lambda '\subsetneq \Lambda \Rightarrow S_{\Lambda} \subsetneq S_{\Lambda '}$ follows from the definitions.
For the reverse implication $  S_{\Lambda} \subsetneq S_{\Lambda '} \Rightarrow \Lambda '\subsetneq \Lambda$ we argue as follows.
The inclusion $S_\Lambda \subsetneq S_{\Lambda '}$ implies that a character vanishing on $\Lambda$ must also vanish on $\Lambda'$. We need to show that $\Lambda'$ is a proper subgraph of $\Lambda$.
    
By contradiction, suppose that $\Lambda '$ is not a subgraph of $\Lambda$. 
Notice that if $\Lambda ' \setminus \Lambda$ consists of isolated vertices, then $S_\Lambda = S_{\Lambda '}$ (see Remark~\ref{rem:isolated vertices}).
So, we can assume that there is an edge $e_0 \in \ee{\Lambda'} \setminus \ee \Lambda$. 
Since $\Lambda$ is full, the edge $e_0$ cannot have both endpoints in $\Lambda$.
By Lemma~\ref{lem:non-spanning implies critical}, there is a character  $\chi\colon \bbg \Gamma \to \rr$ with $\chi (e_0)=1$ and $\chi(e)=0 $ for all $e\in \ee \Lambda$.
This is a character that vanishes identically on $\Lambda$ but not on $\Lambda'$, which is absurd.
\end{proof}

Recall that if $\Lambda$ is a separating subgraph, then $S_\Lambda\neq \varnothing$.

\begin{maintheoremc}{D}[Graphical description of the BNS-invariant of a BBG]\label{thm:graphical description for bns of bbg}
Let $\Gamma$ be a biconnected   graph with $\flag \Gamma$ simply connected.
Then $\bnsc{\bbg \Gamma}$ is a union of missing subspheres corresponding to  full separating subgraphs. More precisely,
\begin{enumerate}
    
    \item \label{item:complement is union minimal spheres}  $\bnsc{\bbg \Gamma}= \bigcup_\Lambda S_\Lambda$, where $\Lambda$ ranges over the minimal  full separating subgraphs of $\Gamma$.

    \item \label{item:correspondence spheres and separators} There is a bijection between maximal missing subspheres of $\bnsc{\bbg \Gamma}$ and minimal full separating subgraphs of $\Gamma$.
    
\end{enumerate}
\end{maintheoremc}

\begin{proof}
Proof of \eqref{item:complement is union minimal spheres}. We start by proving that $\bnsc{\bbg \Gamma}= \bigcup_\Lambda S_\Lambda$, where $\Lambda$ ranges over the  full separating subgraphs of $\Gamma$.
If $\Lambda$ is a full separating subgraph, then we know that $S_\Lambda \subseteq \bnsc{\bbg\Gamma}$ by \eqref{item:extensible is in complement} in Lemma~\ref{lem: missing subspheres}. 
So one inclusion is clear. 
Vice versa, let $[\chi]\in \bnsc{\bbg \Gamma}$. 
Then by Theorem~\hyperref[thm:graphical criterion for bns of bbg]{C} we have that $\deadedge \chi$ contains a full separating subgraph $\Lambda$. 
In particular, the character $\chi$ vanishes on $\Lambda$, hence $ [\chi] \in S_\Lambda$. This proves the other inclusion.
To see that one can restrict to $\Lambda$ ranging over minimal full separating subgraphs, just observe that the latter correspond to maximal missing subspheres by \eqref{item:inclusion reverse} in Lemma~\ref{lem: missing subspheres}.
This completes the proof of \eqref{item:complement is union minimal spheres}.

Proof of \eqref{item:correspondence spheres and separators}.
By \eqref{item:complement is union minimal spheres}, we know that $\bnsc{\bbg \Gamma}$ is a union of maximal missing subspheres. 
Notice that this is a finite union because $\Gamma$ has only finitely many subgraphs. 
So, each maximal missing subsphere $S$ is of the form $S=S_\Lambda$ for $\Lambda$ a minimal  full separating subgraph.
    
Vice versa, let $\Lambda$ be a minimal  full separating subgraph. 
We know from \eqref{item:extensible is in complement} in Lemma~\ref{lem: missing subspheres} that $S_\Lambda$ is a missing subsphere.
We claim that $S_\Lambda$ is a maximal missing subsphere in  $\bnsc{\bbg \Gamma}$.
Let $S$ be a maximal missing subsphere in  $\bnsc{\bbg \Gamma}$ such that $S_\Lambda \subseteq S$. 
By the previous paragraph, we know that $S=S_{\Lambda '}$ for some  minimal  full separating subgraph $\Lambda '$.
If we had $S_\Lambda \subsetneq S= S_{\Lambda'}$, then by \eqref{item:inclusion reverse} in Lemma~\ref{lem: missing subspheres} it would follow that $\Lambda' \subsetneq \Lambda$.
But this would contradict the minimality of $\Lambda$.
Thus, we have $S_\Lambda=S_{\Lambda '}=S$. Hence, the missing subsphere $S_\Lambda$ is maximal.
\end{proof}

The following example establishes a correspondence between the cut edges in $\Gamma$ and the missing hyperspheres  (the missing subspheres of codimension one) in $\bnsc{\bbg \Gamma}$.
It should be compared with the case of RAAGs, where the correspondence is between the cut vertices of $\Gamma$ and the missing hyperspheres in $\bnsc{\raag \Gamma}$ (compare Remark~\ref{rem:complement BNS for RAAGs} and Example~\ref{ex: bns of raag on tree}).

\begin{example}[Hyperspheres]\label{ex:bijection between cut edges and hyperplanes for BBGs}
Let $\Gamma$ be a biconnected graph with $\flag \Gamma$ simply connected.
Let $e$ be a cut edge of $\Gamma$.
Notice that $e$ is a minimal separating subgraph since $\Gamma$ is biconnected, and it is also clearly  full.
So by Theorem~\hyperref[thm:graphical description for bns of bbg]{D} we know that $S_e$ is a maximal missing subsphere in $\bnsc{\bbg \Gamma}$.
We want to show that the subspace $W_e=\operatorname{span}(S_e)$ is a hyperplane.
To see this, let $T$ be a spanning tree of $\Gamma$ with $\ee T=\lbrace e_{1},\dots,e_{m}\rbrace$, and let $y_i$ be the coordinate dual to $e_i$ in the sense of \S\ref{sec:coordinates}. 
This means that $y_i(\chi)=\chi(e_i)$ for all $\chi\in \operatorname{Hom}(\bbg \Gamma, \rr)$.
Note that $W_{e_i}$ is the hyperplane given by the equation $y_{i}=0$.
If $e\in \ee T$, then $e=e_i$ for some $i=1,\dots,m$ and $W_e=W_{e_i}$ is a hyperplane.
If $e\not \in \ee T$, then there is a unique path $(e_{j_1},\dots,e_{j_p})$  in $T$ connecting the endpoints of $e$. 
Since $(e_{j_1},\dots,e_{j_p},e)$ is a cycle in $\Gamma$, the word $e_{j_1}\dots e_{j_p}e$ is a relator in the Dicks--Leary presentation. 
So, we have $\chi (e_{j_1}) + \dots +\chi (e_{j_p}) +  \chi (e) =0$. Therefore, we obtain $\chi(e)=0$ if and only if $y_{j_{1}}(\chi) + \dots + y_{j_{p}}(\chi)=0$. 
This means that $W_e$ is the hyperplane defined by the equation $y_{j_{1}} + \dots + y_{j_{p}}=0$.

Vice versa, let $S\subseteq \bnsc{\bbg \Gamma}$ be a hypersphere. 
We claim that $S=S_e$ for some cut edge $e$.
To see this, let $[\chi]\in S$.
By Theorem~\hyperref[thm:graphical criterion for bns of bbg]{C} we know that $\deadedge \chi$  contains a full subgraph $\Lambda$ that separates $\Gamma$. 
Since $\Gamma$ is biconnected, the subgraph $\Lambda$ must contain at least one edge.
In particular, the character $\chi$ vanishes on $\ee \Lambda$, and therefore, we have $[\chi] \in \bigcap_{e\in \ee \Lambda} S_e$.
This proves $S\subseteq \bigcap_{e\in \ee \Lambda} S_e$.
However, by the discussion above, we know that $S_e$ is a hypersphere. Since $S$ is also a hypersphere, the subgraph $\Lambda$ must consist of a single edge $e$ only.
In particular, it is a cut edge.
\end{example}

\begin{remark}\label{rem:general position arrangements}
The linear span of the arrangement of the missing subspheres of $\bnsc{G}$ gives rise to a subspace arrangement in $\operatorname{Hom}(G,\rr)$.
The main difference between RAAGs and BBGs is that the arrangement for a RAAG is always  ``in general position'' while the arrangement for a BBG is not. We will discuss the details in the next section.
\end{remark}


\subsection{The inclusion-exclusion principle}\label{sec:IEP}

Given a group $G$, one can consider the collection of  maximal missing subspheres. That is, the maximal great subspheres of the character sphere $\chars G$ that are in the complement of the BNS-invariant $\bns G$ (see Remark~\ref{rmk:missing subspheres}).
Additionally, one can also consider the collection of maximal missing subspaces in $\operatorname{Hom}(G,\rr)$, that is, the linear spans of the maximal missing subspheres.
This provides an arrangement of (great) subspheres in $\chars G$ and an arrangement of (linear) subspaces in $\operatorname{Hom}(G,\rr)$ that can be used as invariants for $G$.
For instance, these arrangements completely determine the BNS-invariant when $G$ is a RAAG or a BBG (see Remark~\ref{rem:complement BNS for RAAGs} or  Theorem~\hyperref[thm:graphical description for bns of bbg]{D} respectively).
Moreover, in the case of RAAGs, these arrangements satisfy a certain form of the inclusion-exclusion principle (see \S\ref{sec:IEP RAAG behavior}).
This fact can be used to detect when a group $G$ is not a RAAG.
We take this point of view from the work of Koban and Piggott in \cite{KobanPiggottTheBNSofthepuresymmetricautomorphismofRAAG} and Day and Wade in \cite{DayWadeSubspaceArrangementBNSinvariantsandpuresymmetricOuterAutomorphismsofRAAGs}.
The former focuses on the subsphere arrangement, while the latter focuses on the subspace arrangement.
In this section, we find it convenient to focus on the subspace arrangement.

Let $V$ be a real vector space. (The reader should think $V=\operatorname{Hom}(G,\rr)$ for a group $G$.)
For convenience, we fix some background inner product on $V$. All arguments in the following are combinatorial and do not depend on the choice of inner product.
We say that a finite collection of linear subspaces $\lbrace W_{j}\rbrace_{j\in J}$ of $V$ satisfies the \textit{inclusion-exclusion principle} if the following equality holds:

\begin{equation}\label{eq:IEP subspaces}
\dim{\left(\sum^{|J|}_{j=1} W_j\right)}
=
\sum^{|J|}_{k=1}(-1)^{k+1}\left(\sum_{I\subset J, |I|=k}\dim\left(\bigcap_{j\in I}W_j\right)\right)
\end{equation}

Notice that if an arrangement satisfies \eqref{eq:IEP subspaces}, then any linearly equivalent arrangement also satisfies \eqref{eq:IEP subspaces}.
Here are two examples. The first is a RAAG, and the collection of maximal subspaces in the complement of its BNS-invariant satisfies the inclusion-exclusion principle. The second is a BBG, and the collection of maximal subspaces in the complement of its BNS-invariant does not satisfy the inclusion-exclusion principle. 
Note that this BBG is known to be not isomorphic to any RAAG by \cite{PapadimaSuciuAlgebraicinvariantsforBBGs}.

\begin{example}[Trees]\label{ex: bns of raag on tree}
Let $\Gamma$ be a tree on $n$ vertices, and let $\lbrace v_{1},\dots,v_{m}\rbrace$ be the set of cut vertices of $\Gamma$.  
Then it follows that $\bns{\raag \Gamma}$ is obtained from  $\chars{\raag \Gamma}=S^n$ by removing the hyperspheres $S_i$ defined by $ x_i=0$ for $i=1,\dots, m$ (see \S\ref{sec:coordinates}). 
The associated missing subspaces satisfy the inclusion-exclusion principle \eqref{eq:IEP subspaces}.
\end{example}


\begin{example}[The trefoil]\label{ex: bns of PS}
Let $\Gamma$ be the (oriented) trefoil graph with a choice of a spanning tree $T$ whose edge set is $\ee T=\lbrace e_{1},e_{2},e_{3},e_{4},e_{5}\rbrace$; see Figure~\ref{fig: oriented trefoil graph with a spanning tree}.
We consider the three cut edges $e_1$, $e_2$, and $f$.
By Example~\ref{ex:bijection between cut edges and hyperplanes for BBGs}, we have that $S_{e_{1}}$, $S_{e_{2}}$, and $S_{f}$ are missing hyperspheres in $\bnsc{\bbg \Gamma}$.
By Theorem~\hyperref[thm:graphical description for bns of bbg]{D}, we have $\bnsc{\bbg \Gamma}=S_{e_{1}} \cup S_{e_{2}} \cup S_{f}$.
If $y_1,\dots, y_5$ are the dual coordinates on $\mathrm{Hom}(\bbg\Gamma,\rr)\cong\rr^{5}$ with respect to $T$ (in the sense of \S\ref{sec:coordinates}), then $S_{e_{1}}$, $S_{e_{2}}$, and $S_{f}$ are given by $y_{1}=0$, $y_{2}=0$, and $y_{1}-y_{2}=0$, respectively.
To see the latter, first note that we have a relator $e_{1}f=e_{2}=fe_{1}$ in the Dicks--Leary presentation. 
Then for any character $\chi\in\mathrm{Hom}(\bbg\Gamma,\rr)$, we have $\chi(e_{1})+\chi(f)=\chi(e_{2})$. 
Thus, we obtain $\chi(f)=0$ if and only if $\chi(e_{1})=\chi(e_{2})$. 
Therefore, the hypersphere $S_{f}$ is defined by $y_{1}=y_{2}$, that is, the equation $y_{1}-y_{2}=0$.
A direct computation shows that the associated missing subspaces do not satisfy the inclusion-exclusion principle \eqref{eq:IEP subspaces}.
\end{example}

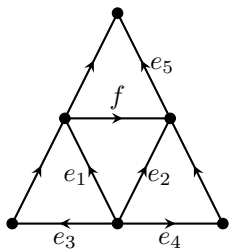
\begin{figure}[h]
    \centering
    \input{pictures/Example_BBG_cut_edges_and_hyperplanes}
    \caption{An oriented trefoil graph with a spanning tree.}
    \label{fig: oriented trefoil graph with a spanning tree}
\end{figure}

It is natural to ask whether the phenomenon from Example~\ref{ex: bns of PS} is actually a general obstruction for a BBG to be a RAAG.
In \cite{DayWadeSubspaceArrangementBNSinvariantsandpuresymmetricOuterAutomorphismsofRAAGs}, Day and Wade developed a homology theory $H_\ast(\mathcal{V})$ for a subspace arrangement $\mathcal V$ in a vector space that is designed to measure the failure of the inclusion-exclusion principle for $\mathcal V$.
They proved that if $G$ is a RAAG, then
$ H_k(\mathcal V_G) = 0$ for all $k>0$, where $\mathcal V_G$ denotes the arrangement of maximal subspaces corresponding to the maximal missing spheres in $\bnsc G$; see \cite[Theorem B]{DayWadeSubspaceArrangementBNSinvariantsandpuresymmetricOuterAutomorphismsofRAAGs}.

\clearpage
Given our description of the BNS-invariant for BBGs from \S\ref{sec:coordinates} and Theorem~\hyperref[thm:graphical description for bns of bbg]{D}, we can determine that certain BBGs are not RAAGs.
For example, a direct computation shows that the group $G=\bbg \Gamma$ from Example~\ref{ex: bns of PS} has  $H_1(\mathcal V_G)\neq 0$.
On the other hand, there are BBGs that cannot be distinguished from RAAGs in this way, as in the next example.

\begin{example}[The extended trefoil]\label{ex:extended trefoil}
Let $\Gamma$ be the trefoil graph with one extra triangle attached; see Figure~\ref{fig: A special but not extra-special triangulation}. 
Imitating Example~\ref{ex: bns of PS}, we choose a spanning tree $T$ whose edge set is $\ee T=\lbrace e_{1},e_{2},e_{3},e_{4},e_{5},e_6\rbrace$.
By Theorem~\hyperref[thm:graphical description for bns of bbg]{D}, 
we have $\bnsc{\bbg \Gamma}=S_{e_{1}} \cup S_{e_{2}} \cup S_{f}  \cup S_{e_{5}}$.
If $y_1,\dots, y_6$ are the dual coordinates on $\mathrm{Hom}(\bbg\Gamma,\rr)\cong\rr^{6}$ with respect to $T$ (in the sense of \S\ref{sec:coordinates}), then these missing hyperspheres are defined by the hyperplanes  given by $y_{1}=0$, $y_{2}=0$,  $y_{1}-y_{2}=0$, and $y_5=0$, respectively.
A direct computation shows that $H_k(\mathcal V_{\bbg \Gamma})=0$ for all $k\geq 0$, that is, these homology groups look like the homology groups for the arrangement associated to a RAAG.
However, we will show  that this BBG is not a RAAG in Example~\ref{ex:extended trefoil continued}.
\end{example}

\begin{figure}[h]
    \centering
    \input{pictures/ex_bbg_on_PS_with_an_extra_triangle}
    \caption{The extended trefoil: a new example of a BBG that is not a RAAG.}
    \label{fig: A special but not extra-special triangulation}
\end{figure}
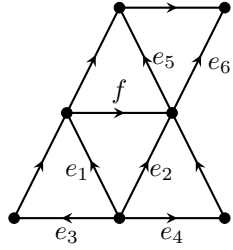


Our goal now is to obtain a criterion to detect when a BBG is not a RAAG that is still based on a certain failure of the inclusion-exclusion principle in the complement of the BNS-invariant.
The obstruction always involves only a collection of three subspaces, regardless of the complexity of the graph.
So, we find it convenient to introduce the following notation:

\begin{equation}\label{eq:def IEP}
\begin{split}
    \iep{W_1}{W_2}{W_3} = &  \dim W_1+ \dim W_2+ \dim W_3  \\
     - &  \dim (W_1\cap W_2) - \dim (W_1\cap W_3) - \dim (W_2\cap W_3) \\
    + &  \dim (W_1\cap W_2\cap W_3).
\end{split}
\end{equation}


\subsubsection{RAAG behavior}\label{sec:IEP RAAG behavior}

The following lemma states that the arrangement defining the BNS-invariant of any RAAG satisfies the inclusion-exclusion principle.
This is due to the fact that in this case, the missing subspaces are effectively described in terms of sets of vertices of $\Gamma$ and the inclusion-exclusion principle holds for subsets of a given set.
The argument follows the proof of \cite[Lemma 5.3]{KobanPiggottTheBNSofthepuresymmetricautomorphismofRAAG}. We include a proof for completeness.

\begin{lemma}\label{lem:koban piggott}
Let $\Gamma$ be a connected graph.
Let $\{W_j\}_{j\in J}$ be a collection of maximal missing subspaces in $\operatorname{Hom}(\raag \Gamma,\rr)$.
Then $\{W_j\}_{j\in J}$ satisfies \eqref{eq:IEP subspaces}.
In particular, when $J=\{1,2,3\}$ we have $\dim(W_1+W_2+W_3)=\iep{W_1}{W_2}{W_3}$.
\end{lemma}

\begin{proof}
Recall that the subspace $W_j$ corresponds to a minimal full separating subgraph $\Lambda_j$ of $\Gamma$ (see Remark~\ref{rem:complement BNS for RAAGs}).
Moreover, the dimension of $W_j$ is equal to the number of vertices in the complement $A_j=\Gamma \setminus \Lambda_j$ of $\Lambda_j$ (those vertices provide a basis for $W_j$, in the sense of \S\ref{sec:coordinates}.)
It follows that
\begin{align*}
\dim{\left(\sum^{|J|}_{j=1} W_j\right)}
&=\left| \bigcup_{j=1}^{|J|} \vv{A_j} \right| \\
&=\sum^{|J|}_{k=1}(-1)^{k+1} \left( \sum\limits_{\substack{I\subset J \\ |I|=k}} \left| \bigcap_{j\in I} \vv{A_j} \right| \right) \\
&=\sum^{|J|}_{k=1}(-1)^{k+1}\left(\sum\limits_{\substack{I\subset J \\ |I|=k}}\dim\left(\bigcap_{j\in I}W_j\right)\right).
\end{align*}
This means precisely that $\{W_j\}_{j\in J}$ satisfies \eqref{eq:IEP subspaces}, as desired.
\end{proof}


\subsubsection{Non RAAG behavior}\label{sec:IEP NON-RAAG behavior}
We now want to identify a condition that is not compatible with the property established in \S\ref{sec:IEP RAAG behavior} for the arrangement associated to a RAAG.
More precisely, we look for a sharp lower bound for the term $\iep{W_1}{W_2}{W_3}$.
The key condition is the one in Lemma~\ref{lem:LA general}. It is inspired by \cite{DayWadeSubspaceArrangementBNSinvariantsandpuresymmetricOuterAutomorphismsofRAAGs}, and it could be interpreted in the setting of the homology theory introduced in that paper (see Remark~\ref{rem:DW dual class}).
For the reader's convenience, we provide a self-contained exposition.

Let $V$ be a real vector space of dimension $n$. Once again, the reader should think of the case $V=\operatorname{Hom}(G,\rr)\cong\rr^n$ for some group $G$ with $n$ generators.
We fix some inner product, an orthonormal basis $\{e_1,\dots,e_n\}$, and the corresponding coordinates $\{y_1,\dots,y_n\}$, that is, $y_i(e_j)=\delta_{ij}$.
Consider three subspaces of $V$ given by the following systems of equations:
\begin{equation}\label{eq: redundant}
\begin{split}
W_1 & = \{y_1=0, \ \sum_{i=1}^n \lambda^1_{ij}y_i=0 \textrm{ for } j=1,\dots,m_1\}, \\
    W_2 & = \{y_2=0, \ \sum_{i=1}^n \lambda^2_{ij}y_i=0 \textrm{ for } j=1,\dots,m_2\}, \\
    W_3 & = \{y_1-y_2=0, \ \sum_{i=1}^n \lambda^3_{ij}y_i=0 \textrm{ for } j=1,\dots,m_3\},
\end{split}
\end{equation}
where for $k\in\lbrace 1,2,3\rbrace$, we have $\lambda^k_{ij}\in \rr$, and $m_k$ is a non-negative integer (possibly zero, in which case it is understood that the subspace is just given by the first equation, as in Example~\ref{ex: bns of PS}). 
Without loss of generality, we assume that each set of equations is minimal. That is, we have $\dim {W_k}=n-(m_k+1)$.

We now proceed to compute the term $\iep{W_1}{W_2}{W_3}$ defined in \eqref{eq:def IEP}.
In the naive system of equations that defines the intersection $W_1\cap W_2\cap W_3$ (that is, the one obtained by putting  all the equations together), there is an obvious linear relation among the equations $y_1=0, y_2=0$, and $y_1-y_2=0$. 
This can cause the dimension of $W_1\cap W_2\cap W_3$ to be higher than expected.
From this perspective, one of the three equations is redundant.
We find it convenient to work with the orthogonal complements.
For $i,j\in \{1,2,3\}, i\neq j$,
consider the following natural maps:

\begin{equation}
\begin{split}
    & I_{ij}:W_i^\perp \cap W_j^\perp \longrightarrow W_i^\perp \oplus W_j^\perp,  \  I_{ij}(u)=(u,-u), \\
    & F_{ij}:W_i^\perp \oplus W_j^\perp \longrightarrow W_i^\perp + W_j^\perp,  \ F_{ij}(\zeta_i,\zeta_j)=\zeta_i+\zeta_j,\\
    & J_{ij}:W_i^\perp \oplus W_j^\perp \longrightarrow W_1^\perp \oplus W_2^\perp \oplus W_3^\perp, 
\end{split}
\end{equation}
where the last one is the natural inclusion (for example, $J_{12}(\zeta_1,\zeta_2)=(\zeta_1,\zeta_2,0)$).
These maps fit in the diagram in Figure~\ref{pic:diagram}, where the first row is exact. 

\begin{figure}[ht]
    \centering
    \include{pictures/linear_algebra}
    \caption{The diagram for Lemma~\ref{lem:LA general}.}
    \label{pic:diagram}
\end{figure}

Let $K_{ij}\subseteq W_1^\perp \oplus W_2^\perp \oplus W_3^\perp$ be the image of $J_{ij}\circ I_{ij}$.
By construction, we have $K_{ij}\cong (W_i+W_j)^\perp=W_i^\perp \cap W_j^\perp$.
Finally, consider the vector  $\xi=(-e_1,e_2,e_1-e_2)\in W_1^\perp \oplus W_2^\perp \oplus W_3^\perp$.
We say that a triple of subspaces $\{W_1,W_2,W_3\}$ as above is a \textit{redundant triple of subspaces} if $\xi\not \in K_{12} + K_{23} + K_{13}$.

\begin{remark}\label{rem:DW dual class}
Although we will not need it, we observe that the condition $\xi\not \in K_{12} + K_{23} + K_{13}$ described above can be interpreted in the sense of the subspace arrangement homology introduced in \cite{DayWadeSubspaceArrangementBNSinvariantsandpuresymmetricOuterAutomorphismsofRAAGs} as follows.
Consider the arrangement $\mathcal W^\perp$ given by the orthogonal complements $\{W_1^\perp,W_2^\perp,W_3^\perp\}$. Then $\{W_1,W_2,W_3\}$ is a redundant triple of subspaces precisely when $\xi$ defines a non-trivial class in $H_1(\mathcal W^\perp)$.
\end{remark}

\begin{lemma}\label{lem:LA general}
In the above notation, if $\{W_1,W_2,W_3\}$ is a redundant triple of subspaces, then it does not satisfy the inclusion-exclusion principle.
More precisely,
\begin{equation*}
\dim (W_1+W_2+W_3 )+1 \leq \iep{W_1}{W_2}{W_3}.
\end{equation*}
\end{lemma}

\begin{proof}
We will compute all the terms that appear in $\iep{W_1}{W_2}{W_3}$ (see \eqref{eq:def IEP}).
The exactness of the first row of the diagram in Figure~\ref{pic:diagram}  yields that
\begin{equation*}
   \dim ( W_i^\perp + W_j^\perp) =  \dim ( W_i^\perp \oplus W_j^\perp) - \dim ( W_i^\perp \cap W_j^\perp) =  2+m_i+m_j - \dim K_{ij}.
\end{equation*}
It follows that
\begin{equation}\label{eq:appendix_intersection2}
    \begin{split}
\dim( W_i\cap W_j)  = & n- \dim ( (W_i\cap W_j)^\perp )    \\ 
                    = &  n- \dim(W_i^\perp + W_j^\perp)    \\   
                    = & n - (2+m_i+m_j) +\dim K_{ij}.
    \end{split}
\end{equation}
We deal with the triple intersection similarly. 
Consider the map
\begin{equation*}
    F:W_1^\perp \oplus W_2^\perp \oplus W_3^\perp \longrightarrow W_1^\perp + W_2^\perp + W_3^\perp, \ F(\zeta_1,\zeta_2,\zeta_3)=\zeta_1+\zeta_2+\zeta_3.
\end{equation*}
We have $\dim (W_1^\perp \oplus W_2^\perp \oplus W_3^\perp)= 3+m_1+m_2+m_3$.
Since $F$ is surjective, its codomain has dimension $3+m_1+m_2+m_3-\dim (\ker F) $. It follows that
\begin{equation}\label{eq:appendix_intersection3}
\begin{split}
    \dim (W_1\cap W_2\cap W_3) 
    & = n- \dim ( (W_1\cap W_2\cap W_3)^\perp )  \\
    & = n - \dim ( W_1^\perp + W_2^\perp + W_3^\perp) \\
    & = n - (3+m_1+m_2+m_3) +\dim (\ker F).
\end{split}
\end{equation}

Using $\dim {W_k}=n-(m_k+1)$, \eqref{eq:appendix_intersection2} and \eqref{eq:appendix_intersection3}, we obtain:
\begin{equation}\label{eq:IEPnKerFij}
    \iep{W_1}{W_2}{W_3}=n +\dim (\ker F) -  \dim K_{12} - \dim K_{13} - \dim K_{23}.
\end{equation}

We now claim that $\dim(\ker F)\geq 1+\dim K_{12} + \dim K_{13} + \dim K_{23}$. 
The vector $\xi=(-e_1,e_2,e_1-e_2)$ is  in $\ker F$, and $K_{ij}$ is a subspace of $\ker F$ by definition.
A direct computation shows that  $K_{ij} \cap K_{ik} =0$. 

By assumption, we also have $\xi\notin K_{12}+K_{13}+K_{23}$. 
Therefore, the direct sum $\operatorname{span}(\xi)\oplus K_{12} \oplus K_{13} \oplus K_{23}$ is a subspace of $\ker F$. 
This proves the claim. 
Then it follows from  \eqref{eq:IEPnKerFij} that 
\begin{equation*}
\begin{split}
\iep{W_1}{W_2}{W_3}
& = n +\dim (\ker F) -  \dim K_{12} - \dim K_{13} - \dim K_{23}  \\
& \geq n+1  \\
& \geq \dim ( W_1+W_2+W_3)+1.
\end{split}
\end{equation*}
This completes the proof.
\end{proof}

On the other hand, if $\{W_1,W_2,W_3\}$ is not a redundant triple of subspaces, then we have the dichotomy in the following statement. This criterion will be useful in the proof of Theorem~\hyperref[thm:redundant triple criterion]{E}.
\begin{lemma}\label{new linear algebra}
In the above notation, if $\{W_1,W_2,W_3\}$ is not a redundant triple of subspaces, then one of the following situations occurs:
\begin{enumerate}
    \item \label{item:e1 e2 in all} either $e_1,e_2\in W_j^\perp$ for all $j=1,2,3$,

    \item \label{item:engaged in all} or there exists some $i\geq 3$ such that $e_i \not \in W_j$ for all $j=1,2,3$. 
\end{enumerate}
\end{lemma}

\begin{proof}
Recall that $K_{ij}$ is the image of the natural map $J_{ij}\circ I_{ij}:W_i^\perp \cap W_j^\perp \to W_1^\perp \oplus W_2^\perp \oplus W_3^\perp$ (see Figure~\ref{pic:diagram} at the beginning of \S\ref{sec:IEP NON-RAAG behavior}).
We have an induced map
$$
K:(W_1^\perp \cap W_2^\perp ) \oplus (W_2^\perp \cap W_3^\perp ) \oplus (W_1^\perp \cap W_3^\perp  )  \to W_1^\perp \oplus W_2^\perp \oplus W_3^\perp,
$$
$$
K(a,b,c) = (a+c,-a+b,-b-c),
$$
whose image is precisely $K_{12} + K_{23} + K_{13}$.
Since $\{W_1,W_2,W_3\}$ is not a redundant triple of subspaces, we have $\xi=(-e_1,e_2,e_1-e_2)\in \operatorname{Im}(K)$.
This means that there exist $a=\sum_{i=1}^n a_ie_i\in W_1^\perp \cap W_2^\perp$, $b=\sum_{i=1}^n b_ie_i\in W_2^\perp \cap W_3^\perp$, and $c=\sum_{i=1}^n c_ie_i\in W_1^\perp \cap W_3^\perp$, such that $a+c=-e_1$, $-a+b=e_2$, and $-b-c=e_1-e_2$, where $a_i,b_i,c_i\in \rr$.
A direct computation shows that $a$, $b$, and $c$ must satisfy the following relations:
\begin{equation}\label{eq:new linear relations}
    a_1=b_1=-1-c_1, \ a_2=-c_2=b_2-1, \ \text{and} \ a_i=b_i=-c_i \ \text{for} \  i\geq 3.
\end{equation}
 
Note that if $ a_i$, $b_i$, and $c_i$ are equal to zero for all $i\geq 3$, then $a=a_1e_1+a_2e_2\in  W_1^\perp \cap W_2^\perp$. 
Since $e_1\in W_1^\perp$, we have $e_2\in W_1^\perp$. 
Similar arguments show that $e_1$ and $e_2$ also belong to $W_2^\perp$ and $W_3^\perp$.
Therefore, we are in case \eqref{item:e1 e2 in all}.

If \eqref{item:e1 e2 in all} does not occur, then we can reduce to the case that one of $a$, $b$, and $c$ has at least one non-zero coordinate along $e_i$ for some $i\geq 3$. But $a_i\neq 0$ implies that  $W_1^\perp$ and $e_i$ are not orthogonal, so we have $e_i\not \in W_1$.
Thanks to \eqref{eq:new linear relations}, 
we also know that  $b$ and $c$ have non-zero coordinates along $e_i$. Then a similar argument shows that $e_i\not \in W_2, W_3$. Therefore, we are in case \eqref{item:engaged in all}.
\end{proof}


Finally, we obtain a criterion to certify that a group is not a RAAG.

\begin{proposition}\label{prop:criterion non RAAG}
Let $G$ be a finitely generated group. 
Suppose that there exist three maximal missing 
subspaces $W_1$, $W_2$, and $W_3$
in $\operatorname{Hom}(G, \rr)$.
If they form a redundant triple of subspaces, then $G$ is not a RAAG.
\end{proposition}

\begin{proof}
Since $\{W_1,W_2,W_3\}$ is a redundant triple of subspaces, by Lemma~\ref{lem:LA general} we have that $\dim (W_1+W_2+W_3 ) +1 \leq \iep{W_1}{W_2}{W_3}.$
Assume by contradiction that $G$ is a RAAG.
Then by Lemma~\ref{lem:koban piggott} we have
$\dim (W_1+W_2+W_3 ) = \iep{W_1}{W_2}{W_3}$.
This leads to a contradiction.
\end{proof}

The fact that certain BBGs are not isomorphic to RAAGs can be obtained via the methods in \cite{PapadimaSuciuAlgebraicinvariantsforBBGs} or \cite{DayWadeSubspaceArrangementBNSinvariantsandpuresymmetricOuterAutomorphismsofRAAGs}, such as the BBG defined on the trefoil graph in Example~\ref{ex: bns of PS}. 
Proposition~\ref{prop:criterion non RAAG} allows us to obtain new examples that were not covered by previous criteria, such as the BBG defined on the extended trefoil (see Examples~\ref{ex:extended trefoil} and \ref{ex:extended trefoil continued}).

\subsection{Redundant triples for BBGs}\label{sec:redundant triples BBGs}
The purpose of this section is to find a general graphical criterion to certify that a BBG is not a RAAG.
The idea is to start from a triangle in the flag complex $\flag \Gamma$ and find suitable subspaces of the links of its vertices that induce a redundant triple of subspaces in the complement of $\bns{\bbg \Gamma}$.
Let $\tau$ be a triangle in $\flag \Gamma$ with vertices $(v_1,v_2,v_3)$. Let  $e_j$ be the edge opposite to $v_j$. 
We say that $\tau$  is a \textit{redundant triangle} if for each $j=1,2,3$, there exists a subgraph $\Lambda_j\subseteq \lk{v_j,\Gamma}$ such that:
    \begin{enumerate}
        \item $e_j \in \ee{\Lambda_j}$;
        \item $\Lambda_j$ is a minimal separating subgraph of $\Gamma$;
        \item \label{item: omega}  $\Lambda_1\cap \Lambda_2 \cap \Lambda_3$ is the empty subgraph.
    \end{enumerate}

\begin{example}
The central triangle in the trefoil graph in Figure~\ref{fig:trefoil} is redundant.
However, if we consider the cone over the trefoil graph, then the central triangle in the base trefoil graph is not redundant. Redundant triangles can appear in higher-dimensional complexes; see Example~\ref{ex:higher dimensional}.
\end{example}

The purpose of this section is to prove the following theorem.

\begin{maintheoremc}{E}\label{thm:redundant triple criterion}
 Let $\Gamma$ be a biconnected graph such that $\flag \Gamma$ is simply connected.
 If $\Gamma$ has a redundant triangle, then $\bbg \Gamma$ is not a RAAG.
 \end{maintheoremc}

We start by considering  a redundant triangle $\tau$ with a choice of subgraph $\Lambda_j$ of the link $\lk{v_j,\Gamma}$ as in the above definition of redundant triangle. 
We denote by $W_j=W_{\Lambda_j}$ the induced subspace of $V=\operatorname{Hom}(\bbg \Gamma,\rr)$.
 By Theorem~\hyperref[thm:graphical description for bns of bbg]{D}, we know that $W_j=W_{\Lambda_j}$ is a maximal subspace in the complement of $\bns{\bbg\Gamma}$.
We want to show that $\{W_1,W_2,W_3\}$ is a redundant triple of subspaces.
To do this, we will choose some suitable coordinates on $V$, that is, a suitable spanning tree for $\Gamma$.
Notice that different spanning trees correspond to different bases on  $\operatorname{Hom}(\bbg \Gamma,\rr)$. In particular, the linear isomorphism class of the arrangement of missing subspaces does not depend on these choices, and we can work with a convenient spanning tree.

To construct a desired spanning tree, we will need the following terminology. 
Let $v\in\vv\Gamma$. 
The \textit{spoke} of $v$ in $\Gamma$ is the subgraph $\spoke v$ consisting of the edges that contain $v$. Note that  $\spoke v$ is a spanning tree of $\st{v}$.
Let $\Lambda$ be a subgraph of $\lk v$.
We define the \textit{relative star} of $v$ with respect to $\Lambda$ to be the full subgraph $\relstar v\Lambda$ of $\st{v}$ generated by $\{v\} \cup \vv \Lambda$.
We define the \textit{relative spoke} of $v$ with respect to $\Lambda$ to be the subgraph $\relspoke v\Lambda$ of $\spoke{v}$ consisting of the edges that connect $v$ to a vertex of $\Lambda$. 
Note that $\relspoke v\Lambda$ is a spanning tree of $\relstar v\Lambda$.
We now construct a spanning tree $T$ for $\Gamma$ as follows. 
\begin{itemize}
    \item Let $T_3=\relspoke{v_3}{\Lambda_3}$.
    Since we chose $\Lambda_3$ to contain $e_3$, we have $v_1,v_2\in \vv{\Lambda_3}$ and $e_1,e_2 \in \ee{T_3}$.
    
    \item Let $Z_2=\relspoke{v_2}{\Lambda_2 \setminus \relstar{v_3}{\Lambda_3}}$ and let $T_2=T_3\cup Z_2$.
    Notice that $T_2$ is a spanning tree of $\relstar{v_2}{\Lambda_2}\cup \relstar{v_3}{\Lambda_3}$.
    
    \item Let $Z_1=\relspoke{v_1}{\Lambda_1 \setminus (\relstar{v_2}{\Lambda_2}\cup \relstar{v_3}{\Lambda_3})}$ and let $T_1=T_2\cup Z_1$.
    Notice that $T_1$ is a spanning tree of $\relstar{v_1}{\Lambda_1}\cup \relstar{v_2}{\Lambda_2}\cup \relstar{v_3}{\Lambda_3}$.  
    
    \item Finally, extend $T_1$ to a spanning tree $T$ for $\Gamma$.
\end{itemize}

To fix notation, say $\ee T = \{f_1,f_2,\dots, f_n\}$.
Without loss of generality, say $f_1=e_1$ and $f_2=e_2$. 
Fix an arbitrary orientation for the edges of $T$.
With respect to the associated system of coordinates the subspaces $W_1$, $W_2$, and $W_3$ are given by equations of the form \eqref{eq: redundant}:
\begin{equation*}
    W_1 = \{y_1=0, \dots\}, \  W_2 = \{y_2=0, \dots\},  \ \text{and} \  W_3 = \{y_1-y_2=0, \dots\}.
\end{equation*}
Recall that $\{\chi_f \mid f\in \ee T\}$ is a basis for $\operatorname{Hom}(\bbg \Gamma,\rr)$, where $\chi_f:\bbg \Gamma \to \rr$ is the character  defined by $\chi_f(e)=1$ if $f=e$ and $\chi_f(e)=0$ if $f\neq e$.
We also fix a background inner product with respect to which $\{\chi_f \mid f\in \ee T\}$ is an orthonormal basis.
We now proceed to prove some technical lemmas that will be used to recognize the edges $f\in \ee T$ for which the associated character $\chi_f$ is in one of the subspaces $W_1$, $W_2$, and $W_3$. This is needed to use Lemma~\ref{new linear algebra}.
We start with the following general fact.

\begin{lemma}\label{lem: chi_f not in lk then f is in a spanning tree}
    Let $v\in\Gamma$, and let $\Lambda$ be a subgraph of $\lk{v}$. 
    Let $T_\Lambda$ be a spanning tree of $\relstar v\Lambda$. 
    If $f\not \in \ee {T_\Lambda}$, then $\chi_f\in W_{\Lambda}$.
\end{lemma}
\begin{proof}
    Suppose $f\notin \ee {T_\Lambda}$. 
    Then $\chi_f=0$ on $T_\Lambda$. 
    Since $T_\Lambda$ is a spanning tree of $\relstar v\Lambda$, the character $\chi_f$ vanishes on $\relstar v\Lambda$ by Lemma~\ref{lem:characters of f.p BBGs are given by assigning values on spanning trees}. 
    In particular, it vanishes on $\Lambda$, hence $\chi_f\in W_{\Lambda}$.
\end{proof}

We now proceed to use Lemma~\ref{lem: chi_f not in lk then f is in a spanning tree} for each $\Lambda_j$, with respect to a suitable choice of spanning tree for $\relstar{v_j}{\Lambda_j}$.

\begin{lemma}\label{lem:involve3}
    Let $f\in \ee T$. 
    If $f \not \in \ee{T_3}$, then $\chi_f \in W_3$.
\end{lemma}
\begin{proof}
Since $T_3$ is a spanning tree of $\relstar{v_3}{\Lambda_3}$, the statement follows directly from Lemma~\ref{lem: chi_f not in lk then f is in a spanning tree}.
\end{proof}

\begin{lemma}\label{lem:involve2}
    Let $f\in \ee T$.
    If $f\not \in \ee{Z_2}$, $f\neq e_1$, and $f$ does not join $v_3$ to a vertex in $\Lambda_2\cap \Lambda_3$,
    then $\chi_f \in W_2$.
\end{lemma}
\begin{proof}
We construct a spanning tree for $\relstar{v_2}{\Lambda_2}$ as follows.
First, note that $Z_2$ is a spanning tree of $\relstar{v_2}{\Lambda_2 \setminus \relstar{v_3}{\Lambda_3}}$ by construction.
If $u$ is a vertex in $\relstar{v_2}{\Lambda_2}$ but not in $Z_2$, then either $u=v_3$ or $u\in \vv{\Lambda_3}$.
Let $T_2'$ be the result of  extending $Z_2$ with the edge $e_1=(v_2,v_3)$ and all the edges that join $v_3$ to the vertices in $\Lambda_2\cap \Lambda_3$.
This gives a spanning subgraph $T_2'$ of $\relstar{v_2}{\Lambda_2}$. 
Note that $T_2'$ is a tree because it is a subgraph of $T$.
By the choice of $f$, we have $f\not \in \ee{T_2'}$. Then it follows from Lemma~\ref{lem: chi_f not in lk then f is in a spanning tree} that $\chi_f\in W_2$.
\end{proof}
 
\begin{lemma}\label{lem:involve1}
    Let $f\in \ee T$.
    If $f\not \in \ee{Z_1}$,  $f\neq e_1,e_2$, and $f$ does not join $v_3$ to a vertex in $\Lambda_1\cap \Lambda_3$ nor $v_2$ to a vertex in $\Lambda_1\cap \Lambda_2$,
    then $\chi_f \in W_1$.
\end{lemma}
\begin{proof}
We construct a spanning tree for $\relstar{v_1}{\Lambda_1}$ as follows.
First, note that $Z_1$ is a spanning tree for $\relstar{v_1}{\Lambda_1 \setminus (\relstar{v_2}{\Lambda_2}\cup \relstar{v_3}{\Lambda_3})}$ by construction.
If $u$ is a vertex in $\relstar{v_1}{\Lambda_1}$ but not in $Z_1$, then either $u=v_2$, $u=v_3$, $u\in \vv{\Lambda_2}$, or $u\in \vv{\Lambda_3}$.
Let $T_1'$ be the result of  extending $Z_1$ with the edges $e_1=(v_2,v_3)$, $e_2=(v_1,v_3)$, all the edges that join $v_3$ to the vertices in $\Lambda_1\cap \Lambda_3$, and all the edges that join $v_2$ to the vertices in $\Lambda_1\cap \Lambda_2$.
This gives a spanning subgraph $T_1'$ of $\relstar{v_1}{\Lambda_1}$. 
Note that $T_1'$ is a tree because it is a subgraph of $T$.
By the choice of $f$, we have $f\not \in \ee{T_1'}$. Then it follows from Lemma~\ref{lem: chi_f not in lk then f is in a spanning tree} that $\chi_f\in W_1$.
\end{proof}

\begin{lemma}\label{lem:involved in all only e1 e2}
Let $f\in \ee T$. 
If $\chi_f\not \in W_j$ for all $j=1,2,3$, then $f=e_1$ or $f=e_2$.
\end{lemma}
\begin{proof}
By contradiction, suppose that there is an edge $f\neq e_1,e_2$ such that $\chi_f\not \in W_j$ for all $j=1,2,3$.
Since $\chi_f\not \in W_3$, we know that $f\in \ee{T_3}$ by Lemma~\ref{lem:involve3}. 
In particular, we have  $v_3\in \vv f$.
Since $f\neq e_1,e_2$, this implies that $v_1,v_2\not \in \vv f$, and in particular this means that $f\notin \ee{Z_1}, \ee{Z_2}$.
The assumption $\chi_f\not \in W_2$ implies that $f$ joins $v_3$ to a vertex in $\Lambda_2\cap \Lambda_3$, thanks to Lemma~\ref{lem:involve2}.
Similarly, the assumption $\chi_f\not \in W_1$ implies that $f$ joins $v_3$ to a vertex in $\Lambda_1 \cap \Lambda_3$, thanks to Lemma~\ref{lem:involve1}. 
Therefore, we have obtained that $f$ connects $v_3$ to a vertex in $\Lambda_1 \cap \Lambda_2\cap \Lambda_3$. But this is absurd because this intersection is empty, by condition \eqref{item: omega} in the definition of redundant triangle.
\end{proof}

We are now ready for the proof of Theorem~\hyperref[thm:redundant triple criterion]{E}.

\begin{proof}[Proof of Theorem~{\hyperref[thm:redundant triple criterion]{E}}]
Recall by construction that the subspaces $W_1$, $W_2$, and $W_3$ are given by equations of the  form \eqref{eq: redundant} with respect to the coordinates defined by the spanning tree $T$ constructed above.
Suppose by contradiction that $\{W_1,W_2,W_3\}$ is not a redundant triple
By Lemma~\ref{new linear algebra}, one of the following cases occurs:
\begin{enumerate}
    \item \label{item:e1 e2 in all bbg} either $\chi_{e_1},\chi_{e_2}\in W_j^\perp$ for all $j=1,2,3$,

    \item \label{item:engaged in all bbg} or there exists some $i\geq 3$ such that $\chi_{f_i} \not \in W_j$ for all $j=1,2,3$. 
\end{enumerate}

We claim that neither of these two situations can occur in our setting.
To see that \eqref{item:e1 e2 in all} does not occur, observe that $\chi_{e_1}+\chi_{e_2} \in W_3$ and $\chi_{e_1}+\chi_{e_2}$ is not orthogonal to $\chi_{e_1}$, so $\chi_{e_1}\not \in W_3^\perp$. The same is true for $\chi_{e_2}$. 
On the other hand, \eqref{item:engaged in all} does not occur  by Lemma~\ref{lem:involved in all only e1 e2}.
We have reached a contradiction, so $\{W_1,W_2,W_3\}$ is a redundant triple of subspaces.
Then it follows from Proposition~\ref{prop:criterion non RAAG} that $\bbg\Gamma$ is not a RAAG.
 \end{proof}

We will use Theorem~\hyperref[thm:redundant triple criterion]{E} in \S\ref{section: BBGs on 2-dim flag complexes} to prove that certain BBGs are not isomorphic to RAAGs (see Theorem~\hyperref[body main thm 2dim]{A} for the case in which $\flag \Gamma$ is 2-dimensional and Example~\ref{ex:higher dimensional} for a higher-dimensional example).


\subsection{Resonance varieties for BBGs}\label{sec:resonance varieties}
The goal of this section is to show that for a finitely presented BBG, the complement of its BNS-invariant coincides with the restriction of its first real resonance variety to the character sphere.

Let $A=H^*(G,\rr)$ be the cohomology algebra of $G$ over $\rr$.
For each $a\in A^1=H^1(G,\rr)=\operatorname{Hom}(G,\rr)$, we have $a^2=0$. So, we can define a cochain complex $(A,a)$
$$(A,a): A^0 \to A^1 \to A^2 \to \cdots,$$
where the coboundary is given by the right-multiplication by $a$.
The \textit{(first) resonance variety} is defined to be the set of points in $A^1$ so that the above chain complex fails to be exact, that is,
$$ \mathcal R_1(G)=\{ a \in A^1 \mid H^1(A,a) \neq 0\}.$$

In many cases of interest, the resonance variety $\mathcal R_1(G)$ is an affine algebraic subvariety of the vector space $A^1=H^1(G,\rr)=\operatorname{Hom}(G,\rr)$.
For $G$ a RAAG or a BBG, these varieties have been computed in \cite{PapadimaSuciuAlgebraicinvariantsforRAAGs} and \cite{PapadimaSuciuAlgebraicinvariantsforBBGs}, respectively.
These varieties turn out to be defined by linear equations; that is, they are subspace arrangements.
Following the notation in \cite{PapadimaSuciuAlgebraicinvariantsforBBGs}, let $\Gamma$ be a finite graph. For any $U\subseteq \vv \Gamma$, let $H_U$ be the set of characters $\chi:\raag\Gamma \to \rr$ vanishing (at least) on all the vertices in the complement of $U$.
In our notation, this means $U\subseteq \dead{\hat \chi}$.
Moreover, let $H'_U$ be the image of $H_U$ under the restriction map $ r\colon \operatorname{Hom}(\raag \Gamma,\rr) \to \operatorname{Hom}(\bbg \Gamma,\rr)$; see \S\ref{sec:coordinates}.

\begin{proposition}\label{prop:bns resonance}
Let $\Gamma$ be a biconnected   graph with $\flag \Gamma$ simply connected.
Then $\bnsc{\bbg \Gamma} = \mathcal R_1(\bbg \Gamma) \cap \chars{\bbg \Gamma}$. 
\end{proposition}
\proof
By \cite[Theorem 1.4]{PapadimaSuciuAlgebraicinvariantsforBBGs}, we have that $\mathcal R_1(\bbg \Gamma)$ is the union of the subspaces $H'_U$, where $U$ runs through the maximal collections of vertices that induce disconnected subgraphs.
Similarly, it follows  from Theorem~\hyperref[thm:graphical description for bns of bbg]{D}   that $\bnsc{\bbg \Gamma} $ is the union of the subspheres $S_\Lambda$, where $\Lambda$ runs through the minimal separating subgraphs.
Note  that  $U$ is a maximal collection of vertices inducing a disconnected subgraph precisely when the subgraph $\Lambda$ induced by $\vv \Gamma \setminus U$ is a minimal separating subgraph.
So, it is enough to show that for each such $U$,  we have $  H'_U = W_\Lambda$, where  $W_\Lambda$ is the linear span of the sphere $S_\Lambda$, as defined above in \S\ref{sec:graphical description}.

To show this equality, let $\chi:\bbg \Gamma \to \rr$. Then $\chi \in H'_U$ if and only if there is an extension $\hat \chi$ of $\chi$ to $\raag \Gamma$ such that $\hat \chi \in H_U$.
This means that $\Lambda\subseteq \dead{\hat \chi}$.
Note that $\Lambda$ is connected by  \eqref{item:minimal separating is connected} in Lemma~\ref{lem:link connected}.
So, by Lemma~\ref{a subgraph is edge dead iff there is an extension of characters on BBG to RAAG}, we have $\Lambda\subseteq \dead{\hat \chi}$ if and only if  $\Lambda \subseteq \deadedge \chi$, which is equivalent to $\chi \in W_\Lambda$.
\endproof

\begin{remark}\label{rem:upper bound}
We note that in general one has the inclusion 
$\bnsc{G} \subseteq \mathcal R_1(G) \cap \chars{G}$ thanks to \cite[Theorem 15.8]{PapadimaandSuciuBNSRinvariantsandHomologyJumpingLoci}.
However, the equality does not always hold; see \cite[\S 8]{SU21} for some examples, such as the Baumslag--Solitar group $\operatorname{BS}(1,2)$.
\end{remark}

We now recall a construction  that reduces an Artin group to a RAAG (see \cite[\S 11.9]{oddconstruction} or \cite[\S 9]{PapadimaSuciuAlgebraicinvariantsforBBGs}).
Let $(\Gamma,m)$ be a weighted graph, where $m\colon \ee{\Gamma}\to\nn$ is an assignment of positive integers on the edge set. 
We denote by $\raag{\Gamma,m}$ the associated \textit{Artin group}.
When $m=2$ on every edge, it reduces to $\raag{\Gamma,m}=\raag \Gamma$.
The \textit{odd contraction} of $(\Gamma,m)$ is an unweighted graph $\widetilde{\Gamma}$ defined as follows. Let $\Gamma_{odd}$ be the graph whose vertex set is $\vv{\Gamma}$ and edge set is $\lbrace e\in\ee{\Gamma} \ \vert \ \text{$m(e)$ is an odd number}\rbrace$. The vertex set $\vv{\widetilde{\Gamma}}$ of $\widetilde{\Gamma}$ is the set of connected components of $\Gamma_{odd}$, and two vertices $C$ and $C'$ are connected by an edge if there exist adjacent vertices $v\in \vv C$ and $v'\in \vv C'$ in the original graph $\Gamma$.

\begin{corollary}\label{cor:not Artin}
Let $\Gamma$ be a biconnected graph such that $\flag \Gamma$ is simply connected.
 If $\Gamma$ has a redundant triangle, then $\bbg \Gamma$ is not an Artin group.
\end{corollary}
\begin{proof}
Let $\tau$ be a redundant triangle, with chosen minimal full separating subgraphs $\{\Lambda_1,\Lambda_2,\Lambda_3\}$.
 Let $W_j=W_{\Lambda_j}$ be the subspace of $V=\operatorname{Hom}(\bbg\Gamma,\rr)$ defined by $\Lambda_j$.
Arguing as in the proof of Theorem~\hyperref[thm:redundant triple criterion]{E}, we have that $W_j$ is a maximal missing subspace in the complement of $\bns{\bbg \Gamma}$, and that $\{W_1,W_2,W_3\}$ is a redundant triple in subspaces of $V$. 
By Lemma~\ref{lem:LA general}, we have
$$
\dim (W_1+W_2+W_3 )+1 \leq \iep{W_1}{W_2}{W_3}.
$$

Now, assume by contradiction that $\bbg \Gamma$ is isomorphic to an Artin group $\raag {\Gamma ',m}$.
Let $\widetilde{\Gamma}'$ be the odd contraction of $(\Gamma ',m)$.
Then $\raag{\widetilde{\Gamma}'}$ is a RAAG.
Notice that $\raag {\Gamma ',m}$ and $\raag{\widetilde{\Gamma}'}$ have the same abelianization. Hence, the three spaces 
$\operatorname{Hom}(\raag {\Gamma ',m},\rr)$, $\operatorname{Hom}(\raag{\widetilde{\Gamma}'},\rr)$, and $V$ can be identified together. In particular,  the three character spheres $\chars{\raag{\Gamma,m}}$, $\chars{\raag{\widetilde{\Gamma}'}}$, and $\chars{\bbg \Gamma}$ can be identified as well.

Arguing as in \cite[Proposition 9.4]{PapadimaSuciuAlgebraicinvariantsforBBGs}, 
there is an ambient isomorphism of the resonance varieties $\mathcal R_1(\bbg \Gamma)\cong \mathcal R_1(\raag{ \Gamma',m})\cong  \mathcal R_1(\raag{\widetilde{\Gamma}'})$, seen as subvarieties of $V$.
Since $\raag{\widetilde{\Gamma}'}$ is a RAAG, by \cite[Theorem 5.5]{PapadimaSuciuAlgebraicinvariantsforRAAGs} we have $\bnsc{\raag{\widetilde{\Gamma}'}} = \mathcal R_1(\raag{\widetilde{\Gamma}'}) \cap \chars{\raag{\widetilde{\Gamma}'}}$.
Similarly, since $\bbg \Gamma$ is a BBG, by Proposition~\ref{prop:bns resonance} we have
$\bnsc{\bbg \Gamma} = \mathcal R_1(\bbg \Gamma) \cap \chars{\bbg \Gamma}$.
It follows that we have an ambient isomorphism of  the complements of the BNS-invariant $\bnsc{\bbg \Gamma}\cong \bnsc{\raag{\widetilde{\Gamma}'}}$ (seen as arrangements of subspheres in $\chars{\bbg \Gamma}$), as well as an ambient isomorphism of the associated arrangements of (linear) subspaces of $V$.
In particular, the arrangement of maximal missing subspaces of $\bbg \Gamma$ inside $V$ is ambient isomorphic to the arrangement of maximal missing subspaces of a RAAG.
Applying Lemma~\ref{lem:koban piggott} to the triple $\{W_1,W_2,W_3\}$ gives
$$\iep{W_1}{W_2}{W_3} = \dim (W_1+W_2+W_3 ).$$
This leads to a contradiction.
\end{proof}


\section{BBGs on 2-dimensional flag complexes}\label{section: BBGs on 2-dim flag complexes}

If $\flag \Gamma$ is a simply connected flag complex of dimension one, then $\Gamma$ is a tree. In this case, the group $\bbg \Gamma$ is a free group generated by all the edges of $\Gamma$, and in particular, it is a RAAG. 
The goal of this section is to determine what happens in dimension two. 
Namely, we will show that the BBG defined on a $2$-dimensional complex is a RAAG if and only if a certain poison subgraph is avoided.
We will discuss some higher dimensional examples at the end; see Examples~\ref{ex:cone over PS} and Example~\ref{ex:higher dimensional}.

Throughout this section, we assume that $\Gamma$ is a biconnected graph such that $\flag \Gamma$ is 2-dimensional and simply connected unless otherwise stated. 
Note that by Lemma~\ref{lem:link connected} this implies that $\flag \Gamma$ is homogeneous of dimension two. 
We say that

\begin{itemize}
    \item An edge $e$ is a \textit{boundary edge} if it is contained in exactly one triangle. Denote by $\partial \flag \Gamma$ the \textit{boundary} of $\flag \Gamma$.
    This is a 1-dimensional subcomplex consisting of boundary edges.
    An edge $e$ is an \textit{interior edge} if $e\cap \partial \flag \Gamma = \varnothing$.
    Equivalently, none of its vertices is on the boundary.

    \item A \textit{boundary vertex} is a vertex contained in  $\partial \flag \Gamma$. Equivalently, it is contained in at least one boundary edge.
    A vertex $v$ is an \textit{interior vertex} if it is contained only in edges that are not boundary edges.
    
    \item A triangle $\tau$ is an \textit{interior triangle} if $\tau \cap \partial \flag \Gamma = \varnothing$.
    A triangle $\tau$ is called a \textit{crowned triangle} if none of its edges is in $\partial \flag \Gamma$. 
    This is weaker than being an interior triangle because a crowned triangle can have vertices in $\partial \flag \Gamma$.
    If $\tau$ is a crowned triangle, each of its edges is contained in at least one  triangle different from $\tau$. 
\end{itemize}

\begin{remark}
We will prove in Lemma~\ref{lem:crowned tri is redundant in dim 2} that in dimension two, a crowned triangle is redundant in the sense of \S\ref{sec:redundant triples BBGs}.
If $\partial \flag \Gamma$ is empty, then every triangle is crowned, simply because no edge can be a boundary edge.
Note that a vertex is either a boundary vertex or an interior vertex, but we might have edges which are neither boundary edges nor interior edges. 
For example, the trefoil graph (see Figure~\ref{fig:trefoil}) has no interior edges, but only six of its nine edges are boundary edges. Moreover, it has no interior triangles, but it has one crowned triangle.
Notice that a crowned triangle is contained in a trefoil subgraph of $\Gamma$, but the trefoil subgraph is not necessarily a full subgraph of $\Gamma$;  see Figure~\ref{fig: diamond and house}.
\end{remark}

\begin{figure}[h]
    \centering
    \input{pictures/diamond_house}
    \caption{Graphs that contain crowned triangles, but the resulting trefoil subgraphs are not full subgraphs.}
    \label{fig: diamond and house}
\end{figure}
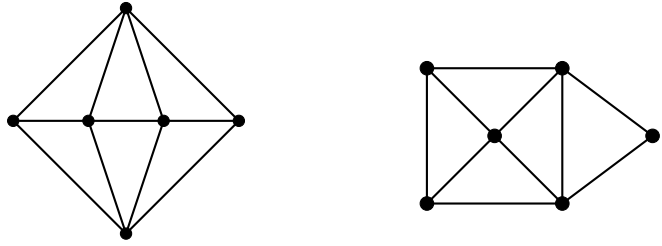


\subsection{Complexes without crowned triangles}

The goal of this section is to provide a characterization of complexes without crowned triangles.

\begin{lemma}\label{lem: v is interior iff its link has no deg 1 vertices}
A vertex $v \in \vv \Gamma$ is an interior vertex if and only if for each vertex $w$ in $\lk{v,\Gamma}$, its degree in $\lk{v,\Gamma}$ is at least two.
\end{lemma}
\begin{proof}
First of all, notice that for a vertex $w$ in $\lk{v,\Gamma}$, its degree in $\lk{v,\Gamma}$ is equal to the number of triangles of $\flag \Gamma$ that contain the edge $(v,w)$.

Suppose that $v \in \vv \Gamma$ is an interior vertex, and let $w$ be a vertex of $\lk{v,\Gamma}$.
Since $v$ is interior, the edge $(v,w)$ is not a boundary edge, hence it is contained in at least two triangles.
Therefore, the vertex $w$ has degree at least two in
$\lk{v,\Gamma}$.
Conversely, let $v\in \vv \Gamma$ and let $e=(v,w)$ be an edge containing $v$, where $w$ is some vertex in $\lk{v,\Gamma}$.
Since the degree of $w$ in  $\lk{v,\Gamma}$ is at least two, the edge $e$ must be contained in at least two triangles. 
Thus, the edge $e$ is not in $\partial \flag \Gamma$. Hence, the vertex $v$ is an interior vertex.
\end{proof}

\begin{lemma}\label{lem: no crowned triangles implies no interior 1-2 simplex and at most 1 interior vertex}
Let $\Gamma$ be a biconnected   graph such that $\flag \Gamma$ is $2$-dimensional and simply connected. If $\flag \Gamma$ has no crowned triangles, then $\flag \Gamma$ has no interior triangles, no interior edges, and has at most one interior vertex.
\end{lemma}

\begin{proof}
Since an interior triangle is automatically a crowned triangle, it is clear that $\flag \Gamma$ has no interior triangles. 

For the second statement, assume that there is  an interior edge $e=(u,v)$  of $\flag \Gamma$. 
Since $e$ is an interior edge, it is contained in at least two triangles. Let  $\tau$ be a triangle containing $e$. Let $w$ be the third vertex of $\tau$, and let $e_1$ and $e_2$ be the other two edges of $\tau$.
Since $u$ and $v$ are interior vertices, we have that $e_1$ and $e_2$ are not in $\partial \flag \Gamma$. So, no edge of $\tau$ is a boundary edge. That is, the triangle $\tau$ is a crowned triangle, a contradiction.

Finally, let $v$ be an interior vertex. 
By definition, none of the edges containing $v$ is in $\partial \flag \Gamma$. 
We claim that $\Gamma = \st{v,\Gamma}$, and in particular, there are no other interior vertices. 
First, take a triangle $\tau$ containing $v$, then the two edges of $\tau$ that meet at $v$ are not in $\partial \flag \Gamma$.
The third edge of $\tau$ must be a boundary edge; otherwise, the triangle $\tau$ would be a crowned triangle. 
This shows that all vertices in $\lk{v,\Gamma}$ are in $\partial \flag \Gamma$.
Now, assume by contradiction that there is a vertex $u$ at distance two from $v$. 
Let $w$ be a vertex in $\lk{v,\Gamma}$ that is adjacent to $u$. Note that $w$ is a boundary vertex. Since $\lk{w,\Gamma}$ is connected by \eqref{item:link of a vertex is connected} in Lemma~\ref{lem:link connected}, there is a path $p$ in $\lk{w, \Gamma}$ from $u$ to a vertex $u'$ in $\lk{v,\Gamma}\cap\lk{w,\Gamma}$; see Figure~\ref{fig: proof of only one interior vertex}. Then the path $p$, together with the edges $(w,u)$ and $(w,u')$, bounds a triangulated disk in $\flag\Gamma$. Then the edge $(w,u')$ is contained in more than one triangle, and therefore, it is not a boundary edge, and the triangle formed by the vertices $v$, $w$, and $u'$ is a crowned triangle, a contradiction. 
\end{proof}

\begin{figure}[h]
    \centering
    \input{pictures/proof_only_one_interior_vertex}
    \caption{The path $p$ and the edges $(w,u)$ and $(w,u')$ bound a triangulated disk in $\flag \Gamma$. This implies that the edge $(w,u')$ is not a boundary edge, and the vertices $v$, $w$, and $u'$ form a crowned triangle.}
    \label{fig: proof of only one interior vertex}
\end{figure}
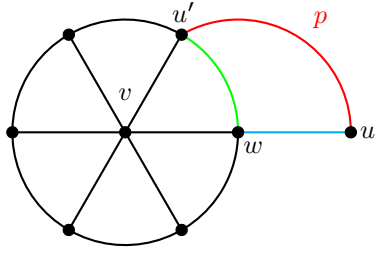

Before we prove the next result, we give some terminologies on graphs.
A graph $\Gamma$ is called an \emph{edge-bonding}  of two graphs $\Gamma_1$ and $\Gamma_2$ if it is obtained by identifying two edges $e_1\in \ee {\Gamma_1}$ and $e_2\in \ee {\Gamma_2}$. 
If $e$ denotes the image of $e_1$ and $e_2$ in $\Gamma$, we also write $\Gamma=\Gamma_1\cup_{e}\Gamma_2$ and say that $e$ is the \textit{bonding edge}.

\begin{remark}\label{rem:simultaneous edge bonding}
Since an edge-bonding involves identifying two edges from two different graphs, one can perform several edge-bondings of a collection of graphs simultaneously.
In particular, if one performs a sequence of edge-bondings, then the result can actually be obtained by a simultaneous edge-bonding.

We also note that there are two ways of identifying $e_1$ and $e_2$ that can result in two different graphs. However, this will not be relevant in the following.
\end{remark}

Our goal is to decompose a given graph as an edge-bonding of certain elementary pieces that we now define.
A \emph{fan} is a cone over a path. 
Let $\Gamma_0$ be a connected   graph having no vertices  of degree $1$  and whose associated flag complex $\flag {\Gamma _{0}}$ is $1$-dimensional. Note that $\Gamma_0$ contains no triangles.
The cone over such a $\Gamma_0$ is called a \emph{simple cone}; see Figure~\ref{fig: example of simple cone} for an example.

\begin{figure}[h]
    \centering
    \input{pictures/example_simple_cone}
    \caption{A simple cone.}
    \label{fig: example of simple cone}
\end{figure}
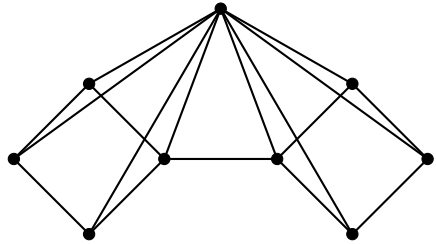

\begin{remark}\label{rem: no further decomposition}
Fans and simple cones could be further decomposed via edge-bonding by disconnecting them along a cut edge.
For example, a fan can be decomposed into triangles.
However, we will not take this point of view. Instead, it will be convenient to decompose a graph into fans and simple cones and regard them as elementary pieces.
\end{remark}

It follows from Corollary~\ref{cor: cone graph gives an isomorphism between BBG and RAAG} that the BBG defined on a fan or simple cone is a RAAG.
Here are some further properties of fans and simple cones that follow directly from the definitions.

\begin{lemma}\label{lem:easy fans simple cones}
Let $\Gamma$ be a fan or a simple cone. The following statements hold.
\begin{enumerate}
    \item The flag complex $\flag \Gamma$ is $2$-dimensional, simply connected, and contractible.
    
    \item The flag complex $\flag \Gamma$ has no interior edges, no interior triangles, and no crowned triangles.

    \item  If $\Gamma = \{v\}\ast P$ is a fan over a path $P$  with endpoints $u$ and $w$, then $\partial \flag \Gamma = P \cup \{(v,u),(v,w)\}$, and there are no interior vertices.
    
    \item  If $\Gamma = \{v\}\ast \Gamma_0$ is a simple cone over $\Gamma_0$, then $\partial \flag \Gamma = \Gamma_0$, and the cone vertex $v$ is the only interior vertex.

\end{enumerate}

\end{lemma}

\begin{lemma}\label{lem: no crowned triangles implies edge-bonding of wheels and fans}
Let $\Gamma$ be a biconnected   graph such that $\flag \Gamma$ is $2$-dimensional and simply connected. Suppose that $\flag  \Gamma$ has no crowned triangles. Then $  \Gamma$ decomposes as edge-bondings of fans and simple cones. 

\end{lemma}

\begin{proof}
We argue by induction on the number of cut edges of $\Gamma$. Suppose that $\Gamma$ has no cut edges. By Lemma~\ref{lem: no crowned triangles implies no interior 1-2 simplex and at most 1 interior vertex}, the complex $\flag\Gamma$ contains at most one interior vertex. We claim that if $\flag \Gamma$ contains no interior vertices, then $\Gamma$ is a fan. Let $v\in\vv \Gamma$. Since $v$ is a boundary vertex, its link has degree one vertices by Lemma~\ref{lem: v is interior iff its link has no deg 1 vertices}. 
Moreover, since $\Gamma$ has no cut edges, the link of $v$ has no cut vertices. 
Then $\lk{v,\Gamma}$ must be a single edge, and therefore, the graph $\Gamma$ is a triangle, which is a fan. 
Thus, the claim is proved. 
If $\flag \Gamma$ contains one interior vertex $u$, then $\Gamma = \st{u,\Gamma}$ as in the proof of Lemma~\ref{lem: no crowned triangles implies no interior 1-2 simplex and at most 1 interior vertex}. So, the graph $\Gamma$ is the cone over $\lk{u,\Gamma}$. Since $u$ is an interior vertex, its link has no degree one vertices. Note that the flag complex on $\lk{u,\Gamma}$ is $1$-dimensional; otherwise, the dimension of $\flag \Gamma$ would be greater than two. Thus, the graph $\Gamma=\st{u,\Gamma}$ is a simple cone. This proves the base case of induction.

Suppose that the conclusion holds for graphs having $n$ cut edges, $n\geq1$. Assume that $\Gamma$ has $n+1$ cut edges. Let $e$ be a cut edge of $\Gamma$. Cutting along $e$ gives some connected components $\Gamma_1,\dots,\Gamma_k$. Each of these components, as a full subgraph of $\Gamma$, satisfies all the assumptions of the lemma and has at most $n$ cut edges. By  induction,  the subgraphs $\Gamma_1,\dots,\Gamma_k$ are edge-bondings of fans and simple cones. Therefore, the graph $\Gamma$, as an edge-bonding of $\Gamma_1,\dots,\Gamma_k$, is also an edge-bonding of fans and simple cones.
\end{proof}

\begin{remark}
The decomposition in Lemma~\ref{lem: no crowned triangles implies edge-bonding of wheels and fans} is not unique (for instance, it is not maximal; see Remark~\ref{rem: no further decomposition}). We do not need this fact in this paper. 
\end{remark}

We now proceed to study the ways in which one can perform edge-bondings of fans and simple cones.
Recall from \S\ref{sec:redundant triples BBGs} that the spoke of a vertex $v$ is the collection of edges containing $v$.
When $\Gamma$ is a fan, write $\Gamma=\lbrace v\rbrace\ast P_n$, where $P_n$ is the path on $n$ labelled vertices; see Figure~\ref{fig: peripheral edges}. We call the edges $(v,w_1)$ and $(v,w_n)$ \emph{peripheral edges}, and the edges $(w_1,w_2)$ and $(w_{n-1},w_n)$ are called \emph{modified-peripheral edges}. A \emph{peripheral triangle} is a triangle containing a peripheral edge and a modified-peripheral edge.

\begin{figure}[h]
    \centering
    \input{pictures/peripheral_edges}
    \caption{The red edges are peripheral edges, and the green edges are modified-peripheral edges. The left-most and right-most triangles are peripheral triangles.}
    \label{fig: peripheral edges}
\end{figure}
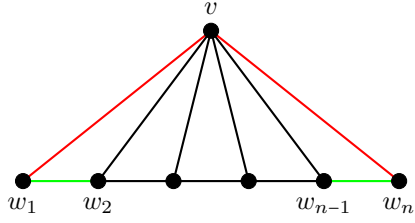

We say that an edge of a fan is \textit{good} if either it belongs to the spoke or it is a modified-peripheral edge. 
Similarly, we say that an edge of a simple cone is \textit{good} if it belongs to the spoke. 
We say an edge is \textit{bad} if it is not good.
Note that a bad edge is necessarily a boundary edge; see Lemma~\ref{lem:easy fans simple cones}.
We extend this definition to more general graphs as follows: let $\Gamma$ be a graph obtained via an edge-bonding on a collection of fans and simple cones, and let $e\in \ee \Gamma$ be a bonding edge of $\Gamma$.
We say that $e$ is \textit{good} if it is good in each fan component or simple cone component of $\Gamma$ that contains $e$. 
We say that $e$ is \textit{bad} otherwise.
These concepts are motivated by the fact that forming edge-bonding along good edges does not create crowned triangles; see the following example.

\begin{example}\label{ex: good edge-bonding}
Let $\Gamma_1$ and $\Gamma_2$ be a fan and a simple cone, respectively.
If we form the edge-bonding of $\Gamma_1$ and $\Gamma_2$ by identifying a good edge in each of them, the resulting graph has no crowned triangles.
The situation is analogous if $\Gamma_1$ and $\Gamma_2$ are both  fans or both simple cones.
\end{example}

\begin{lemma}\label{good edge-bonding must be along spokes or modified-peripheral edges}
Let $\Gamma=\Gamma_1\cup_{e}\Gamma_2$, where
$\Gamma_1$ is a fan or a simple cone, and  $\Gamma_2$ is any graph obtained via edge-bonding of fans and simple cones. 
If $e$ is a bad edge of $\Gamma_1$, then $\Gamma$ contains a crowned triangle.
\end{lemma}

\begin{proof}
If $e\in \ee {\Gamma_1}$ is bad, then it is in $\partial \flag {\Gamma_1}$. In particular, there is a unique triangle $\tau$ of $\Gamma_1$ containing $e$ (namely, the cone over $e$), and the other two edges of $\tau$ are not boundary edges (in the case of a fan, recall that a modified-peripheral edge is good). 
When we form an edge-bonding along $e$, the edge $e$ is no longer a boundary edge in $\Gamma$, so $\tau$ becomes a  crowned triangle in $\Gamma$. 

\end{proof}

\begin{proposition}\label{prop: tree 2-spanner iff no crowned triangles}
Let $\Gamma$ be a biconnected   graph such that $\flag \Gamma$ is $2$-dimensional and simply connected. Then $\Gamma$ admits a tree $2$-spanner if and only if $\flag \Gamma$ does not contain crowned triangles. 
\end{proposition}

\begin{proof}
Let $T$ be a tree $2$-spanner of $\Gamma$. Suppose by contradiction that $\Gamma$ contains a crowned triangle $\tau$ whose edges are $e$, $f$, and $g$. 
By Lemma~\ref{tree 2spanner triangle dicothomy}, either two of $e$, $f$, and $g$ are in $\ee T$ or none of them is in $\ee T$.
If $e$, $f$, and $g$ are not in $\ee T$, then by Lemma~\ref{tree 2spanner tetrahedron}, the graph $\Gamma$ contains a $K_4$.
This contradicts the fact that $\flag \Gamma$ is 2-dimensional.
Now consider the case that $e\notin \ee T$ and $f$ and $g$ are in $\ee T$.
Since $\tau$ is a crowned triangle, the edge $e$ is not on the boundary of $\flag \Gamma$, and there is another triangle $\tau'$ based on $e$ that is different from $\tau$. 
Denote the other edges of $\tau'$ by $f'$ and $g'$. 
Note that $f'$ and $g'$ cannot be in $\ee T$ by the uniqueness part of Lemma~\ref{tree 2spanner dicothomy}.
This means that none of the edges of $\tau'$ is in $\ee T$. 
Again, by Lemma~\ref{tree 2spanner tetrahedron} we obtain a $K_4$, hence a contradiction.
Therefore, the graph $\Gamma$ has no crowned triangles. 

Conversely, suppose that $\flag  \Gamma$ has no crowned triangles. 
By Lemma~\ref{lem: no crowned triangles implies edge-bonding of wheels and fans} the graph $\Gamma$ decomposes as edge-bondings of some fans and simple cones $\Gamma_1,\dots,\Gamma_{m}$.
Let $\lbrace e_1,\dots,e_n\rbrace$ be the set of bonding edges.
Note that by Remark~\ref{rem:simultaneous edge bonding} these edge-bonding operations can be performed simultaneously. 
Since $\Gamma$ has no crowned triangles, by Lemma~\ref{good edge-bonding must be along spokes or modified-peripheral edges}, each of the edges in $\lbrace e_1,\dots,e_n\rbrace$ is good.
We now construct a tree $2$-spanner for $\Gamma$. We do this by constructing a tree $2$-spanner $T_i$ for each $\Gamma_i$ and then gluing them together.
For a simple cone component $\Gamma_i$, choose $T_i$ to be the spoke.
For a fan component $\Gamma_i$, write $\Gamma_i=\lbrace v_i\rbrace\ast P_{n_i}$ and order the vertices of $P_{n_i}$ as $w_1,\dots,w_{n_i}$. 
Define $T_i$ to consist of the edges $(v_i,w_{n_2}),\dots,(v_i,w_{n_{i-1}})$, together with two more edges, one from each peripheral triangle, chosen as follows.
If the peripheral edge or the modified-peripheral edge in a peripheral triangle is involved in some edge-bondings, then choose that edge to be in $T_i$.
If none of them is involved in any edge-bonding, then choose either one of them.
Note that it is not possible that both the peripheral edge and the modified-peripheral edge of the same peripheral triangle are involved in edge-bondings; otherwise, the graph $\Gamma$ would contain a crowned triangle. 
In all the cases, this provides a tree $2$-spanner $T_i$ in $\Gamma_i$.
Moreover, if $e$ is a bonding edge for $\Gamma$ that appears in a component  $\Gamma_i$, then $e$ is in $T_i$.
It follows from \cite[Theorem 4.4]{CaiCorneilTreeSpanners} that $T=\bigcup^{m}_{i=1}T_i$ is a tree $2$-spanner of $\Gamma$.
\end{proof}

\begin{remark}
When $\Gamma$ is a $2$-tree (recall from \S\ref{section: example 2-trees}), the flag complex $\flag \Gamma$ is a biconnected contractible 2-dimensional flag complex.
In \cite{caionspanning2trees}, Cai showed that a $2$-tree admits a tree $2$-spanner if and only if it does not contain a trefoil subgraph (see Figure~\ref{fig:trefoil}). 
Proposition~\ref{prop: tree 2-spanner iff no crowned triangles} generalizes Cai's result to any biconnected and simply connected 2-dimensional flag complex.
Note that a trefoil subgraph in a $2$-tree is necessarily full, but this is not the case in general; see Figure~\ref{fig: diamond and house}.
\end{remark}

\subsection{The RAAG recognition problem in dimension 2}
In this section, we provide a complete answer to the RAAG recognition problem on $2$-dimensional complexes.
In other words, we completely characterize the graphs $\Gamma$ such that $\bbg \Gamma$ is a RAAG, under the assumption $\dim \flag \Gamma=2$.

Observe that a RAAG is always finitely presented (recall that all graphs are finite in our setting).
On the other hand, by \cite[Main Theorem (3)]{bestvinabradymorsetheoryandfinitenesspropertiesofgroups}, a BBG is  finitely presented precisely when the defining flag complex is simply connected.
Therefore, we can assume that $\flag \Gamma$ is simply connected.
Moreover, by Corollary~\ref{cor:biconnected components}  we can assume that  $\Gamma$ is also biconnected. 
Note that RAAGs are actually groups of type $F$, so one could even restrict to the case that $\flag \Gamma$ is contractible, thanks to \cite[Main Theorem]{bestvinabradymorsetheoryandfinitenesspropertiesofgroups}; compare this with Corollary~\ref{cor:tree2spanner implies contractible}. However, we do not need this fact. 
We start by showing that in dimension two, any crowned triangle is redundant.
 
\begin{lemma}\label{lem:crowned tri is redundant in dim 2}
If $\dim (\flag \Gamma)=2$, then every crowned triangle is a redundant triangle.
\end{lemma}
\begin{proof}
Let $\tau$ be a crowned triangle with edges $e_1,e_2,e_3$ and vertices $v_1,v_2,v_3$, where $v_j$ is opposite to $e_j$.
Since $\tau$ is a crowned triangle, no edge $e_j$ is a boundary edge. Hence, there is another triangle $\tau_j$ adjacent to $\tau$ along $e_j$. Let $u_j$ be the vertex of $\tau_j$ not in $\tau$. If $u_j$ were adjacent to $v_j$, then we would have a $K_4$, which is impossible since $\dim \flag \Gamma = 2$.
Thus, the vertices $v_j$ and $u_j$ are not adjacent.
As a consequence, we can choose a full subgraph $\Lambda_j\subseteq \lk{v_j,\Gamma}$ that contains $e_j$ and is a minimal full separating subgraph of $\Gamma$.
Finally, note that the intersection $\Lambda_1\cap \Lambda_2\cap \Lambda_3$ cannot contain any vertex. Otherwise, we would see a $K_4$, which is against the assumption that $\flag \Gamma$ is 2-dimensional.
\end{proof}

\begin{maintheoremc}{A}\label{body main thm 2dim}
Let $\Gamma$ be a biconnected graph such that $\flag \Gamma$ is $2$-dimensional and simply connected. Then the following statements are equivalent. 
\begin{enumerate}

    \item \label{item: tree 2-spanner} $\Gamma$ admits a tree $2$-spanner.
   
    \item \label{item: crowned triangles} $\flag \Gamma$ does not contain crowned triangles.
    
    \item \label{item: BBG not RAAG} $\bbg \Gamma$ is a RAAG.
    
    \item \label{item: BBG an Artin} $\bbg \Gamma$ is an Artin group.
\end{enumerate}
\end{maintheoremc}

\begin{proof}
The implications \eqref{item: tree 2-spanner} $\Leftrightarrow$ \eqref{item: crowned triangles} follows from Proposition~\ref{prop: tree 2-spanner iff no crowned triangles}.
Moreover, the implication \eqref{item: tree 2-spanner} $\Rightarrow$ \eqref{item: BBG not RAAG} is Theorem~\hyperref[containing a tree 2-spanner implies that BBG is a RAAG]{B}.
The implication \eqref{item: BBG not RAAG} $\Rightarrow$ \eqref{item: BBG an Artin} is obvious.

We prove the implication \eqref{item: BBG not RAAG} $\Rightarrow$  \eqref{item: crowned triangles} as follows.
Assume that $\flag  \Gamma$ contains a crowned triangle $\tau$.
Then by Lemma~\ref{lem:crowned tri is redundant in dim 2} we know that $\tau$ is also a redundant triangle.
Then it follows from Theorem~\hyperref[thm:redundant triple criterion]{E} that $\bbg\Gamma$ is not a RAAG.
The implication
\eqref{item: BBG an Artin} $\Rightarrow$ \eqref{item: crowned triangles} is obtained in a similar way, using  Corollary~\ref{cor:not Artin} instead of Theorem~\hyperref[thm:redundant triple criterion]{E}.
\end{proof}

Papadima and Suciu in \cite[Proposition 9.4]{PapadimaSuciuAlgebraicinvariantsforBBGs} showed that if $\flag \Gamma$ is a certain type of triangulation of the $2$-disk (which they call \textit{extra-special triangulation}), then $\bbg \Gamma$ is not a RAAG.
Those triangulations always contain  a crowned triangle, so Theorem~\hyperref[body main thm 2dim]{A} recovers Papadima--Suciu's result and extends it to a wider class of graphs, such as arbitrary triangulations of disks (see Example~\ref{ex:extended trefoil continued}), or even flag complexes that are not triangulations of disks (see Example~\ref{ex: a bouquet of triangles}.)

\begin{example}[The extended trefoil continued]\label{ex:extended trefoil continued}
Let $\Gamma$ be the graph in Figure~\ref{fig: A special but not extra-special triangulation}.
Since $\Gamma$ contains a crowned triangle, the group $\bbg \Gamma$ is not a RAAG by Theorem~\hyperref[body main thm 2dim]{A}. 
Note that this fact does not follow from \cite{PapadimaSuciuAlgebraicinvariantsforBBGs}: the flag complex $\flag \Gamma$ is a triangulation of the disk but not an extra-special triangulation.
This fact also does not follow from \cite{DayWadeSubspaceArrangementBNSinvariantsandpuresymmetricOuterAutomorphismsofRAAGs}, because all the subspace arrangement homology groups vanish for this group $\bbg\Gamma$, that is, they look like those of a RAAG (as observed in Example~\ref{ex:extended trefoil}).
\end{example}


\begin{remark}\label{rem: higher dimensional}
The criterion for a BBG to be a RAAG from Theorem~\hyperref[containing a tree 2-spanner implies that BBG is a RAAG]{B} works in any dimension.
On the other hand, Theorem~\hyperref[body main thm 2dim]{A} fails for higher dimensional complexes. 
Indeed, the mere existence of a crowned triangle is not very informative in higher dimension cases; see Example~\ref{ex:cone over PS}.
However, the existence of a redundant triangle is an obstruction for a BBG to be a RAAG even in higher dimensional complexes; see Example~\ref{ex:higher dimensional}.
\end{remark}

\begin{example}[A crowned triangle in dimension three does not imply that the BBG is not a RAAG]\label{ex:cone over PS}
Let $\Gamma$ be the cone over the trefoil graph in Figure~\ref{fig:trefoil}. 
Then $\flag \Gamma$ is $3$-dimensional and $\Gamma$ contains a crowned triangle (the one sitting in the trefoil graph). 
However, this crowned triangle is not a redundant triangle, and the group $\bbg \Gamma$ is actually a RAAG by Corollary~\ref{cor: cone graph gives an isomorphism between BBG and RAAG}. 
\end{example}

\begin{example}[A redundant triangle in dimension three implies that the BBG is not a RAAG]\label{ex:higher dimensional}
Consider the graph $\Gamma$ in Figure~\ref{fig: no crowned triangles and no tree 2-spanner}. 
Then $\flag \Gamma$ is $3$-dimensional and every $3$-simplex has a $2$-face in $\partial \flag{\Gamma}$. 
However, we can show that this $\bbg \Gamma$ is not a RAAG.
The triangle induced by the vertices $v_1$, $v_2$, and $v_3$ is a redundant triangle. 
Indeed, the full subgraphs $\Lambda_1$, $\Lambda_2$, and $\Lambda_3$  induced by the sets of vertices $\lbrace u,v_2,v_3\rbrace$, $\lbrace u,v_1,v_3\rbrace$, and $\lbrace v_1,v_2,w\rbrace$, respectively, satisfy condition \eqref{item: omega} in the definition of redundant triangle.
Then it follows from Theorem~\hyperref[thm:redundant triple criterion]{E} that this $\bbg \Gamma$ is not a RAAG. 
\end{example}

\clearpage

\begin{figure}[ht!]
    \centering    \input{pictures/ex_tetrahedron_with_3_simplices_and_one_extra}
    \caption{A $3$-dimensional complex that contains a redundant triangle.}
    \label{fig: no crowned triangles and no tree 2-spanner}
\end{figure}
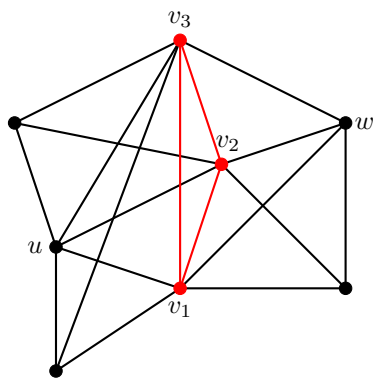


\printbibliography

\end{document}

%% file: pictures/trefoil.tex
\begin{tikzpicture}[scale=0.5]
\draw [thick] (-1,0)--(0,2)--(1,0);
\draw [thick] (-1,0)--(0,-2);
\draw [thick] (-1,0)--(1,0);
\draw [thick] (0,-2)--(1,0);
\draw [thick] (1,0)--(2,-2)--(0,-2);
\draw [thick] (-1,0)--(-2,-2)--(0,-2);

\draw [fill] (0,2) circle [radius=0.1];
\draw [fill] (-1,0) circle [radius=0.1];
\draw [fill] (1,0) circle [radius=0.1];
\draw [fill] (-2,-2) circle [radius=0.1];
\draw [fill] (0,-2) circle [radius=0.1];
\draw [fill] (2,-2) circle [radius=0.1];
\end{tikzpicture}

%% file: pictures/extended_PS_unoriented.tex
\begin{tikzpicture}[scale=0.5]
\draw [thick] (-1,0)--(0,2);
\draw [thick] (-1,0)--(1,0);
\draw [thick] (-2,-2)--(-1,0);
\draw [thick] (2,-2)--(1,0);
\draw [thick] (0,2)--(2,2);

\draw [thick] (0,-2)--(-1,0);

\draw [thick] (0,-2)--(1,0);

\draw [thick] (0,-2)--(2,-2);

\draw [thick] (0,-2)--(-2,-2);

\draw [thick] (1,0)--(0,2);

\draw [thick] (1,0)--(2,2);

\draw [fill] (0,2) circle [radius=0.1];
\draw [fill] (-1,0) circle [radius=0.1];
\draw [fill] (1,0) circle [radius=0.1];
\draw [fill] (-2,-2) circle [radius=0.1];
\draw [fill] (0,-2) circle [radius=0.1];
\draw [fill] (2,-2) circle [radius=0.1];
\draw [fill] (2,2) circle [radius=0.1];
\end{tikzpicture}

%% file: pictures/tree2spanner_simply_connected.tex
\begin{tikzpicture}[scale=0.9]

\draw [thick] (2,0)--(1,2);
\draw [thick, red] (1,2)--(2,4);
\draw [thick, red] (2,4)--(4,4);
\draw [thick] (4,4)--(5,2);
\draw [thick] (5,2)--(4,0);
\draw [thick, red] (4,0)--(2,0);

\draw [thick, red] (0,0.25)--(2,0);
\draw [thick, red] (0,0.25)--(1,2);

\draw [thick, red] (6,3.75)--(4,4);
\draw [thick, red] (6,3.75)--(5,2);

\draw [thick, red] (6,0.25)--(4,0);
\draw [thick, red] (6,0.25)--(5,2);

\draw [fill] (2,0) circle [radius=0.1];
\draw [fill] (1,2) circle [radius=0.1];
\draw [fill] (2,4) circle [radius=0.1];
\draw [fill] (4,4) circle [radius=0.1];
\draw [fill] (5,2) circle [radius=0.1];
\draw [fill] (4,0) circle [radius=0.1];

\draw [fill] (0,0.25) circle [radius=0.1];
\draw [fill] (6,0.25) circle [radius=0.1];
\draw [fill] (6,3.75) circle [radius=0.1];

\node [left] at (1,2) {$v_1$};
\node [above] at (2,4) {$v_2$};
\node [above] at (4,4) {$v_3$};
\node [right] at (5,2) {$v_4$};
\node [below] at (4,0) {$v_5$};
\node [below] at (2,0) {$v_6$};

\node [above right] at (6,3.75) {$w_3$};
\node [below right] at (6,0.25) {$w_4$};
\node [below left] at (0,0.25) {$w_6$};

\node [left, red] at (1.5,3) {$e_1$};
\node [above, red] at (3,4) {$e_2$};
\node [right] at (4.5,3) {$e_3$};
\node [right] at (4.5,1) {$e_4$};
\node [below, red] at (3,0) {$e_5$};
\node [left] at (1.5,1) {$e_6$};

\node [above, red] at (5,3.875) {$e^{-}_3$};
\node [below right, red] at (5.5,3.25) {$e^{+}_3$};

\node [below, red] at (5,0.125) {$e^{-}_4$};
\node [above right, red] at (5.5,0.9) {$e^{+}_4$};

\node [below, red] at (1,0.125) {$e^{-}_6$};
\node [above left, red] at (0.5,0.9) {$e^{+}_6$};

\node at (3,2) {$C$};
\node [red] at (0.5,3.5) {$L$};

\begin{scope}[shift={(7.5,0)}]

\draw [thick] (2,0)--(1,2);
\draw [thick] (1,2)--(2,4);
\draw [thick] (2,4)--(4,4);
\draw [thick] (4,4)--(5,2);
\draw [thick] (5,2)--(4,0);
\draw [thick] (4,0)--(2,0);

\draw [thick, red] (3,2)--(2,0);
\draw [thick, red] (3,2)--(1,2);
\draw [thick, red] (3,2)--(2,4);
\draw [thick, red] (3,2)--(4,4);
\draw [thick, red] (3,2)--(5,2);
\draw [thick, red] (3,2)--(4,0);

\draw [fill] (3,2) circle [radius=0.1];
\draw [fill] (2,0) circle [radius=0.1];
\draw [fill] (1,2) circle [radius=0.1];
\draw [fill] (2,4) circle [radius=0.1];
\draw [fill] (4,4) circle [radius=0.1];
\draw [fill] (5,2) circle [radius=0.1];
\draw [fill] (4,0) circle [radius=0.1];

\node [below] at (3,1.8) {$w$};
\node [left] at (1,2) {$v_1$};
\node [above] at (2,4) {$v_2$};
\node [above] at (4,4) {$v_3$};
\node [right] at (5,2) {$v_4$};
\node [below] at (4,0) {$v_5$};
\node [below] at (2,0) {$v_6$};

\node [left] at (1.5,3) {$e_1$};
\node [above] at (3,4) {$e_2$};
\node [right] at (4.5,3) {$e_3$};
\node [right] at (4.5,1) {$e_4$};
\node [below] at (3,0) {$e_5$};
\node [left] at (1.5,1) {$e_6$};

\node at (3,0.5) {$C$};
\node [red] at (3,3.5) {$L$};

\end{scope}
\end{tikzpicture}

%% file: pictures/proof_BBG_tree2-spanner_RAAG.tex
\begin{tikzpicture}[scale=1.25]
\draw [thick] (0,1)--(2,2)--(2,0)--(0,1);
\draw [thick, red] (1.3,1)--(0,1);
\draw [thick, red] (1.3,1)--(2,2);
\draw [thick, red] (1.3,1)--(2,0);

\draw [thick, red] (2,2)--(4,1)--(2,0);

\draw [fill] (0,1) circle [radius=0.05];
\draw [fill] (2,2) circle [radius=0.05];
\draw [fill] (2,0) circle [radius=0.05];
\draw [fill] (1.3,1) circle [radius=0.05];
\draw [fill] (4,1) circle [radius=0.05];

\node [right] at (2,1) {$e$};
\node [above] at (3,1.5) {$g$};
\node [below] at (3,0.5) {$f$};

\node [above] at (1,1.5) {$g'$};
\node [below] at (1,0.5) {$f'$};

\node [above] at (1,1) {$e''$};
\node [above right] at (1.5,0.5) {\small $f''$};
\node [below right] at (1.5,1.5) {\small $g''$};

\begin{scope}[shift={(5.5,0)}]
\draw [thick] (0,1)--(2,2)--(2,0)--(0,1);
\draw [thick, red] (1.3,1)--(0,1);
\draw [thick, red] (1.3,1)--(2,2);
\draw [thick, red] (1.3,1)--(2,0);

\draw [fill] (0,1) circle [radius=0.05];
\draw [fill] (2,2) circle [radius=0.05];
\draw [fill] (2,0) circle [radius=0.05];
\draw [fill] (1.3,1) circle [radius=0.05];

\node [right] at (2,1) {$e$};

\node [above] at (1,1.5) {$g'$};
\node [below] at (1,0.5) {$f'$};

\node [above] at (1,1) {$e''$};
\node [above right] at (1.5,0.5) {\small $f$};
\node [below right] at (1.5,1.5) {\small $g$};
\end{scope}
\end{tikzpicture}

%% file: pictures/Example_Whitney_Twist.tex
\begin{tikzpicture}[scale=0.7]

\draw [thick, red] (2,0)--(0,0);
\draw [thick, red] (2,0)--(2,2);
\draw [thick] (2,2)--(0,2);
\draw [thick] (0,0)--(0,2);
\draw [thick, red] (2,0)--(0,2);

\draw [thick, red] (2,0)--(4,0);
\draw [thick] (4,0)--(4,2);
\draw [thick] (4,2)--(2,2);
\draw [thick, red] (2,0)--(4,2);

\draw [fill] (0,0) circle [radius=0.1];
\draw [fill] (2,0) circle [radius=0.1];
\draw [fill] (2,2) circle [radius=0.1];
\draw [fill] (0,2) circle [radius=0.1];
\draw [fill] (4,0) circle [radius=0.1];
\draw [fill] (4,2) circle [radius=0.1];

\begin{scope}[shift={(7,0)}]

\draw [thick] (2,0)--(0,0);
\draw [thick,red] (2,0)--(2,2);
\draw [thick] (2,2)--(0,2);
\draw [thick,red] (0,0)--(0,2);
\draw [thick, red] (0,2)--(2,0);

\draw [thick] (2,0)--(4,0);
\draw [thick, red] (4,0)--(4,2);
\draw [thick] (4,2)--(2,2);
\draw [thick,red] (2,2)--(2,0);
\draw [thick, red] (2,2)--(4,0);

\draw [fill] (0,0) circle [radius=0.1];
\draw [fill] (2,0) circle [radius=0.1];
\draw [fill] (2,2) circle [radius=0.1];
\draw [fill] (0,2) circle [radius=0.1];
\draw [fill] (4,0) circle [radius=0.1];
\draw [fill] (4,2) circle [radius=0.1];
\end{scope}
\end{tikzpicture}

%% file: pictures/ex_bbg_on_a_bouquet_of_triangles.tex
\begin{tikzpicture}[scale=0.7]

\draw [thick, red] (2,0)--(0,1);
\draw [thick, red] (2,0)--(1,3);
\draw [thick, red] (2,0)--(3,3);

\draw [thick] (0,1)--(2,2);
\draw [thick, red] (2,0)--(2,2);

\draw [thick] (2,2)--(4,1);
\draw [thick, red] (4,1)--(2,0);

\draw [thick] (2,2)--(1,-1);
\draw [thick, red] (1,-1)--(2,0);

\draw [thick] (2,0)--(4,-1);
\draw [thick, red] (4,-1)--(4,1);

\draw [thick] (4,1)--(6,0);
\draw [thick, red] (6,0)--(4,-1);

\draw [thick]  (3,3)--(2,2);

\draw [thick] (1,3)--(2,2);

\draw [fill] (2,0) circle [radius=0.1];
\draw [fill] (6,0) circle [radius=0.1];
\draw [fill] (2,2) circle [radius=0.1];
\draw [fill] (0,1) circle [radius=0.1];
\draw [fill] (4,1) circle [radius=0.1];
\draw [fill] (4,-1) circle [radius=0.1];
\draw [fill] (1,-1) circle [radius=0.1];
\draw [fill] (1,3) circle [radius=0.1];
\draw [fill] (3,3) circle [radius=0.1];

\end{tikzpicture}

%% file: pictures/bad_dead_subgraph_tikz_version.tex
\begin{tikzpicture}[scale=0.5]

\draw [thick] (1,0)--(0,2)--(1,4)--(3,4)--(4,2)--(3,0)--(1,0);

\draw [thick] (4,2)--(6,2);

\draw [fill] (1,0) circle [radius=0.15];
\node [below left] at (1,0) {\small $1$};
\draw [fill, red] (0,2) circle [radius=0.15];
\node [above left] at (0,2) {\small $0$};
\draw [fill] (1,4) circle [radius=0.15];
\node [above left] at (1,4) {\small $1$};
\draw [fill] (3,4) circle [radius=0.15];
\node [above right] at (3,4) {\small $1$};
\draw [fill, red] (4,2) circle [radius=0.15];
\node [above right] at (3.8,2) {\small $0$};
\draw [fill] (3,0) circle [radius=0.15];
\node [below right] at (3,0) {\small $1$};
\draw [fill] (6,2) circle [radius=0.15];
\node [above right] at (6,2) {\small $0$};

\end{tikzpicture}

%% file: pictures/living_edge_subgraph_dead_edge_subgraphs_not_full.tex
\begin{tikzpicture}[scale=0.6]

\draw [thick, middlearrow={stealth}] (0,0)--(4,0);
\draw [thick, middlearrow={stealth}] (4,0)--(4,4);
\draw [thick, middlearrow={stealth}] (4,4)--(0,4);
\draw [thick, middlearrow={stealth}] (0,4)--(0,0);

\draw [thick, middlearrow={stealth}] (2,2)--(0,0);
\draw [thick, middlearrow={stealth}] (2,2)--(4,0);
\draw [thick, middlearrow={stealth}] (2,2)--(4,4);
\draw [thick, middlearrow={stealth}] (2,2)--(0,4);

\draw [fill] (0,0) circle [radius=0.15];
\draw [fill] (4,0) circle [radius=0.15];
\draw [fill] (4,4) circle [radius=0.15];
\draw [fill] (0,4) circle [radius=0.15];
\draw [fill] (2,2) circle [radius=0.15];

\node [above left] at (1,1) {$1$};
\node [above right] at (3,1) {$1$};
\node [below left] at (1,3) {$2$};
\node [below right] at (3,3) {$2$};

\node [below] at (2,0) {$0$}; 
\node [right] at (4,2) {$1$}; 
\node [above] at (2,4) {$0$}; 
\node [left] at (0,2) {$-1$}; 
\end{tikzpicture}

%% file: pictures/living_edge_subgraph_dead_edge_subgraphs_example.tex
\begin{tikzpicture}[scale=0.6]

\draw [thick, middlearrow={stealth}] (4,3)--(0,0);
\draw [thick, middlearrow={stealth}] (4,3)--(2,0);
\draw [thick, middlearrow={stealth}] (4,3)--(4,0);
\draw [thick, middlearrow={stealth}] (4,3)--(6,0);
\draw [thick, middlearrow={stealth}] (4,3)--(8,0);
\draw [thick, middlearrow={stealth}] (0,0)--(2,0);
\draw [thick, middlearrow={stealth}] (2,0)--(4,0);
\draw [thick, middlearrow={stealth}] (4,0)--(6,0);
\draw [thick, middlearrow={stealth}] (6,0)--(8,0);

\draw [fill] (4,3) circle [radius=0.12];
\draw [fill, red] (0,0) circle [radius=0.12];
\draw [fill] (2,0) circle [radius=0.12];
\draw [fill] (4,0) circle [radius=0.12];
\draw [fill] (6,0) circle [radius=0.12];
\draw [fill, red] (8,0) circle [radius=0.12];

\node [above] at (4,3) {$0$};
\node [below] at (0,0) {$1$};
\node [below] at (2,0) {$0$};
\node [below] at (4,0) {$0$};
\node [below] at (6,0) {$0$};
\node [below] at (8,0) {$1$};

\node [above] at (1.8,1.5) {$1$};
\node [above] at (6.2,1.5) {$1$};
\node [below] at (1,0) {$-1$};
\node [below] at (7,0) {$-1$};

\end{tikzpicture} 

%% file: pictures/ex_bns_bbg_square.tex
\begin{tikzpicture}[scale=1]

\draw [thick, middlearrow={stealth}] (0,0)--(2,0);
\draw [thick, middlearrow={stealth}] (2,0)--(2,2);
\draw [thick, middlearrow={stealth}] (2,2)--(0,2);
\draw [thick, middlearrow={stealth}] (0,2)--(0,0);

\draw [fill] (0,0) circle [radius=0.1];
\node [below left] at (0,0) {$1$};

\draw [fill] (2,0) circle [radius=0.1];
\node [below right] at (2,0) {$0$};

\draw [fill] (2,2) circle [radius=0.1];
\node [above right] at (2,2) {$1$};

\draw [fill] (0,2) circle [radius=0.1];
\node [above left] at (0,2) {$0$};


\node [below] at (1,0) {$-1$};

\node [right] at (2,1) {$1$};

\node [above] at (1,2) {$-1$};

\node [left] at (0,1) {$1$};

\begin{scope}[shift={(5,0)}]

\draw [thick, middlearrow={stealth}] (2,1)--(0,0);
\draw [thick, middlearrow={stealth}] (2,1)--(0,2);
\draw [thick, middlearrow={stealth}] (0,2)--(0,0);

\draw [thick, middlearrow={stealth}] (2,1)--(4,0);
\draw [thick, middlearrow={stealth}] (2,1)--(4,2);
\draw [thick, middlearrow={stealth}] (4,2)--(4,0);

\draw [fill] (0,0) circle [radius=0.1];
\node [below left] at (0,0) {$1$};

\draw [fill] (0,2) circle [radius=0.1];
\node [above left] at (0,2) {$1$};

\draw [fill] (2,1) circle [radius=0.1];
\node [below] at (2,0.9) {$0$};

\draw [fill] (4,0) circle [radius=0.1];
\node [below right] at (4,0) {$1$};

\draw [fill] (4,2) circle [radius=0.1];
\node [above right] at (4,2) {$1$};


\node [left] at (0,1) {$0$};

\node [above] at (1,1.5) {$1$};

\node [below] at (1,0.5) {$1$};

\node [right] at (4,1) {$0$};

\node [above] at (3,1.5) {$1$};

\node [below] at (3,0.5) {$1$};  
\end{scope}

\end{tikzpicture}

%% file: pictures/Example_BBG_cut_edges_and_hyperplanes.tex
\begin{tikzpicture}[scale=0.7]
\draw [thick, middlearrow={stealth}] (-1,0)--(0,2);
\draw [thick, middlearrow={stealth}] (-1,0)--(1,0);
\node [above] at (0,0) {$f$};
\draw [thick, middlearrow={stealth}] (-2,-2)--(-1,0);
\draw [thick, middlearrow={stealth}] (2,-2)--(1,0);

\draw [thick, middlearrow={stealth}] (0,-2)--(-1,0);
\node [below] at (-0.8,-0.8) {$e_{1}$};
\draw [thick, middlearrow={stealth}] (0,-2)--(1,0);
\node [below] at (0.8,-0.8) {$e_{2}$};
\draw [thick, middlearrow={stealth}] (0,-2)--(2,-2);
\node [below] at (1,-2) {$e_{4}$};
\draw [thick, middlearrow={stealth}] (0,-2)--(-2,-2);
\node [below] at (-1,-2) {$e_{3}$};
\draw [thick, middlearrow={stealth}] (1,0)--(0,2);
\node at (0.85,1) {$e_{5}$};

\draw [fill] (0,2) circle [radius=0.1];
\draw [fill] (-1,0) circle [radius=0.1];
\draw [fill] (1,0) circle [radius=0.1];
\draw [fill] (-2,-2) circle [radius=0.1];
\draw [fill] (0,-2) circle [radius=0.1];
\draw [fill] (2,-2) circle [radius=0.1];
\end{tikzpicture}

%% file: pictures/ex_bbg_on_PS_with_an_extra_triangle.tex
\begin{tikzpicture}[scale=0.7]
\draw [thick, middlearrow={stealth}] (-1,0)--(0,2);
\draw [thick, middlearrow={stealth}] (-1,0)--(1,0);
\node [above] at (0,0) {$f$};
\draw [thick, middlearrow={stealth}] (-2,-2)--(-1,0);
\draw [thick, middlearrow={stealth}] (2,-2)--(1,0);
\draw [thick, middlearrow={stealth}] (0,2)--(2,2);

\draw [thick, middlearrow={stealth}] (0,-2)--(-1,0);
\node [below] at (-0.8,-0.8) {$e_{1}$};
\draw [thick, middlearrow={stealth}] (0,-2)--(1,0);
\node [below] at (0.8,-0.8) {$e_{2}$};
\draw [thick, middlearrow={stealth}] (0,-2)--(2,-2);
\node [below] at (1,-2) {$e_{4}$};
\draw [thick, middlearrow={stealth}] (0,-2)--(-2,-2);
\node [below] at (-1,-2) {$e_{3}$};
\draw [thick, middlearrow={stealth}] (1,0)--(0,2);
\node at (0.85,1) {$e_{5}$};
\draw [thick, middlearrow={stealth}] (1,0)--(2,2);
\node [right] at (1.5,0.9) {$e_{6}$};

\draw [fill] (0,2) circle [radius=0.1];
\draw [fill] (-1,0) circle [radius=0.1];
\draw [fill] (1,0) circle [radius=0.1];
\draw [fill] (-2,-2) circle [radius=0.1];
\draw [fill] (0,-2) circle [radius=0.1];
\draw [fill] (2,-2) circle [radius=0.1];
\draw [fill] (2,2) circle [radius=0.1];
\end{tikzpicture}

%% file: pictures/linear_algebra.tex
\begin{tikzpicture}
 
\node at (-7,2) (01) {$0$};
\node at (-4,2) (int) {$(W_i+W_j)^\perp=W_i^\perp \cap W_j^\perp$};
\node at (0,2) (Dij) {$W_i^\perp \oplus W_j^\perp$};
\node at (3,2) (sum) {$W_i^\perp + W_j^\perp$};
\node at (5,2) (02) {$0$};

\node at (0,0) (D) {$W_1^\perp \oplus W_2^\perp \oplus W_3^\perp$};

\draw [->] (01) edge (int);
\draw [->] (int) edge (Dij);
\draw [->] (Dij) edge (D);
\draw [->] (Dij) edge (sum);
\draw [->] (sum) edge (02);

\node at (-1.5,2.25) {$I_{ij}$};
\node at (1.5,2.25) {$F_{ij}$};
\node at (.5,1) {$J_{ij}$};
 
\end{tikzpicture}

 

%% file: pictures/diamond_house.tex
\begin{tikzpicture}[scale=0.5]

\draw [thick] (0,3)--(2,3)--(4,3)--(6,3);

\draw [thick] (3,0)--(0,3)--(3,6);
\draw [thick] (3,0)--(2,3)--(3,6);
\draw [thick] (3,0)--(4,3)--(3,6);
\draw [thick] (3,0)--(6,3)--(3,6);

\draw [fill] (0,3) circle [radius=0.15];
\draw [fill] (2,3) circle [radius=0.15];
\draw [fill] (4,3) circle [radius=0.15];
\draw [fill] (6,3) circle [radius=0.15];
\draw [fill] (3,0) circle [radius=0.15];
\draw [fill] (3,6) circle [radius=0.15];

\begin{scope}[shift={(11,-1)}, scale=1.2]

\draw [thick] (0,1.5)--(0,4.5)--(3,4.5)--(3,1.5)--(0,1.5);

\draw [thick] (0,1.5)--(3,4.5);
\draw [thick] (3,1.5)--(0,4.5);

\draw [thick] (3,1.5)--(5,3)--(3,4.5);

\draw [fill] (0,1.5) circle [radius=0.15];
\draw [fill] (0,4.5) circle [radius=0.15];
\draw [fill] (3,4.5) circle [radius=0.15];
\draw [fill] (3,1.5) circle [radius=0.15];
\draw [fill] (1.5,3) circle [radius=0.15];
\draw [fill] (5,3) circle [radius=0.15];

\end{scope}

\end{tikzpicture}

%% file: pictures/proof_only_one_interior_vertex.tex
\begin{tikzpicture}[scale=1,rotate=-30]

\draw [thick, red] (30:3) arc [start angle=30, end angle=150, radius=1.5cm];

\draw [thick, cyan] (30:1.5cm)--(30:3cm);
\draw [fill] (30:3cm) circle [radius=0.075];

\node [red] at (60:3cm) {$p$};

\node [below right] at (30:1.45cm) {$w$};
\node [right] at (30:3cm) {$u$};

\draw [thick, green] (30:1.5) arc [start angle=30, end angle=90, radius=1.5cm];
\draw [thick] (90:1.5) arc [start angle=90, end angle=390, radius=1.5cm];
	
\draw [thick] (0,0)--(30:1.5cm);
\draw [thick] (0,0)--(90:1.5cm);
\draw [thick] (0,0)--(150:1.5cm);
\draw [thick] (0,0)--(210:1.5cm);	
\draw [thick] (0,0)--(270:1.5cm);
\draw [thick] (0,0)--(330:1.5cm);

\draw [fill] (0,0) circle [radius=0.075];	
\draw [fill] (30:1.5cm) circle [radius=0.075];
\draw [fill] (90:1.5cm) circle [radius=0.075];
\draw [fill] (150:1.5cm) circle [radius=0.075];
\draw [fill] (210:1.5cm) circle [radius=0.075];
\draw [fill] (270:1.5cm) circle [radius=0.075];
\draw [fill] (330:1.5cm) circle [radius=0.075];

\node at (120:0.5cm) {$v$};
\node [above] at (90:1.55cm) {$u'$};

\end{tikzpicture}

%% file: pictures/example_simple_cone.tex
\begin{tikzpicture}[scale=0.5]
\draw [thick] (2,0)--(0,2)--(2,4)--(4,2)--(2,0);
\draw [thick] (4,2)--(7,2);
\draw [thick] (7,2)--(9,4)--(11,2)--(9,0)--(7,2);

\draw[thick] (5.5,6)--(2,4);
\draw[thick] (5.5,6)--(0,2);
\draw[thick] (5.5,6)--(2,0);
\draw[thick] (5.5,6)--(4,2);
\draw[thick] (5.5,6)--(7,2);
\draw[thick] (5.5,6)--(9,4);
\draw[thick] (5.5,6)--(11,2);
\draw[thick] (5.5,6)--(9,0);

\draw [fill] (5.5,6) circle [radius=0.15];
\draw [fill] (2,0) circle [radius=0.15];
\draw [fill] (0,2) circle [radius=0.15];
\draw [fill] (2,4) circle [radius=0.15];
\draw [fill] (4,2) circle [radius=0.15];
\draw [fill] (7,2) circle [radius=0.15];
\draw [fill] (9,4) circle [radius=0.15];
\draw [fill] (11,2) circle [radius=0.15];
\draw [fill] (9,0) circle [radius=0.15];

\end{tikzpicture}

%% file: pictures/peripheral_edges.tex
\begin{tikzpicture}[scale=1]
\draw [thick, green] (0,0)--(1,0);
\draw [thick] (1,0)--(4,0);
\draw [thick, green] (4,0)--(5,0);

\draw [thick, red] (2.5,2)--(0,0);
\draw [thick] (2.5,2)--(1,0);
\draw [thick] (2.5,2)--(2,0);
\draw [thick] (2.5,2)--(3,0);
\draw [thick] (2.5,2)--(4,0);
\draw [thick, red] (2.5,2)--(5,0);

\node [above] at (2.5,2.1) {$v$};
\node [below] at (0,-0.1) {$w_1$};
\node [below] at (1,-0.1) {$w_2$};
\node [below] at (4,-0.1) {$w_{n-1}$};
\node [below] at (5,-0.1) {$w_n$};

\draw [fill] (2.5,2) circle [radius=0.1];
\draw [fill] (0,0) circle [radius=0.1];
\draw [fill] (1,0) circle [radius=0.1];
\draw [fill] (2,0) circle [radius=0.1];
\draw [fill] (3,0) circle [radius=0.1];
\draw [fill] (4,0) circle [radius=0.1];
\draw [fill] (5,0) circle [radius=0.1];
\end{tikzpicture}

%% file: pictures/ex_tetrahedron_with_3_simplices_and_one_extra.tex
\begin{tikzpicture}[scale=0.55]

\draw [thick] (3,3)--(6,2); 
\draw [thick,red] (6,2)--(7,5);
\draw [thick] (7,5)--(3,3);

\draw [thick] (3,3)--(2,6);
\draw [thick] (2,6)--(7,5);

\draw [thick] (6,2)--(10,6);
\draw [thick] (10,6)--(7,5);

\draw [thick] (3,3)--(3,0)--(6,2);

\draw [thick] (6,2)--(10,2);
\draw [thick] (10,2)--(10,6);
\draw [thick] (10,2)--(7,5);

\draw [thick] (6,8)--(2,6);
\draw [thick] (6,8)--(3,3);
\draw [thick] (6,8)--(3,0);
\draw [thick, red] (6,8)--(6,2);
\draw [thick, red] (6,8)--(7,5);
\draw [thick] (6,8)--(10,6);

\draw [fill,red] (6,8) circle [radius=0.15];
\draw [fill] (2,6) circle [radius=0.15];
\draw [fill] (3,3) circle [radius=0.15];
\draw [fill] (3,0) circle [radius=0.15];
\draw [fill,red] (6,2) circle [radius=0.15];
\draw [fill,red] (7,5) circle [radius=0.15];
\draw [fill] (10,6) circle [radius=0.15];
\draw [fill] (10,2) circle [radius=0.15];

\node [left] at (2.9,3) {$u$};
\node [below] at (6,1.9) {$v_1$};
\node [above] at (7.16,5.1) {$v_2$};
\node [above] at (6,8.1) {$v_3$};
\node [right] at (10,6) {$w$};

\end{tikzpicture}